\documentclass[10pt, letterpaper]{amsart}
\usepackage[margin=1in]{geometry}
\usepackage{preamble_article_R_T_orthogonal}

\title{An R=T theorem for certain orthogonal Shimura varieties}
\author{Hao Peng and Dmitri Whitmore}

\begin{document}

\begin{abstract}
We prove an almost minimal $\msf R=\bb T$ theorem for self-dual Galois representations with coefficients in a finite field satisfying a property called rigid. We also prove the rigidity property for a large family of residual Galois representations attached to regular algebraic self-dual representations. Our theorem is based on a Taylor--Wiles patching argument for $G$-valued Galois representation, where $G$ equals $\GO_{2m}$ or $\GSp_{2m}$.
\end{abstract}

\maketitle

\tableofcontents

\section{Introduction}

Let $F$ be a totally real number field. We prove an $\msf R=\bb T$ theorem (see Theorem~\ref{peeoieowivmincs}) for self-dual Galois representations $\ovl\rho$ of $\Gal_F$ with coefficients in a finite field satisfying certain \emph{rigidity} property (see Definition~\ref{isiinefies}). We then show that for any elliptic curve $A$ over $F$, the residual Galois representations attached to symmetric powers of $A$ is rigid when the residual characteristic is large. More generally, we show that if $\Pi$ is a regular algebraic symplectically self-dual cuspidal representation of general linear groups over $F$ with a supercuspidal component, then the residual residual Galois representations attached to $\Pi$ is rigid when the residual characteristic is large.

The $\msf R=\bb T$ theorem for $\GL_2$ are crucial to Wiles and Taylor--Wiles' work on the Fermat's Last Theorem; cf.\cites{Wil95, T-W95}. For the original Taylor--Wiles method to work when considering a representation
\begin{equation*}
\rho: \Gal(\ovl F/F)\to G(\ovl{\bb Q_\ell}),
\end{equation*}
valued in a (not necessarily connected) reductive group $G$, one needs
\begin{equation*}
[F: \bb Q](\dim G-\dim B)=\sum_{v|\infty}\bx H^0(\Gal(\ovl F_v/F_v), \ad^0(\ovl\rho)),
\end{equation*}
where $B$ denotes a Borel subgroup of $G$, and $\ad^0(\ovl\rho)$ denotes the kernel from $\ad(\ovl\rho)$ to its $G$-coinvariants. As $F$ is totally real, this numerical coincidence holds for the disconnected group $\mrs G_n$ (with identity component $\GL_n\times\GL_1$) defined in \cite{CHT08}. This leads to $\msf R=\bb T$ theorems for unitary Shimura sets, which are essential in the proofs of potential modularity theorem for conjugate self-dual Galois representations and of the Sato--Tate conjecture (cf. for example, \cite{Tay08} and \cite{HSBT}). The case of unitary Shimura varieties is treated in~\cite{LTXZZa}.

This numerical coincidence also holds for $G=\GSp(2m)$ or $G=\GO(2m)$. The corresponding $\msf R=\bb T$ theorems for definite classical groups are treated in the work of Xiaoyu Zhang \cite{Zha25}. In this work, we consider orthogonal Shimura varieties (including orthogonal Shimura sets) associated to special orthogonal groups that are not necessarily unramified at every finite place, and with arbitrary neat level structures; this flexibility is crucial for our applications. For example, the main result will be used in the Ph.D. thesis of the first author \cite{Pen26}, which proves the Beilinson--Bloch--Kato conjecture for a large family of regular algebraic self-dual cuspidal representations of general linear groups over $F$; see \cite[\S1.2]{Pen25c}. 

We state one of the results that is the least technical to formulate. Let $\Pi$ be a $0$-REASDC (that is, regular essentially algebraic self-dual cuspidal with trivial central character) representation of $\GL_{2m}(\Ade_F)$ for some positive integer $m$ (see Definition~\ref{psoisieneifemiws}). Let $E\subset\bb C$ be a number field \sut there exists a family of Galois representations
\begin{equation*}
\{\rho_{\Pi, \lbd}: \Gal(\ovl F/F)\to \GL_{2m}(E_\lbd)\}_\lbd
\end{equation*}
indexed by finite places of $E$ \sut $\rho_{\Pi, \lbd}$ is conjugate to $\rho_{\Pi, \lbd}^\vee(1-2m)$ (see Proposition~\ref{associateGlaosireifnies}, Definition~\ref{strongienfeihenss}, and Remark~\ref{oeietureps}). Let $\Pla$ be a finite set of finite places of $F$ \sut for every finite place $v$ of $F$, $\Pi_v$ is unramified and the underlying rational prime of $v$ is unramified in $F$.

Consider a pair $(\mbf V, \mdc K)$ in which
\begin{itemize}
\item
$\mbf V$ is a nondegenerate quadratic space over $F$ of dimension $2m+1$ that is of signature $(N, 0)$ or $(N-2, 2)$ at every infinite place, 
and
\item
$\mdc K=\prod_v\mdc K_v$ is a neat open compact subgroup of $\SO(\mbf V)(\Ade_F^\infty)$, \sut $\mdc K_v$ is the stabilizer of a self-dual lattice for every $v$ not in $\Pla$.
\end{itemize}
We assume that there exists a cuspidal automorphic representation $\pi$ of $\SO(\mbf V)(\Ade_F)$ with nonzero $\mdc K$-invariants and with automorphic functorial lift $\Pi$ (see Definition~\ref{funcroliaureofnils}). Let $\bb T^{\mfk m, \Pla\cup\Pla(\ell_\lbd)}$ denote the Hecke algebra of $\SO(\mbf V)$ away from $\Pla$ and places lying above $\ell$ (see \textup{Definition~\ref{slabossleiisTakeows}}). Every such pair $(\mbf V, \mdc K)$ defines a complex Shimura variety $\bSh(\mbf V,\mdc K)$ of dimension $d(\mbf V)$. The Archimedean weights $\xi$ of $\Pi$ (see Definition~\ref{psoisieneifemiws}) defines an $\mcl O_\lbd$-linear local system $\mrs L_{\xi, \lbd}$ on $\bSh(\mbf V, \mdc K)$. Let $\bb T_\lbd$ be the image of $\bb T^{\mfk m,\Pla\cup\Pla(\ell_\lbd)}$ in $\End_{\mcl O_\lbd}\paren{\etH^{d(\mbf V)}(\bSh(\mbf V, \mdc K), \mrs L_{\xi, \lbd})}$.

For every finite place $\lbd$ of $E$, we write $\ell_\lbd$ its underlying rational prime and $\mcl O_\lbd$ the ring of integers of $E_\lbd$.

\begin{mainthm}\label{oeiiunughufmeis}
Suppose that 
\begin{enumerate}
\item
there exists a finite place $v$ of $F$ at which $\Pi_v$ is supercuspidal, and
\item
either the local system $\mrs L_{\xi, \lbd}$ is constant or $d(\mbf V)=0$.
\end{enumerate}
Then for all but finitely many finite places $\lbd$ of $E$, we can attach to $\Pi$ a maximal ideal $\mfk m_\lbd$ of $\bb T^{\mfk m, \Pla\cup\Pla(\ell_\lbd)}$ with residue field $\mcl O_\lbd/\lbd$, \sut 
\begin{enumerate}
\item
the localization $\bb T_{\lbd,\mfk m_\lbd}$ is canonically isomorphic to a commutative $\mcl O_\lbd$-algebra $\msf R_{\mrs S_\lbd}$ that classifies self-dual deformations of the residual representation of $\rho_{\Pi, \lbd}$ that are crystalline with regular Fontaine--Laffaille weights at places above $\ell_\lbd$ and unramified outside $\Pla$ and places above $\ell_\lbd$; and
\item
the localization $\etH^{d(\mbf V)}(\bSh(\mbf V, \mdc K), \mrs L_{\xi, \lbd})_{\mfk m_\lbd}$ is a finite free module over $\bb T_{\lbd, \mfk m_\lbd}$.
\end{enumerate}
\end{mainthm}
\begin{rem*}
The finite set of finite places of $E$ that are excluded in Theorem~\ref{oeiiunughufmeis} can be effectively bounded.
\end{rem*}

Theorem~\ref{oeiiunughufmeis} extends previous results in the conjugate self-dual setting \cites{CHT08, LTXZZa} to the self-dual setting. 
It is a consequence of the two main results of this article, namely, Theorem~\ref{peeoieowivmincs} and Theorem~\ref{ososinbuufemws}. In this article, a more general result is proved for orthogonal Shimura varieties of type $\bx B$ or $\bx D^{\bb R}$, under an extra assumption on certain vanishing of localized cohomologies of $\bSh(\mbf V, \mdc K)$ off middle degree (see~Theorem~\ref{peeoieowivmincs}(D4)). Note that, when $\mrs L_{\xi, \lbd}$ is a constant local system, this assumption is proved by the author \cite{Pen25a} when $\bSh(\mbf V, \mdc K)$ is proper, and later by Yang--Zhu \cite{Y-Z25} in general.

The proof of Theorem~\ref{oeiiunughufmeis} uses the Taylor--Wiles method following \cite{Whi23a}, which provided a framework for handling Galois representations valued in arbitrary connected reductive groups under a mild condition on the residual image.
Let $\ovl\rho$ be the residual representation of $\rho_{\Pi, \lbd}$, considered as a representation valued in $G(\mcl O_\lbd/\lbd)$, where $G=\GSp(2m)$ or $\GO(2m)$. In the Taylor--Wiles method, one considers liftings of $\ovl\rho$ for which ramification is allowed at certain auxiliary sets of places. In our setting, a Taylor--Wiles place $v$ of $F$ satisfies
\begin{itemize}
\item
$\ovl\rho$ is unramified at $v$;
\item
the residue field $\kappa(v)$ satisfies $\kappa(v)\equiv 1\modu \ell$; and
\item
the image of the Frobenius element $\ovl g=\ovl\rho(\phi_v)$ is semisimple and contained in the identity component of $G$.
\end{itemize}
We then consider Taylor--Wiles deformation problems as defined in \cite{Whi23a}, which consists of liftings of $\ovl\rho$ for which the inertia group $I_v$ is valued in certain torus. Since our deformation problem is in general not equal to the unrestricted deformation problem, some kind of local-global compatibility is required to show that Galois representations (attached to those automorphic representations arising in the Taylor--Wiles method) satisfy our Taylor--Wiles deformation problems.
The main input here (aside from local-global compatibility up to semisimplification) is the local Langlands correspondence for principal series representations constructed in \cite{Sol25}, whose construction is explicit enough for us to establish the necessary vanishing of the monodromy operator in the associated Weil--Deligne representation whilst also coinciding with Arthur's construction of local Langlands for classical groups.

In order to control the arising universal deformation rings, one requires the existence sets of Taylor--Wiles places for which a certain Galois cohomology group (the dual Selmer group) is forced to vanish. We impose a `big image' condition on $\ovl\rho$ using the notion of \emph{$G$-adequate} subgroups from \cite{Whi23a}, which can be verified to hold whenever $\ell$ (the rational prime underlying $\lbd$) is large 
and $\ovl\rho$ is absolutely irreducible. Moreover, in \S5, we verify that $\ovl\rho$ is absolutely irreducible for all but finitely many $\lbd$.

Finally, in order to deal with finite places of $F$ away from $\ell$ where $\ovl\rho$ is ramified, we use the notion of \emph{minimally ramified} representations from \cite{Boo19a}. We introduce a notion of \emph{rigid} Galois representation in \S4.1, so that when $\ovl\rho$ is rigid, any lifting of $\ovl\rho$ is minimally ramified at finite places of $F$ away from $\ell$. Moreover, in \S5, we verify that that $\ovl\rho$ is rigid for all but finitely many $\lbd$. We emphasize that the the proof of rigidity ultimately relies on classical modularity lifting theorems, in particular, the classical $\msf R=\bb T$ theorems in the unitary setting \cites{CHT08, Tho12}.

\subsection{Notation and conventions}

\begin{note}\enskip
\begin{itemize}
\item
Let $\bb N =\{0, 1, 2, 3, . . . \}$ be the monoid of nonnegative integers and set $\bb Z_+=\bb N\setm\{0\}$. We write $\bb Z$, $\bb Q$, $\bb R$, and $\bb C$ for the integers, rational numbers, real numbers, and complex numbers, respectively.
\item
All rings are commutative and unital, and ring homomorphisms preserve units.
\item
We take square roots only of positive real numbers and always choose the positive root.
\item
Suppose $X$ is a set.
\begin{itemize}
\item
Let $\uno\in X$ denote the distinguished trivial element (this notation is only used when the notion of triviality is clear from context).
\item
Let $\#X$ be the cardinality of $X$.
\item
For $a, b\in X$, we define the Kronecker symbol
\begin{equation*}
\delta_{a, b}\defining\begin{cases} 1 &\If a=b\\ 0 &\If a\ne b\end{cases}.
\end{equation*}
\end{itemize}
\item
For two topological (resp. algebraic) groups $H$ and $G$, we write $H\le G$ if $H$ is a closed subgroup of $G$.
\item
For each algebraic group $G$ over a field $K$, we write $G^\circ$ for the neutral component of $G$.
\item
For each square matrix $M$ over a ring, we write $M^\top$ for its transpose.
\item 
For each positive integer $n$, we denote by $\Sym_n$ the $n$-th symmetric group acting on $\{1, \ldots, n\}$.
\item
We use standard notation from category theory. The category of sets is denoted by $\Set$. The category of schemes is denoted by $\Sch$.
\item
For each rational prime $p$, we fix an algebraic closure $\ovl{\bb Q_p}$ of $\bb Q_p$ with residue field $\ovl{\bb F_p}$. For every integer $r\in\bb Z_+$, we denote by $\bb Q_{p^r}$ the unique unramified extension of $\bb Q_p$ of degree $r$ inside $\ovl{\bb Q_p}$, and by $\bb F_{p^r}$ its residue field.
\item
We fix a totally real number field $F$ and an algebraic closure $\ovl F$ of $F$. Set $\Gal_F\defining\Gal(\ovl F/F)$.
\item
We fix an integer $m\ge 1$ and an element $\mfk d$ in $F$ \sut $(-1)^m\mfk d$ is totally nonnegative (that is, $(-1)^m\tau(\mfk d)\ge 0$ for every embedding $\tau: F\to \bb R$). 
Denote by $\chi_{\mfk d}: F^\times\bsh \Ade_F^\times\to \bb C^\times$ the character corresponding to the Galois extension $F(\sqrt{\mfk d})/F$ via global class field theory (which is either trivial or quadratic).
\item
We denote by $\Mat_{2m}$ (resp. $\GL_{2m}$) the scheme over $\bb Z$ of $2m\times 2m$ square matrices (resp. invertible $2m\times 2m$ square matrices).
\item
We define $N=2m+\delta_{\mfk d, 0}$, and assume $N\ge 3$.
\item
For each positive integer $k$, we denote by $I_k$ the unit element of $\GL_N(\bb Z)$. Let $J_k=\sum_{i=1}^ke_{1, k+1-i}$, which is an anti-diagonal $k\times k$ matrix in $\Mat_k(\bb Z)$, let $J'_k=\sum_{i=1}^k(-1)^{i-1}e_{i, k+1-i}$, which is an anti-diagonal $k\times k$ matrix in $\Mat_k(\bb Z)$, and let $N_k=\sum_{i=1}^{k-1}e_{i, i+1}$, which is a nilpotent matrix.
\item
Denote by $\Pla^\infty$ the set of real embeddings of $F$, and by $\Pla^\fin$ the set of finite places of $F$.
\item
For any set $S$ of rational primes, denote by $\Pla(S)$ the set of all finite places $v$ of $F$ with $\chr\kappa_v\in S$. If $S=\{p\}$ is a singleton, we write simply $\Pla(p)\defining \Pla(\{p\})=\{v\in\fPla_F: v|p\}$.
\item
Let $\Pla^\bad$ denote the (finite) set of finite places of $F$ whose underlying rational prime is ramified in $F(\sqrt{\mfk d})$.
\item
For each finite set $\Pla$ of places of $F$, we write
\begin{equation*}
\Ade_{F, \Pla}\defining\prod\nolimits_{v\in\Pla}F_v, \quad \Ade_F^\Pla\defining\prod_{v\in \Pla_F\setm\Pla} \nolimits' F_v, \quad \Ade_F\defining \Ade_F^\infty\times(F\otimes_{\bb Q}\bb R)\end{equation*}
\item
For each place $v$ of $F$, we 
\begin{itemize}
\item
write $\mcl O_v$ and $F_v$ for the completion of $\mcl O_F$ (resp. $F$) at $v$;
\item
let $\kappa_v$ denote the residue field, $\norml{v}\defining\#\kappa_v$, and write $\chr \kappa_v$ for the residue characteristic.
\item
we fix an algebraic closure $\ovl{F_v}$ of $F_v$ with residue field $\ovl{\kappa_v}$ and an embedding $\iota_v: \ovl F\inj \ovl{F_v}$ extending $F\inj F_v$; via $\iota_v$ we regard $\Gal_{F_v}\defining \Gal(\ovl{F_v}/F_v)$ as a decomposition subgroup of $\Gal_F$;
\item
for any map $r: \Gal_F\to X$, we write $r_v\defining r|_{\Gal_{F_v}}$;
\item
let $\bx I_{F_v}\subset \Gal_{F_v}$ be the inertia subgroup;
\item
let $W_{F_v}\subset \Gal_F$ be the Weil group.
\item
fix an \emph{arithmetic} Frobenius element $\phi_v\in\Gal_{F_v}$.
\end{itemize}
\item
For each finite set $\Pla$ of finite places of $F$, we denote by $\Gal_{F, \Pla}$ the Galois group of the maximal subextension of $\ovl F/F$ that is unramified outside $\Pla\cup\Pla^\infty$.
\item
For each rational prime $\ell$, we denote by $\ve_\ell: \Gal_F\to \bb Z_\ell^\times$ the $\ell$-adic cyclotomic character.
\item
Suppose $\Gamma, G$ are groups and $L$ is a ring.
\begin{itemize}
\item
We denote by $\Gamma^\ab$ the maximal abelian quotient of $\Gamma$;
\item
For a homomorphism $\rho: \Gamma\to \GL_r(L)$ for some $r\in\bb Z_+$, we denote by $\rho^\vee: \Gamma\to \GL_r(L)$ the contragredient homomorphism, which is defined by the formula $\rho^\vee(x)=(\rho(x)^\top)^{-1}$.
\item
We say that two homomorphisms $\rho_1, \rho_2: \Gamma\to G$ are conjugate if there exists an element $g\in G$ \sut $\rho_1=g\circ\rho_2\circ g^{-1}$.
\item
Suppose $\Gamma$ is a subgroup of another group $\tld\Gamma$, and an element $\gamma\in\tld\Gamma$ normalizes $\tld\Gamma$. For any $\rho: \Gamma\to G$ for some $r\in\bb Z_+$, we let $\rho^\gamma: \Gamma\to G$ be the homomorphism defined by $\rho^\gamma(x)=\rho(\gamma x\gamma^{-1})$ for every $x\in\Gamma$.
\end{itemize}
\item
For each positive integer $n$, let $\mu_n$ be the group scheme over $\Spec\bb Z$ representing the group of $n$-th roots of unity.
\end{itemize}
\end{note}

\subsection{Acknowledgments}

We thank Anne-Marie Aubert, Yifeng Liu, Zeyu Wang, Yichao Tian, Liang Xiao, and Jialiang Zou for helpful conversations.

\section{Preliminaries}

\subsection{Orthogonal Satake parameters and orthogonal Hecke algebras}

In this subsection, we introduce quadratic spaces, their associated special orthogonal groups and orthogonal Hecke algebras.

\begin{defi}
Let $R$ be a $\mcl O_F$-algebra. 
A quadratic space over $R$ of dimension $N$ is a projective $R$-module $V$ of rank $N$ together with a symmetric perfect pairing
\begin{equation*}
V\times V\to R
\end{equation*}
that is $R$-linear in both variables.

For any quadratic space $V$ of dimension $N$ over $R$, we denote by $V_\sharp$ the quadratic space $V\oplus Re$ where $\norml{e}=1$, which is a quadratic space of dimension $N+1$ over $R$. For any $R$-linear isometry $f: V\to V'$ of quadratic spaces over $R$, we denote by $f_\sharp: V_\sharp\to V_\sharp'$ the induced isometry.
\end{defi}

We introduce abstract orthogonal Satake parameters.

\begin{defi}\label{slabossleiisTakeows}
Let $L$ be a ring. For any multiset $\bm\alpha=\{\alpha_1, \ldots, \alpha_{2m}\}\subset L$, we put
\begin{equation*}
P_{\bm\alpha}(T)=\prod_{i\in[2m]_+}(T-\alpha_i)\in L[T].
\end{equation*}

For each finite place $v$ of $F$ not in $\Pla^\bad$, we define an \tbf{abstract orthogonal Satake parameter} in $L$ at $v$ to be a multiset $\bm\alpha$ consisting of $2m$ elements in $L$ \sut $P_{\bm\alpha}(T)=\chi_{\mfk d}(\phi_v)T^{2m}P_{\bm\alpha}(T^{-1})$.
\end{defi}

For each finite place $v$ of $F$ not in $\Pla^\bad$, let $\Lbd^{m, \mfk d}_v$ be the unique up to isomorphism quadratic space over $\mcl O_v$ of rank $N$ with discriminant $\mfk d$ (resp. discriminant $1$) when $N$ is even (resp. odd), 
and $\SO^{m, \mfk d}_v$ the associated special orthogonal group over $\mcl O_v$, considered as a subgroup of $\GL(\Lbd^{m, \mfk d})$. In particular, we are considering odd special orthogonal groups when $\mfk d=0$ (thus $N$ is odd), and considering even special orthogonal groups when $\mfk d\ne 0$ (thus $N$ is even). Under suitable basis, the associated quadratic form of $\Lbd^{m, \mfk d}_v$ is given by the matrix
\begin{equation*}
\begin{cases}
J_{2m+1} &\If \mfk d=0,\\
J_{2m}&\If \mfk d\in F_v^2\setm\{0\},\\
\begin{bmatrix}
&     &  & J_{m-1}\\
& 1 &            0& \\
& 0& -\mfk d &  \\
J_{m-1}&  &             &  \\
\end{bmatrix} &\If \mfk d\notin F_v^2.\\
\end{cases}
\end{equation*}

Consider the integral local spherical Hecke algebra
\begin{equation*}
\bb T^{m, \mfk d}_v\defining \bb Z[\SO^{m, \mfk d}_v(\mcl O_{F_v})\bsh \SO^{m, \mfk d}_v(F_v)/\SO^{m, \mfk d}_v(\mcl O_{F_v})]
\end{equation*}
with unit element $\uno_{\SO^{m, \mfk d}_v(\mcl O_{F_v})}$. We write $A^{m, \mfk d}_v$ for a maximal split diagonal torus of $\SO^{m, \mfk d}_v$ and $X_\bullet(A^{m, \mfk d}_v)$ its cocharacter group, then we have the Satake transform
\begin{equation}\label{Sianineifiems}
\CT^\cl: \bb Z\Bkt{\sqrt{\norml{v}}}\otimes\bb T^{m, \mfk d}_v\to \bb Z\Bkt{\sqrt{\norml{v}}}[X_\bullet(A^{m, \mfk d}_v)].
\end{equation}
It is a homomorphism of rings. 

\begin{defi}\label{aosisieeinifmews}
For a finite set $\Pla$ of finite places of $F$ containing $\Pla^\bad$, we define the \emph{abstract orthogonal Hecke algebra} away from $\Pla$ to be the restricted tensor product
\begin{equation*}
\bb T^{m, \mfk d, \Pla}\defining\btimes_{v\in \fPla_F\setm\Pla}\nolimits'\bb T^{m, \mfk d}_v
\end{equation*}
\wrt unit elements. It is a ring.
\end{defi}

\begin{defi}\label{ieifiemifesqpo}
Suppose $v$ is a finite place of $F$ not in $\Pla^\bad$, and $\bm\alpha$ is an abstract orthogonal Satake parameter in $L$ at $v$. We construct a Hecke character of $\bb T^{m, \mfk d}_v$. There are two cases:
\begin{itemize}
\item
If $\mfk d$ is a square in $F_v$, then a set of representatives of $X_\bullet(A^{m, \mfk d}_v)$ can be taken as
\begin{equation*}
\{(t_1, \ldots, t_N)|t_1, \ldots, t_N\in\bb Z \text{ satisfying }t_i+t_{N+1-i}=0\text{ for all }i\in[N]_+\}.
\end{equation*}
Because $\Pi_v$ has trivial central character, we can choose an ordering of $\bm\alpha$ as $(\alpha_1, \ldots, \alpha_{2m})$ satisfying $\alpha_{N+1-i}=\alpha_i^{-1}$ for each $i\in[N]_+$. Then we have a unique ring homomorphism
\begin{equation*}
\bb Z\Bkt{\sqrt{\norml{v}}}[X_\bullet(A^{m, \mfk d}_v)]\to \bb C
\end{equation*}
sending $(t_1, \ldots, t_N)$ to $\prod_{i\in[m]_+}\alpha_i^{t_i}$. Composing with the Satake transform \eqref{Sianineifiems}, we obtain a ring homomorphism
\begin{equation*}
\phi_{\Pi_v}: \bb T^{m, \mfk d}_v\to \bb C.
\end{equation*}
It is independent of the ordering of $\bm\alpha$. 
\item
If $\mfk d$ is not a square in $F_v$, then $N=2m$, and a set of representatives of $X_\bullet(A^{m, \mfk d}_v)$ can be taken as
\begin{equation*}
\{(t_1, \ldots, t_{2m})|t_1, \ldots, t_N\in\bb Z \text{ satisfying }t_i+t_{2m+1-i}=0\text{ for all }i\in[2m]_+\text{ And }t_m=t_{m+1}=0\}. 
\end{equation*}
Because $\Pi_v$ has central character $\chi_{\mfk d, v}$, we can choose an ordering of $\bm\alpha$ as $(\alpha_1, \ldots, \alpha_{2m})$ satisfying $\alpha_m=1, \alpha_{m+1}=-1$,
and $\alpha_{2m+1-i}=\alpha_i^{-1}$ for each $i\in[m-1]_+$. Then we have a unique ring homomorphism
\begin{equation*}
\bb Z\Bkt{\sqrt{\norml{v}}}[X_\bullet(A^{m, \mfk d}_v)]\to \bb C.
\end{equation*}
sending $(t_1, \ldots, t_{2m})$ to $\prod_{i\in[m-1]_+}\alpha_i^{t_i}$. Composing with the Satake transform \eqref{Sianineifiems}, we obtain a ring homomorphism
\begin{equation*}
\phi_{\Pi_v}: \bb T^{m, \mfk d}_v\to \bb C.
\end{equation*}
It is independent of the ordering of $\bm\alpha$. 
\end{itemize}
\end{defi}

\subsection{Automorphic representations}

In this subsection, we recall some facts concerning automorphic representations.

\begin{defi}
We denote by $\bb Z_\le^{2m}$ the subset of $\bb Z^{2m}$ consisting of nondecreasing sequences. For a finite set $T$ and an element $\xi=(\xi_\tau)_{\tau\in T}\in(\bb Z_\le^{2m})^T$, set
\begin{equation*}
a_\xi\defining \min_{\tau\in T}\{\xi_{\tau, 1}\}, \quad b_\xi\defining\max_{\tau\in T}\{\xi_{\tau, 2m}\}+N-2
\end{equation*}
\end{defi}

\begin{defi}\label{psoisieneifemiws}
Suppose $\Pi$ is an automorphic representation  of $\GL_{2m}(\Ade_F)$. We say $\Pi$ is $\mfk d$-\tbf{REASDC} (that is, regular essentially algebraic self-dual cuspidal) 
if
\begin{enumerate}
\item
$\Pi$ is a cuspidal automorphic representation;
\item
$\Pi\cong \Pi^\vee$;
\item
the central character of $\Pi$ is equal to $\chi_{\mfk d}$; and
\item
for every Archimedean place $\tau$ of $F$, $\Pi_\tau$ has infinitesimal weight $\lbd_\tau=(\lbd_{\tau, 1}, \ldots, \lbd_{\tau, 2m})\in \bb Z^{2m}+\paren{\paren{\frac{N-1}{2}}^{(2m)}}$,\footnote{This notation means an $2m$-tuple with all entries being $\frac{N-1}{2}$.} satisfying that $\lbd_{\tau, 1}<\lbd_{\tau, 2}<\ldots< \lbd_{\tau, m}< 0<\lbd_{\tau, m+1}<\lbd_{\tau, m+2}<\ldots<\lbd_{\tau, 2m}$.
\end{enumerate}

If $\Pi$ is $\mfk d$-REASDC, then there is a unique element $\xi_\Pi=(\xi_{\tau, 1}, \ldots, \xi_{\tau, 2m})_\tau\in (\bb Z_\le^{2m})^{\Pla^\infty}$, which we call the \emph{Archimedean weights} of $\Pi$, satisfying $\xi_{\tau, i}=-\xi_{\tau, 2m+1-i}$ for every $\tau$ and $i$, \sut
\begin{equation*}
\lbd_\tau=\xi_\tau+\paren{1-\frac{N}{2}, 2-\frac{N}{2}, \ldots, -\delta_{\mfk d, 0}-1, -\delta_{\mfk d, 0}, \delta_{\mfk d, 0}, \delta_{\mfk d, 0}+1, \ldots \frac{N}{2}-1}.
\end{equation*}
\end{defi}

Let $v$ be a finite place of $F$. For every irreducible admissible representation $\Pi$ of $\GL_{2m}(F_v)$, every rational prime $\ell$, and every isomorphism $\iota_\ell: \bb C\xr\sim\ovl{\bb Q_\ell}$, we denote by $\rec_{2m}(\iota_\ell\Pi)$ the (Frobenius-semisimple) Weil--Deligne representation associated with $\iota_\ell\Pi$ via the local Langlands correspondence; see~\cites{Hen00, H-T01}. For each representation $\rho$ of $\Gal_{F_v}$ with coefficients in $\ovl{\bb Q_\ell}$ that is continuous (resp. de Rham) if $v\in\Pla(\ell)$ (resp. $v\notin\Pla(\ell)$), let $\WD(\rho)$ be the corresponding Weil--Deligne representation of $W_{F_v}$ (see, for example, \cite[\S1]{T-Y07}).

\begin{prop}\label{associateGlaosireifnies}
Let $\Pi$ be $\mfk d$-REASDC representation of $\GL_{2m}(\Ade_F)$ with Archimedean weights $\xi_\Pi$. Let $G$ denote the split smooth group scheme $\GSp_{2m}$ (resp. $\GO_{2m}$) over $\bb Q$ if $\mfk d=0$ (resp. $\mfk d\ne 0$). We denote by $\Std$ the standard representation of $G$ over $\mcl O$ of dimension $2m$, and by $\nu: G\to \GL_1$ the similitude character.
\begin{enumerate}
\item
$\Pi_v$ is tempered for each finite place of $F$.
\item
For any rational prime $\ell$ and any isomorphism $\iota_\ell: \bb C\xr\sim\ovl{\bb Q_\ell}$, there is a continuous homomorphism
\begin{equation*}
\rho_{\Pi, \iota_\ell}: \Gal_F\to G(\ovl{\bb Q_\ell})
\end{equation*}
unique up to conjugation, \sut
\begin{equation}\label{iseihienfies}
\WD\paren{\Std\circ\rho_{\Pi, \iota_\ell}|_{\Gal_{F_v}}}^{\bx F\dash\sems}=\iota_\ell\rec_{2m}\paren{\Pi_v\otimes\largel{\det}^{\frac{2-N}{2}}}
\end{equation}
for every finite place $v$ of $F$. Moreover, we have
\begin{equation*}
\nu\circ\rho_{\Pi, \iota_\ell}=\ve_\ell^{2-N}, \quad \det(\Std\circ\rho_{\Pi, \iota_\ell})=\chi_{\mfk d}\cdot\ve_\ell^{m(2-N)}.
\end{equation*}
\item
For every finite place $v$ of $F$ lying above $\ell$, $\Std\circ\rho_{\Pi, \iota_\ell}|_{\Gal_{F_v}}$ is de Rham (and furthermore crystalline if $\Pi_v$ is unramified) with regular Hodge--Tate weights contained in the range $[a_\xi, b_\xi]$.
\end{enumerate}
\end{prop}
\begin{proof}
We first note that $\Pi\otimes\largel{\det}^{\frac{2-N}{2}}$ is regular algebraic. Then Part (1) follows from \cite{Car12}*{Theorem 1.2}. For part (2): Let $J_{\mfk d}\in\Mat_N$ denote the $N$ by $N$ matrix 
\begin{equation*}
\begin{cases}
J_{2m+1} & \If \mfk d=0,\\
\begin{bmatrix}
&     &  & J_{m-1}\\
& 1 &            0& \\
& 0& -\mfk d &  \\
J_{m-1}&  &             &  \\
\end{bmatrix} &\If \mfk d\ne 0,
\end{cases}
\end{equation*}
and let $\SO_N^{\mfk d}$ be the quasi-split special orthogonal group over $F$ defined by
\begin{equation*}
\SO_N^{\mfk d}=\{g\in \GL_N: g^\top J_{\mfk d}g=J_{\mfk d}\}.
\end{equation*}
Then $\SO_N^{\mfk d}$ is quasi-split at every place of $F$. It follows from Arthur's multiplicity formula \cite{Art13} that there is a automorphic representation $\pi$ of $\SO_N^{\mfk d}(\Ade_F)$ whose functorial lifting to $\GL_{2m}(\Ade_F)$ is $\Pi$ (cf. Definition~\ref{funcroliaureofnils}). By \cite[Theorem~1.2.2]{Shi24}, there is a continuous homomorphism
\begin{equation*}
\rho_{\Pi, \iota_\ell}: \Gal_F\to G(\ovl{\bb Q_\ell})
\end{equation*}
associated to $\pi$ satisfying Equation~\eqref{iseihienfies} at every places of $F$ where $\pi$ is unramified. It follows from the Chebotarev density theorem that $\Std\circ\rho_{\Pi, \iota_\ell}$ is just the Galois representation associated to $\Pi$ in the sense of \cite[Theorem~2.1.1]{LGGT}. 
Thus Equation~\eqref{iseihienfies} holds for each place of $F$, by \cite{Car12}*{Theorem 1.1} and \cite{Car14}*{Theorem 1.1}. The last two properties also follow from \cite{Shi24} and the Chebotarev density theorem. 
Part (3) is proved in \cite[Theorem~4.2]{C-H13}.
\end{proof}

It is expected that these associated Galois representations form a compatible system defined over a number field (as defined in \cite{LGGT}*{\S 5.1}). More explicitly, we define the notion of strong coefficient fields:



\begin{defi}\label{isioswooiemis}
For each finite place $v$ of $F$ not in $\Pla^\Pi$, let $\bm\alpha(\Pi_v)$ denote the Satake parameter of $\Pi_v$, which is an abstract Satake parameter in $\bb C$ of dimension $2m$ (see Definition~\ref{slabossleiisTakeows}), and let $\bb Q(\Pi_v)$ denote the subfield of $\bb C$ generated by the coefficients of the polynomial
\begin{equation*}
\prod_{\alpha\in\bm\alpha(\Pi_v)}\bigg(T-\alpha\sqrt{\norml{w}}^N\bigg)\in\bb C[T].
\end{equation*}

We define the \tbf{coefficient field} (also called the \tbf{Hecke field}) of $\Pi$ to be the compositum of the fields $\bb Q(\Pi_v)$ for all finite places $w$ of $F$ not in $\Pla^\Pi$, denoted by $\bb Q(\Pi)$.
\end{defi}


\begin{defi}\label{strongienfeihenss}
Let $\Pi$ be a $\mfk d$-REASDC representation of $\GL_{2m}(\Ade_F)$ (see Definition~\ref{psoisieneifemiws}) then a number field $E\subset \bb C$ is called a \tbf{strong coefficient field} for $\Pi$ if it contains $\bb Q(\Pi)$ (see Definition~\ref{isioswooiemis}), and for each finite place $\lbd\in\fPla_E$, there exists a continuous homomorphism
\begin{equation*}
\rho_{\Pi, \lbd}: \Gal_F\to \GL_{2m}(E_\lbd),
\end{equation*}
\sut for any isomorphism $\iota_\ell: \bb C\xr\sim\ovl{\bb Q_\ell}$ inducing the place $\lbd$, $\rho_{\Pi, \lbd}\otimes_{E_\lbd, \iota_\ell}\ovl{\bb Q_\ell}$ and $\Std\circ \rho_{\Pi, \iota_\ell}$ are conjugate (see Proposition~\ref{associateGlaosireifnies}).
\end{defi}
\begin{rem}\label{oeietureps}
If $\Pi$ is $\mfk d$-REASDC representation of $\GL_{2m}(\Ade_F)$, then a strong coefficient field of $\Pi$ exists by the argument of \cite[Proposition~3.2.5]{C-H13}.
\end{rem}

\subsection{Functorial liftings}\label{ososiieucume}

We recall we recall some general facts about the local and global functorial liftings for special orthogonal groups following \cite{C-Z21}, \cite{Ish24} and \cite{Pen25a}*{\S2}. For each element $\alpha\in\bb C^\times$, we denote by $\udl\alpha: K^\times\to \bb C^\times$ the unramified character with Satake parameter $\alpha$.

Let $K$ be a finite extension of $\bb Q_p$, and let $V$ be a quadratic space over $K$. For an irreducible admissible representation $\pi$ of $\SO(V)$, we recall the local functorial liftings of $\pi$ defined by endoscopy theory \cites{Art13, C-Z21, Ish24} in certain special cases.

Fix a minimal parabolic subgroup $P_{\min}\le\SO(V)$. By Langlands' classification, there exists a unique parabolic subgroup $P\le G$ containing $P_{\min}$ with Levi quotient $M$, a unique tempered admissible irreducible representation $\tau$ of $M(K)$, and a strictly positive unramified character $\chi$ of $P(K)$ \sut $\pi$ is the unique irreducible quotient $\bx J_P^G(\tau(\chi))$ of the normalized parabolic induction $\bx I_P^G(\tau(\chi))$, called the \tbf{Langlands quotient}. Then we may write
\begin{equation*}
M=G_0\times\GL_{r_1}\times\cdots\times\GL_{r_t},
\end{equation*}
with $G_0$ the special orthogonal factor, under which
\begin{equation*}
\chi=\uno\boxtimes(\udl{\alpha_1}\circ\det\nolimits_{r_1})\boxtimes\cdots\boxtimes(\udl{\alpha_t}\circ\det\nolimits_{r_t})
\end{equation*}
for unique real numbers $1<\alpha_1<\cdots<\alpha_t$, where $\det\nolimits_{r_i}$ denotes the determinant map on $\GL_{r_i}(K)$. Suppose $\tau=\tau_0\boxtimes\tau_1\boxtimes\cdots\boxtimes\tau_t$ under the above decomposition, then $\FL(\pi)$ is isomorphic to the unique irreducible quotient of the normalized parabolic induction 
\begin{equation*}
\bx I_{P_\GL}^{\GL_{2m}}\paren{\tau_t^\vee(\udl{\alpha_t^{-1}}\circ\det\nolimits_{r_t})\boxtimes\cdots\boxtimes\tau_1^\vee(\udl{\alpha_1^{-1}}\circ\det\nolimits_{r_1})\boxtimes\FL(\tau_0)\boxtimes\tau_1(\udl{\alpha_1}\circ\det\nolimits_{r_1})\boxtimes\cdots\boxtimes\tau_t(\udl{\alpha_t}\circ\det\nolimits_{r_t})},
\end{equation*}
where $P_\GL$ is a standard parabolic subgroup of $\GL_{2m}$ whose Levi quotient is $\GL_{r_t}\times\cdots\GL_{r_1}\times\GL_{2m_0}\times \GL_{r_1}\times\cdots\times\GL_{r_t}$.

Now we consider the tempered admissible irreducible representation $\tau_0$. There exists a unique parabolic subgroup $P_0$ of $G_0$ containing $P_\bmin\cap G_0$ with Levi quotient group $M_0$, and a discrete series representation $\mu$ of $M_0(K)$ \sut $\tau_0$ is a direct summand of the normalized parabolic induction $\bx I_{P_0}^{G_0}(\mu)$. In fact, $\bx I_{P_0}^{G_0}(\mu)$ is a direct sum of finitely many tempered representations of multiplicity one. Write $\mu=\mu_0\boxtimes\mu_1\boxtimes\cdots\boxtimes\mu_s$, similar to the previous case, where $\mu_0$ is a discrete series representation of the special orthogonal factor of the Levi quotient of $P_0$. Then under similar notation, $\FL(\tau_0)$ is isomorphic to
\begin{equation*}
\bx I^{\GL_{2m_0}}\paren{\mu_t^\vee\boxtimes\cdots\boxtimes\mu_1^\vee\boxtimes\FL(\mu_0)\boxtimes\mu_1\boxtimes\cdots\boxtimes\mu_t},
\end{equation*}
which is an irreducible admissible representation of $\GL_{2m_0}(K)$. Here and henceforth, we suppress the parabolic subgroup $P_\GL$ of $\GL_{2m_0}$ when it is clear. 

Finally, we recall the notion of automorphic functorial liftings.

\begin{defi}\label{funcroliaureofnils}
Let $\mbf V$ be a quadratic space over $F$ of dimension $N$, and let $\pi$ be a discrete automorphic representation of $\SO(\mbf V)(\Ade_F)$. An \tbf{automorphic functorial lifting} of $\pi$ is an automorphic representation $\FL(\pi)$ of $\GL_{2m}(\Ade_F)$ which is a finite isobaric sum of discrete automorphic representations of smaller general linear groups \sut $\FL(\pi)_v$ is a functorial lifting of $\pi_v$ for all but finitely many unramified $v\in\fPla_F$. 

By \cite[Theorem~2.1]{C-Z24}, \cite[Theorem~3.16]{Ish24}, and the strong multiplicity one property for $\GL_N$ (see~\cite{Sha79}), $\FL(\pi)$ exists uniquely. Moreover, if $\FL(\pi)$ is a finite isobaric sum of cuspidal self-dual representations of smaller general linear groups, then $\FL(\pi)_v$ is a functorial lifting of $\pi_v$ for every place $v$ of $F$, by Arthur's multiplicity formula for generic $A$-packets; see \cite[Theorem~2.6]{C-Z24} and \cite[Theorem~3.17]{Ish24}.
\end{defi}

\subsection{Almost unramified representations}

In this subsection, we study a special kind of representations of the nonsplit $p$-adic odd special orthogonal groups.

Let $r$ be a positive integer. Let $K$ be a finite extension of $\bb Q_p$ with ring of integers $\mcl O_K$ and residue field $\kappa$. Let $V$ be a quadratic space of dimension $2r+1$ over $K$ equipped with a lattice $\Lbd\subset V$ satisfying that $\Lbd\subset\Lbd^\vee$ and $\Lbd^\vee/\Lbd\cong\kappa$. Then $\SO(V)$ is a nonsplit pure inner form of $\SO_{2r+1, K}$. Let $\mdc K$ be the stabilizer of $\Lbd$ in $\SO_{2r+1}'$, which is a special maximal subgroup of $\SO(V)$.

\begin{prop}\label{oieiticimcimdws}
Let $\pi$ be an irreducible admissible representation of $\SO(V)(K)$ with nonzero invariants by $\mdc K$. Then there exists an element $(\alpha_2, \ldots, \alpha_r)\in(\bb C^\times)^{r-1}$ satisfying $1\le\largel{\alpha_2}\le\ldots\le\largel{\alpha_r}$, unique up to permutation, \sut the local functorial lifting $\FL(\pi)$ (see \textup{\S\ref{ososiieucume}}) is isomorphic to the unique irreducible quotient of the normalized parabolic induction
\begin{equation*}
\bx I_{P_\GL}^{\GL_{2r}}\paren{\udl{\alpha_r^{-1}}\boxtimes\cdots\boxtimes\udl{\alpha_2^{-1}}\boxtimes\St_2\boxtimes\udl{\alpha_2}\boxtimes\cdots\boxtimes\udl{\alpha_r}}.
\end{equation*}
Here $P_\GL$ is a standard parabolic subgroup of $\GL_{2r}$ whose Levi quotient is $\GL_1^{r-1}\times\GL_2\times \GL_1^{r-1}$, and $\St_2$ is the Steinberg representation of $\GL_2(K)$.
\end{prop}
\begin{proof}
We fix a decomposition
\begin{equation*}
\Lbd=\mcl O_Ke_{-r}\oplus\cdots\oplus\mcl O_Ke_{-2}\oplus\Lbd_3\otimes\mcl O_Ke_2\oplus\cdots\oplus \mcl O_Ke_r,
\end{equation*}
in which $(e_i, e_j)=\delta_{i+j, 0}$ for $2\le\largel{i}, \largel{j}\le r$ and each $e_i$ is orthogonal to $\Lbd_3$. For each $1\le i\le r$, set
\begin{equation*}
V_{2i+1}\defining Ke_{-i}\oplus\cdots\oplus Ke_{-2}\oplus\Lbd_3\otimes_{\mcl O_K}K\otimes Ke_2\oplus\cdots\oplus Ke_i,
\end{equation*}
which is a quadratic subspace of $V$. Let $P_\bmin$ be the minimal parabolic subgroup of $\SO(V)$ defined by the isotropic flag
\begin{equation*}
Ke_{-r}\subset Ke_{-r}\oplus Ke_{-r+1}\subset\cdots\subset Ke_{-r}\oplus\cdots\oplus Ke_{-2},
\end{equation*}
with the Levi quotient
\begin{equation*}
\SO(V_3)\times\GL(Ke_2)\times\cdots\times\GL(Ke_r).
\end{equation*}

We now realize the relative Weyl group $W\cong \{\pm1\}^{r-1}\rtimes\Sym_{r-1}$ 
explicitly as a subgroup of $\mdc K$:
\begin{itemize}
\item
We realize $\Sym_{r-1}$ as the stabilizer of the element $1\in\{1,\ldots, r\}$ in $\Sym_r$;
\item
for each $2\le i\le r$, we let the $(i-1)$-th $``-1''$ in $W$ correspond to the element switching $e_i$ and $e_{-i}$ and fixing other $e_j$ for $j\ne i$; and
\item
we let each $\sigma\in \Sym_{r-1}$ correspond to the element sending $e_{\pm i}$ to $e_{\pm\sigma(i)}$ for every $2\le i\le r$, denoted by $w_\sigma$.
\end{itemize}
For each element $\bm\alpha=(\alpha_2, \ldots, \alpha_r)\in (\bb C^\times)^{r-1}$, let $\phi_{\bm\alpha}$ denote the function in $\bx I_{P_\bmin}^{\SO(\mbf V)}(\uno_3\boxtimes\udl{\alpha_2}\boxtimes\cdots\boxtimes\udl{\alpha_r})$ that takes value $1$ on $\mdc K$, where $\uno_3$ denotes the trivial representation of $\SO(V_3)(K)$.

Because $\pi$ has nonzero invariants by $\mdc K$, it is a constituent of an unramified principal series. In other words, there exists an element $\bm\alpha=(\alpha_2, \ldots,\alpha_r)\in(\bb C^\times)^{r-1}$ satisfying $1\le\largel{\alpha_2}\le\cdots\le\largel{\alpha_r}$, unique up to permutation, \sut $\pi$ is a constituent of the normalized parabolic induction
\begin{equation*}
\bx I_{P_\bmin}^{\SO(V)}(\uno_3\boxtimes\udl{\alpha_2}\boxtimes\cdots\boxtimes\udl{\alpha_r}).
\end{equation*}
Note that there exists a unique nonnegative integer $r_0$ and a unique sequence of positive integers $r_1, \ldots, r_t$ satisfying $r_0+\cdots+r_t=r$, \sut $\largel{\alpha_i}=1$ for $2\le i\le r_0$, 
and
\begin{equation*}
1<\largel{\alpha_{r_0+1}}=\cdots=\largel{\alpha_{r_0+r_1}}<\largel{\alpha_{r_0+r_1+1}}=\cdots<\largel{\alpha_{r_0+\cdots+r_{t-1}+1}}=\cdots=\largel{\alpha_t}.
\end{equation*}
For each $1\le i\le t$, let $\tau_i$ be the irreducible tempered unramified representation of $\GL_{r_i}(K)$ with Satake parameter
\begin{equation*}
\Brace{\frac{\alpha_{r_0+\cdots+r_{i-1}+1}}{\largel{\alpha_{r_0+\cdots+r_i}}}, \ldots, \frac{\alpha_{r_0+\cdots+r_i}}{\largel{\alpha_{r_0+\cdots+r_i}}}}.
\end{equation*}
Let $P_{0,\bmin}\defining \SO(V_{2r_0+1})\cap P_\bmin$ be a minimal parabolic subgroup of $\SO(V_{2r_0+1})$. As $\uno_3\boxtimes\udl{\alpha_2}\boxtimes\cdots\boxtimes\udl{\alpha_{r_0}}$ is a discrete series representation of $P_{0, \bmin}(K)$, the normalized parabolic induction
\begin{equation*}
\bx I_{P_{0, \bmin}}^{\SO(V_{2r_0+1})}(\uno_3\boxtimes\udl{\alpha_2}\boxtimes\cdots\boxtimes\udl{\alpha_{r_0}})
\end{equation*}
is a finite direct sum of irreducible tempered representations of $\SO(V_{2r_0+1})(K)$. Let $\tau_0$ be the unique direct summand with nonzero invariants under $\mdc K\cap\SO(V_{2r_0+1})(K)$. In particular, we obtain a Langlands quotient
\begin{equation*}
\bx J\defining\bx J_P^{\SO(V)}\paren{\tau_0\boxtimes\paren{\boxtimes_{i=1}^t\tau_i\paren{\udl{\largel{\alpha_{r_0+\cdots+r_i}}}\circ\det\nolimits_{r_i}}}},
\end{equation*}
where $P$ is a parabolic subgroup of $\SO(V)$ containing $P_0$ whose Levi quotient is isomorphic to $\SO(V_{2r_0+1})\times \GL_{r_1}\times\cdots\times\GL_{r_t}$.

We claim that $\bx J$ has nonzero invariants under $\mdc K$. Assuming this claim, then $\FL(\pi)$ is isomorphic to the unique irreducible quotient of
\begin{equation*}
\bx I_{P_\GL}^{\GL_{2r}}\paren{\paren{\boxtimes_{i=1}^t\tau_i^\vee\paren{\udl{\largel{\alpha^{-1}_{r_0+\cdots+r_i}}}\circ\det\nolimits_{r_i}}}\boxtimes\FL(\tau_0)\boxtimes\paren{\boxtimes_{i=1}^t\tau_i\paren{\udl{\largel{\alpha_{r_0+\cdots+r_i}}}\circ\det\nolimits_{r_i}}}}.
\end{equation*}
However, $\FL(\tau_0)$ is isomorphic to the normalized induction
\begin{equation*}
\bx I_{P_\GL\cap\GL_{2r_0}}^{\GL_{2r_0}}\paren{\udl{\alpha_{r_0}^{-1}}\boxtimes\cdots\boxtimes\udl{\alpha_{2}^{-1}}\boxtimes\St_2\boxtimes\udl{\alpha_2}\boxtimes\cdots\boxtimes\udl{\alpha_{r_0}}},
\end{equation*}
which is irreducible. So the assertion follows.

Now we prove the claim. Note that there is a canonical inclusion of $G(K)$-representations
\begin{equation*}
\bx I_P^{\SO(V)}\paren{\tau_0\boxtimes\paren{\boxtimes_{i=1}^t\tau_i\paren{\udl{\largel{\alpha_{r_0+\cdots+r_i}}}\circ\det\nolimits_{r_i}}}}\subset \bx I_{P_\bmin}^{\SO(\mbf V)}(\uno_3\boxtimes\udl{\alpha_2}\boxtimes\cdots\boxtimes\udl{\alpha_r}),
\end{equation*}
under which $\phi_{\bm\alpha}$ belongs to the former space by our choice of $\tau_0$. By \cite[Corollary~3.2]{Kon03}, the claim is equivalent to
\begin{equation*}
T_w\phi_{\bm\alpha}\ne 0,
\end{equation*}
where $w\in W$ is the element acting trivially on $V_{2r_0}$ and switching $e_{-(r_0+\cdots+r_{i-1}+j)}$ with $e_{r_0+\cdots+r_i+1-j}$ for every $1\le i\le t$ and every $1\le j\le r_i$; and $T_w$ is the intertwining operator \sut 
\begin{equation*}
T_w\phi_{\bm\alpha}(g)=\int_{N(K)}\phi_{\bm\alpha}(w^{-1}ng)\bx dn, \quad g\in \SO(V)(K).
\end{equation*}
Here $N$ is the nilpotent radical of $P$ and the integral is absolutely convergent by \cite[Proposition~\Rmnum{4}.2.1]{Wal03}.

By \cite[Theorem~3.1]{Cas80}, we have $T_w\phi_{\bm\alpha}=c(\bm\alpha)\phi_{w\bm\alpha}$, where $c(\bm\alpha)$ is the Harish-Chandra $c$-function. Note that the regularity condition in \cite[\S3]{Cas80} can be dropped by continuity of intertwining operators. Under suitable normalization of Haar measure, $c(\bm\alpha)$ is given by the Gindikin--Karpelevich formula:
\begin{equation*}
c(\bm\alpha)=\prod_{i=r_0+1}^r\bigg(\frac{\alpha_i-\#\kappa^{-1}}{\alpha_i-1}\prod_{\substack{1\le j\le i-1\\ \largel{\alpha_j}<\largel{\alpha_i}}}\frac{\alpha_i-\#\kappa^{-1}\alpha_j}{\alpha_i-\alpha_j}\prod_{j=1}^{i-1}\frac{\alpha_i\alpha_j-\#\kappa^{-1}}{\alpha_i\alpha_j-1}\bigg).
\end{equation*}
In particular, $c(\bm\alpha)$ is nonzero, and the claim follows.
\end{proof}
\begin{cor}\label{pssienifienfiwwss}
Suppose $\pi$ is a tempered irreducible admissible representation of $\SO(V)(K)$ 
with nonzero invariants under a special maximal compact open subgroup of $G(K)$, then the monodromy operator of $\WD(\FL(\pi))$ is conjugate to $\begin{bmatrix}1 & 1\\ 0 & 1\end{bmatrix}\oplus \uno_{2m-2}$.
\end{cor}
\begin{proof}
It follows from Proposition~\ref{oieiticimcimdws} that the cuspidal support of $\pi$ is of the form $((K^\times)^{m-1}\times\SO(V_3)(K), \chi\boxtimes\uno)$, where $V_3$ is the anisotropic quadratic space of rank $3$ over $K$ and $\chi$ is an unramified character of $(K^\times)^{m-1}$. We also know that $\chi$ is unitary because $\pi$ is tempered. In particular, $\FL(\pi)$ is the (irreducible) parabolic induction from the Levi
\begin{equation*}
\GL_1^{m-1}\times\SO(V_3)(K)\times\GL_1^{m-1}
\end{equation*}
with inducing data $\chi\boxtimes\St_2\boxtimes\chi^{-1}$, where $\St_2$ is the Steinberg representation of $\GL_2(K)$. 
The corollary follows immediately.
\end{proof}

\section{Deformation}

In this section, we work in the following setting.

\begin{setup}\label{seutpeoeifmeis}\enskip
\begin{itemize}
\item
Let $\ell$ be an odd rational prime. 
\item
Let $E$ be a subfield of $\ovl{\bb Q_\ell}$ finite over $\bb Q_\ell$. We denote by $\mcl O$ the ring of integers of $E$, by $\lbd$ the maximal ideal, and by $\kappa$ the residue field.
\item
Denote by $\mrs C_{\mcl O}$ the category of Noetherian complete local rings over $\mcl O$ with residue field $\kappa$, with morphisms local homomorphisms inducing the identity on $\kappa$. Denote by $\mrs C_{\mcl O}^f$ the full subcategory of Artinian local rings over $\mcl O$.
\item
For each object $\bx R\in \mrs C_{\mcl O}$, we denote by $\mfk m_{\bx R}$ the maximal ideal of $\bx R$.
\item
Let $G$ be the smooth group scheme $\GSp_{2m}$ (resp. $\GO_{2m}$) over $\mcl O$ if $\mfk d=0$ (resp. $\mfk d\ne 0$). Here $\GSp_{2m}$ is defined by the standard alternating pairing on $\mcl O^{\oplus(2m)}$ given by the matrix $J'_{2m}$,
and $\GO_{2m}$ is defined by the standard symmetric pairing given by the matrix $J_{2m}$.
\item
We denote by $\mfk g$ the Lie algebra of $G$, by $G^\circ$ the identity component of $G$, by $G'$ the derived subgroup of $G^\circ$, by $\mfk g'$ the Lie algebra of $G'$, by $Z_G$ the center of $G$.
\item
We denote by $\Std$ the standard representation of $G$ over $\mcl O$ of dimension $2m$, and by $\nu: G\to \GL_1$ the similitude character.
\item
For any smooth group scheme $H$ over $\mcl O$, we denote by $H^\wedge$ the formal completion of $H$ along the unit section. In particular, for each object $\bx R$ in $\mrs C_{\mcl O}$, we have $H^\wedge(\bx R)=\ker(H(\bx R)\to H(\kappa))$.
\item
Let $\Gamma$ be a profinite group satisfying the following $\ell$-finiteness property: for any open subgroup $\Gamma_0\le \Gamma$, here are only finitely many continuous homomorphisms from $\Gamma_0$ to $\bb Z/\ell\bb Z$. (Note that this is true for the absolute Galois group of a local field and for the Galois group of the maximal extension of a number field unramified outside a finite set of places).
\end{itemize}
\end{setup}

\subsection{Some algebras}

In this subsection, we carry out some algebra computations that will be used in the deformation theory.

\begin{lm}\label{psooituries}
Let $\bx R$ be an object of $\mrs C_{\mcl O}$ and let $\rho: \Gamma\to G(\bx R)$ is a homomorphism. Suppose $\Std\circ\rho\modu{\mfk m_{\bx R}}$ is absolutely irreducible, then the centralizer of $\rho$ in $(G')^\wedge(\bx R)$ is trivial.
\end{lm}
\begin{proof}
The lemma is easily reduced to the case when $\bx R$ is Artinian. We then use induction on the length of $\bx R$. The case of length one is immediate. In general we may choose an ideal $\bx I$ of $\bx R$ \sut $\bx I$ has length one. By the induction hypothesis, the centralizer of $\rho$ in $(G')^\wedge(\bx R)$ lies in $G'(1+\bx I)$.

Then it suffices to show that $\bx H^0(\Gamma, \ad^0(\rho)\otimes_{\bx R}\kappa')=0$ for any field $\kappa'$ containing $\kappa$. But this follows from Schur's lemma and that $\ell$ is odd.
\end{proof}

\begin{lm}[Carayol]\label{ososiieiemifes}
Let $\bx R'\subset \bx R$ be Noetherian complete local rings with maximal ideals $\mfk m_{\bx R'}$ and $\mfk m_{\bx R}$, respectively. Suppose that $\bx R'$ and $\bx R$ has the same residue field, and $\mfk m_{\bx R}\cap\bx R'=\mfk m_{\bx R'}$. Let $\Gamma$ be a profinite group and let $\rho: \Delta\to \GL_{2m}(\bx R)$ is a continuous homomorphism, \sut $\rho\modu{\mfk m_{\bx R}}$ is absolutely irreducible and that $\tr(\rho)(\Gamma)\subset \bx R'$. Then there exists a $\GL_{2m}^\wedge(\bx R)$-conjugate $\rho'$ of $\rho$ valued in $\GL_{2m}(\bx R')$.
\end{lm}
\begin{proof}
This is \cite[Lemma~2.1.10]{CHT08}; cf. \cite[Th\'eor\`eme~2]{Car94}.
\end{proof}

We also need to the following version of Carayol's lemma for $G$-valued Galois representations.

\begin{lm}\label{osoosieiwumfiev}
Let $\bx R'\subset \bx R$ be Noetherian complete local rings with maximal ideals $\mfk m_{\bx R'}$ and $\mfk m_{\bx R}$, respectively. Suppose that $\bx R'$ and $\bx R$ has the same residue field $\kappa$ with characteristic $\ell>2$, and $\mfk m_{\bx R}\cap\bx R'=\mfk m_{\bx R'}$. Let $\Gamma$ be a profinite group and let $\rho: \Gamma\to G(\bx R)$ be a continuous homomorphism, \sut $\Std\circ\rho\modu{\mfk m_{\bx R}}$ is absolutely irreducible and that $\tr(\Std\circ\rho)(\Gamma)\subset\bx R', \nu\circ\rho(\Gamma)\subset\bx R'$. Then $\rho$ is $(G')^\wedge(\bx R)$-conjugate to a homomorphism $\rho': \Gamma\to G(\bx R')$.
\end{lm}
\begin{proof}
It follows from Lemma~\ref{ososiieiemifes} that $\Std\circ\rho$ is $\GL_{2m}^\wedge(\bx R)$-conjugate to a representation $\rho'$ valued in $\GL_{2m}(\bx R')$. Because $\rho'$ and $\rho^{\prime, \top}\otimes(\nu\circ\rho)$ are both valued in $\GL_{2m}(\bx R')$ and they are conjugate in $\GL_{2m}(\bx R)$, they are also conjugate in $\GL_{2m}(\bx R')$ by \cite[Th\'eor\`eme~1]{Car94}. Suppose $B$ is an element of $\GL_{2m}(\bx R')$ \sut $\rho'=B\rho^{\prime, \top}\otimes(\nu\circ\rho)B^{-1}$. Then it follows from Schur's lemma that $B$ is symmetric (resp. skew-symmetric) if $N$ is even (resp. $N$ is odd), and $B$ defines a general orthogonal (resp. general symplectic) group isomorphic to $G$. Thus we see that $\rho$ is $\GL_{2m}(\bx R)$-conjugate to a representation valued in $G(\bx R')$. It is also easy to check that one may choose the conjugating matrix to be in $\GL_{2m}^\wedge(\bx R)$. Finally, it follows from Schur's lemma that the conjugating matrix is also in $G(\bx R)$. After rescaling, we may choose the conjugating matrix to be in $(G')^\wedge(\bx R)$. 
\end{proof}

\subsection{Deformation problems}\label{oeiimvimdi}

In this subsection, we recall the notion of deformation problems. We work in the following setting.

\begin{setup}\label{osisienieeimifes}
Suppose $(\Gamma, \ovl\rho, \chi)$ is a triple in which
\begin{itemize}
\item
$\Gamma$ is our fixed topological group,
\item
$\ovl\rho: \Gamma\to G(\kappa)$ is a homomorphism, and
\item
$\chi: \Gamma\to \mcl O^\times$ is a continuous homomorphism with reduction $\ovl\chi: \Gamma\to \kappa^\times$,
\end{itemize}
\sut $\nu\circ \ovl\rho=\ovl\chi$. Let $\ad(\ovl\rho)$ (respectively, $\ad^0(\ovl\rho)$) denote the representation of $\Gamma$ on $\mfk g_\kappa$ (respectively, $\mfk g'_\kappa$) via the adjoint representation.
\end{setup}

\begin{defi}\label{liftings-anadi-defintiieos}
Let $f: \bx R_1\to \bx R_0$ be a morphism in $\mrs C_{\mcl O}$ and $\rho_0: \Gamma\to G(\bx R_0)$ a continuous homomorphism. A \emph{lifting} of $\rho_0$ to $\bx R_1$ (\wrt $\chi$) is a continuous homomorphism $\rho_1: \Gamma\to G(\bx R_1)$ \sut $f_*(\rho_1)=\rho_0$ and $\nu_{\bx R_1}\circ\rho_1=\chi$. Let
\begin{equation*}
\msf D^\square_{\ovl\rho, \mcl O}: \mrs C_{\mcl O}\to \Set
\end{equation*}
be the functor sending each object $\bx R\in \mrs C_{\mcl O}$ to the set of liftings of $\ovl\rho$ to $\bx R$.

For each $\bx R\in \mrs C_{\mcl O}$, two liftings $\rho, \rho'$ of $\ovl\rho$ are called \emph{equivalent} if they are conjugate by an element of $(G')^\wedge(\bx R)$. A \emph{deformation} of $\ovl\rho$ to $\bx R$ is defined to be an equivalence class of deformations of $\ovl\rho$ to $\bx R$. Let
\begin{equation*}
\msf D_{\ovl\rho, \mcl O}: \mrs C_{\mcl O}\to \Set
\end{equation*}
be the functor sending each object $\bx R\in \mrs C_{\mcl O}$ to the set of deformations of $\ovl\rho$ to $\bx R$.

$\msf D^\square_{\ovl\rho, \mcl O}$ is representable by an element $\msf R_{\ovl\rho}^\square\in \mrs C_{\mcl O}$, called the \emph{universal lifting ring}; see \cite[Theorem~1.2.2]{Bal12}. 
\end{defi}

The tangent space of $\msf D_{\ovl\rho, \mcl O}$ is naturally identified with $\bx H^1(\Gamma, \ad^0(\ovl\rho))$. 

\begin{defi}\label{oeieriwowpq}
A \emph{local lifting condition} of $\ovl\rho$ (\wrt $\chi$) is a closed formal subscheme $\mrs D^\square$ of $\Spf\msf R_{\ovl\rho}^\square$ satisfying that $\mrs D^\square(\bx R)$ is closed under equivalence for every $\bx R\in\mrs C_{\mcl O}$.

For each local lifting condition $\mrs D^\square$, we denote by $\mrs I^\square\subset \msf R_{\ovl\rho}^\square$ the closed ideal defining $\mrs D^\square$. Note that $\mrs D^\square$ naturally induces a \emph{local deformation problem} denoted by $\mrs D$, which is a subfunctor of $\mrs D_{\ovl\rho, \mcl O}$.
\end{defi}

\begin{defi}\label{ososeiifneietkeures}
For each \emph{local lifting condition} $\mrs D^\square$ of $\ovl\rho$ (\wrt $\chi$) with the associated local deformation problem $\mrs D$, there is an isomorphism
\begin{equation*}
\Hom_\kappa\big(\mfk m_{\msf R^\square_{\ovl\rho}}/(\mfk m^2_{\msf R^\square_{\ovl\rho}},\lbd), \kappa\big)\cong \Hom_{\mrs C_{\mcl O}}(\msf R^\square_{\ovl\rho}, \kappa[\ve]/(\ve^2))\cong \bx Z^1(\Gamma, \ad^0(\ovl\rho)),
\end{equation*}
where $\bx Z^1(\Gamma, \ad^0(\ovl\rho))$ denotes the group of $1$-cocycles of $\Gamma$ with values in the adjoint representation $\ad^0(\ovl\rho)$ of $\Gamma$ on $\mfk g'_\kappa$.

We define the \emph{tangent space} of $\mrs D$, denoted by $\bx L(\mrs D)$, to be the image of the subspace
\begin{equation*}
\bx L^1(\mrs D^\square)\defining\Hom_\kappa\big(\mfk m_{\msf R^\square_{\ovl\rho}}/(\mfk m^2_{\msf R^\square_{\ovl\rho}}, \mrs I,\lbd), \kappa\big)\subset\bx Z^1(\Gamma, \ad^0(\ovl\rho))
\end{equation*}
under the natural map $\bx Z^1(\Gamma, \ad^0(\ovl\rho))\to \bx H^1(\Gamma, \ad^0(\ovl\rho))$. Here $\mrs I\subset\msf R^\square_{\ovl\rho}$ is the closed ideal defining $\mrs D^\square$.

Because $\mrs I$ is invariant under conjugation by $(G')^\wedge(\bx R)$, we see that $\bx L^1(\mrs D^\square)$ is the preimage of $\bx L(\mrs D)$ in $\bx Z^1(\Gamma, \ad^0(\ovl\rho))$. In particular,
\begin{equation}\label{psisieubifes}
\dim_\kappa\bx L^1(\mrs D^\square)-\dim_\kappa\bx L(\mrs D)=\dim G'_\kappa-\dim_\kappa\bx H^0(\Gamma, \ad^0(\ovl\rho)).
\end{equation}
\end{defi}

\begin{defi}\label{defomrioainoiereid}
Suppose $\Gamma=\Gal_F$. A \emph{global deformation problem} of $\ovl\rho$ (\wrt  $\chi$) is a tuple $(\ovl\rho, \chi, \Pla, \{\mrs D_v^\square\}_{v\in\Pla})$ in which
\begin{itemize}
\item
$\Pla$ is a finite set of finite places of $F$ containing $\Pla(\ell)$ and all places $v$ \sut $\ovl\rho_v$ is ramified.
\item
$\mrs D_v^\square$ is a local lifting condition of $\ovl\rho$ \wrt $\chi_v$ for each $v\in\Pla$.
\end{itemize}
\end{defi}

For the rest of this subsection, we let $\Gamma=\Gal_F$ and fix a global deformation problem $\mrs S=(\ovl\rho, \chi, \Pla, \{\mrs D_v^\square\}_{v\in\Pla})$. For each subset $\Xi\subset\Pla$, set
\begin{equation*}
\msf R_{\mrs S, \Xi}^\square\defining\hat\btimes_{v\in\Xi}(\msf R_{\ovl\rho_v}^\square/\mrs I^\square_v),
\end{equation*}
where the tensor product is taken over $\mcl O$. We recall from \cite[Definition~2.2.1]{CHT08} the notion of framed liftings and framed deformations.

\begin{defi}
For each subset $\Xi\subset\Pla$, a \emph{$\Xi$-framed lifting} of $\ovl\rho$ to an object $\bx R\in \mrs C_{\mcl O}$ is a tuple $(\rho; \{\beta_v\}_{v\in\Xi})$, in which $\rho$ is a lifting of $\ovl\rho$ to $\bx R$ \wrt $\chi$ (see Definition~\ref{liftings-anadi-defintiieos}), and $\beta_v\in (G')^\wedge(\bx R)$ for every $v\in\Xi$. We say that a $\Xi$-framed lifting $(\rho, \{\beta_v\}_{v\in\Xi})$ is of type $\mrs S$ if $\rho$ is unramified outside $\Pla$ and $\rho_v\in \mrs D_v^\square(\bx R)$ for every $v\in\Pla$.

Two $\Xi$-framed liftings $(\rho; \{\beta_v\}_{v\in\Xi})$ and $(\rho'; \{\beta'_v\}_{v\in\Xi})$ are said to be equivalent if there exists $g\in (G')^\wedge(\bx R)$ \sut $\rho'=g\circ\rho\circ g^{-1}$ and $\beta_v'=g\beta_v$ for every $v\in\Xi$. A \emph{$\Xi$-framed deformation} of $\ovl\rho$ is an equivalence class of $\Xi$-framed liftings of $\ovl\rho$.

If a $\Xi$-framed lifting is of type $\mrs S$, so is any equivalent $\Xi$-framed lifting. We say that a $\Xi$-framed deformation is of type $\mrs S$ if some (or equivalently, every) element is of type $\mrs S$. Let
\begin{equation*}
\Def^{\square_\Xi}_{\mrs S}: \mrs C_{\mcl O}\to \Set
\end{equation*}
be the functor sending each object $\bx R\in \mrs C_{\mcl O}$ to the set of $\Xi$-framed deformations of $\ovl\rho$ to $\bx R$ of type $\mrs S$.
\end{defi}

Recall that $\Gal_{F, \Pla}$ is the Galois group of the maximal subextension of $\ovl F/F$ that is unramified outside $\Pla\cup\Pla^\infty$. Following \cite[p.~21]{CHT08}, for each integer $i\in\bb N$, we denote by $\bx H^i_{\mrs S, \Xi}(\Gal_{F,\Pla}, \ad^0(\ovl\rho))$ the cohomology of the complex
\begin{equation*}
\bx C^i_{\mrs S, \Xi}(\Gal_{F,\Pla}, \ad^0(\ovl\rho))\defining\bx C^i(\Gal_{F,\Pla}, \ad^0(\ovl\rho))\oplus\bplus_{v\in \Pla}\frac{\bx C^{i-1}(F_v, \ad^0(\ovl\rho))}{\bx M_v^{i-1}}, 
\end{equation*}
where $\bx M_v^{i-1}=0$ unless $v\in\Pla\setm\Xi$ and $i=0$ in which case 
\begin{equation*}
\bx M_v^i=
\begin{cases}
\bx C^0(F_v,\ad^0(\ovl\rho)) &\If i=0\\
\bx L^1(\mrs D_v^\square) &\If i=1
\end{cases}.
\end{equation*}
The coboundary map is
\begin{equation*}
\bx C^i_{\mrs S, \Xi}(\Gal_{F, \Pla}, \ad^0(\ovl\rho))\to \bx C^{i+1}_{\mrs S, \Xi}(\Gal_{F, \Pla}, \ad^0(\ovl\rho)): (\phi, (\psi_v)_{v\in\Pla})\mapsto (\partial\phi, (\phi|_{F_v}-\partial\psi_v)_{v\in\Pla}).
\end{equation*}
Then there is an exact sequence
\begin{align*}
0&\to \bx H^0_{\mrs S, \Xi}(\Gal_{F,\Pla}, \ad^0(\ovl\rho))\to \bx H^0(\Gal_{F,\Pla}, \ad^0(\ovl\rho))\to \bplus_{v\in\Xi}\bx H^0(F_v, \ad^0(\ovl\rho))\\
&\to \bx H^1_{\mrs S, \Xi}(\Gal_{F,\Pla}, \ad^0(\ovl\rho))\to\bx H^1(\Gal_{F,\Pla}, \ad^0(\ovl\rho))\to \bigg(\bplus_{v\in\Pla\setm\Xi}\frac{\bx H^1(F_v, \ad^0(\ovl\rho))}{\bx L(\mrs D_v)}\bigg)\oplus\bplus_{v\in\Xi}\bx H^1(F_v, \ad^0(\ovl\rho))\\
&\to \bx H^2_{\mrs S, \Xi}(\Gal_{F,\Pla}, \ad^0(\ovl\rho))\to \bx H^2(\Gal_{F,\Pla}, \ad^0(\ovl\rho))\to \bplus_{v\in\Pla}\bx H^2(F_v, \ad^0(\ovl\rho))\\
&\to \bx H^3_{\mrs S, \Xi}(\Gal_{F,\Pla}, \ad^0(\ovl\rho))\to \bx H^3(\Gal_{F,\Pla}, \ad^0(\ovl\rho))\to\ldots\\
\end{align*}

Let $\bx L(\mrs D_v)^\perp\subset\bx H^1(F_v, \ad^0(\ovl\rho)(1))$ be the annihilator of $\bx L(\mrs D_v)$ under the local Tate duality induced from the Galois invariant perfect pairing
\begin{equation*}
\ad^0(\ovl\rho)\times\ad^0(\ovl\rho)\to \kappa: (x, y)\mapsto \tr(xy).
\end{equation*}
We define
\begin{equation}\label{psosiieuueufemes}
\bx H^1_{\mrs L^\perp, \Xi}(\Gal_{F, \Pla}, \ad^0(\ovl\rho)(1))\defining\ker\bigg(\bx H^1(\Gal_{F, \Pla}, \ad^0(\ovl\rho)(1))\to\prod_{v\in\Pla\setm\Xi}\frac{\bx H^1(F_v, \ad^0(\ovl\rho)(1))}{\bx L(\mrs D_v)^\perp}\bigg).
\end{equation}

\begin{lm}\label{sliimiifhiemiws}
The following statements hold:
\begin{enumerate}
\item
$\bx H^i_{\mrs S, \Xi}(\Gal_{F,\Pla}, \ad^0(\ovl\rho))$ vanishes for $i>3$.
\item
$\bx H^0_{\mrs S, \Xi}(\Gal_{F,\Pla}, \ad^0(\ovl\rho))=\bx H^0(\Gal_{F,\Pla}, \ad^0(\ovl\rho))$ if $\Xi=\vn$ and vanishes if $\Xi\ne \vn$.
\item
$\dim_\kappa\bx H^3_{\mrs S, \Xi}(\Gal_{F,\Pla}, \ad^0(\ovl\rho))=\dim_\kappa\bx H^0(\Gal_{F,\Pla}, \ad^0(\ovl\rho)(1))$.
\item
$\dim_\kappa\bx H^2_{\mrs S, \Xi}(\Gal_{F,\Pla}, \ad^0(\ovl\rho))=\dim_\kappa\bx H^1_{\mrs L^\perp, \Xi}(\Gal_{F,\Pla}, \ad^0(\ovl\rho)(1))$.
\item(Wiles' formula)
\begin{align*}
\dim_\kappa\bx H^1_{\mrs S, \Xi}(\Gal_{F,\Pla}, \ad^0(\ovl\rho))&=\delta_{\Xi, \vn}\cdot\dim_\kappa\bx H^0(\Gal_{F,\Pla}, \ad^0(\ovl\rho))+\dim_\kappa\bx H^1_{\mrs L^\perp, \Xi}(\Gal_{F,\Pla}, \ad^0(\ovl\rho)(1))\\
&-\dim_\kappa\bx H^0(\Gal_{F,\Pla}, \ad^0(\ovl\rho)(1))-[F: \bb Q](\dim G_\kappa-\dim B_\kappa)\\
&+\sum_{v\in\Pla\setm\Xi}(\dim_\kappa\bx L(\mrs D_v)-\dim_v\bx H^0(F_v, \ad^0(\ovl\rho))).
\end{align*}
\end{enumerate}
\end{lm}
\begin{proof}
Using the $\ell$-finiteness property, the proof is the same as that of \cite[Lemma~2.3.4]{CHT08}. Here we use that $\bx H^0(\Gal_{F_v}, \ad^0(\mfk g))=m(m+\delta_{\mfk d, 0}-1)=(\dim G_\kappa- \dim B_\kappa)$ for every real place $v$ of $F$. This formula is a corollary of the following claim: Let $\bra{-, -}$ be the nondegenerate symmetric (resp. alternating) pairing on $\kappa^{\oplus 2m}$ defined by $J_{2m}$ (resp. $J_{2m}'$) when $N=2m$ (resp. $N=2m+1$). (In particular, $G_\kappa$ is the subgroup of elements of $\GL_{2m}$ that fix $\bra{-, -}$ up to a scalar.) Then $\ovl\rho(\cc_v)$ is $G(\kappa)$-conjugate to the matrix
\begin{equation*}
\begin{cases}
\begin{bmatrix}
\uno_m & \\
&-\uno_m
\end{bmatrix} & \mfk d=0,\\
J_{2m}  & \mfk d\ne 0,
\end{cases}
\end{equation*}
for any complex conjugation $\cc_v$ at $v$. 

We now prove this claim: We write $s=\rho_{\Pi, \iota_\ell}(\cc_v)$. We first treat the case when $N=2m+1$. It follows from Proposition~\ref{associateGlaosireifnies}(2) that $\nu(s)=-1$. As $s^2=1$, $s$ has eigenvalues $\pm1$. Let $V^\pm$ be the eigenspace of $\pm1$ of $s$. Then
\begin{equation*}
\bra{v,w}=\bra{sv, sw}=-\bra{v, w}
\end{equation*}
for $v, w\in V^+$. Thus $V^+$ is totally isotropic. The same argument shows that $V^-$ is totally isotropic. Thus we can take a basis $\{x_1, \ldots, x_m\}$ of $V^+$ and a basis $\{y_1, \ldots, y_m\}$ of $V^-$, \sut $\bra{x_i, y_i}=\delta_{i,j}$. As a consequence, in the basis $\{x_1, \ldots, x_m, y_1, \ldots, y_m\}$, $s$ is represented by the matrix
\begin{equation*}
\begin{bmatrix}
\uno_m & \\
&-\uno_m
\end{bmatrix}.
\end{equation*}
The claim follows.

In the setting when $N=2m$, we have $\nu(s)=1$. Let $V^\pm$ be the eigenspace of $\pm1$ of $s$. Then $V^+$ is orthogonal to $V^-$. We also know that $\dim V^+=\dim V^-=m$ by \cite[Theorem~1.1]{C-H16}. The claim then follows easily.
\end{proof}

\begin{prop}\label{peoieutnmeumes}
Assume that $\Std\circ\ovl\rho$ is absolutely irreducible. Then for each subset $\Xi\subset\Pla$, the functor $\Def_{\mrs S}^{\square_\Xi}$ is represented by an object $\msf R_{\mrs S}^{\square_\Xi}\in\mrs C_{\mcl O}$. Set $\msf R_{\mrs S}\defining \msf R_{\mrs S}^{\square_\vn}$, called the \emph{global deformation ring} (\wrt $\mrs S$). Note that the natural transformation
\begin{equation*}
\Def_{\mrs S}^{\square_\Xi}\to \prod_{v\in\Xi}\msf D^\square_{\ovl\rho_v, \mcl O}: (\rho, \{\beta_v\}_{v\in\Xi})\mapsto (\beta_v^{-1}\rho_v\beta_v)_{v\in\Xi}
\end{equation*}
induces a tautological morphism $\msf R^\square_{\mrs S, \Xi}\to \msf R^{\square_\Xi}_{\mrs S}$ in $\mrs C_{\mcl O}$.
\begin{enumerate}
\item
There is a canonical isomorphism
\begin{equation*}
\Hom_\kappa\big(\mfk m_{\msf R_{\mrs S}^{\square_\Xi}}/\big(\mfk m_{\msf R_{\mrs S}^{\square_\Xi}}^2, \lbd, \mfk m_{\msf R_{\mrs S, \Xi}^\square}\big), \kappa\big)\cong \bx H^1_{\mrs S, \Xi}(\Gamma_{F, \Pla}, \ad^0(\ovl\rho)),
\end{equation*}
where we regard $\mfk m_{\msf R^\square_{\mrs S, \Xi}}$ as its image under the tautological homomorphism $\msf R^\square_{\mrs S, \Xi}\to \msf R^{\square_\Xi}_{\mrs S}$.
\item
If $\bx H^2_{\mrs S, \Xi}(\Gamma_{F, \Pla}, \ad^0(\ovl\rho))=0$ and $\mrs D_v$ is formally smooth over $\mcl O$ for each $v\in \Pla\setm\Xi$, then $\msf R^{\square_\Xi}_{\mrs S}$ is a power series ring over $\msf R^\square_{\mrs S, \Xi}$ in $\dim_\kappa\bx H^1_{\mrs S, \Xi}(\Gamma_{F, \Pla}, \ad^0(\ovl\rho))$ variables.
\item
Choosing a universal lifting $\rho_{\mrs S}: \Gamma_F\to G(\msf R_{\mrs S})$ of $\ovl\rho$ to $\msf R_{\mrs S}$ determines an extension of the tautological homomorphism $\msf R_{\mrs S}\to \msf R^{\square_\Xi}_{\mrs S}$ to an isomorphism
\begin{equation*}
\msf R_{\mrs S}\hat\otimes_{\mcl O}\mrs T_\Xi\cong \msf R_{\mrs S}^{\square_\Xi},
\end{equation*}
where $\mrs T_\Xi$ is the coordinate ring of the formal smooth group $\prod_{v\in \Xi}(G')^\wedge$ over $\mcl O$. Note that $\mrs T_\Xi$ is isomorphic to a formal power series algebra over $\mcl O$ in $\#\Xi\cdot\dim G'_\kappa$ variables.
\end{enumerate}
\end{prop}
\begin{proof}
Note that $\bx H^0(\Gal_F, \ad^0(\ovl\rho))=0$ since $\Std\circ\ovl\rho$ is absolutely irreducible. Thus the representability of $\msf R_{\mrs S}^{\square_\Xi}$ can be proved the in the same way as the proof of \cite[Proposition~9.2(2)]{Pat16}.

For part~(1), 
recall that 
\begin{equation*}
\Hom_\kappa\big(\mfk m_{\msf R_{\mrs S}^{\square_\Xi}}/\big(\mfk m_{\msf R_{\mrs S}^{\square_\Xi}}^2, \lbd, \mfk m_{\msf R_{\mrs S, \Xi}^\square}\big), \kappa\big)\cong \Hom(\msf R_{\mrs S}^{\square_\Xi}/(\mfk m_{\msf R_{\mrs S, \Xi}^\square}), \kappa[\ve]/(\ve^2))
\end{equation*}
is isomorphic to the subspace of $\Def^{\square_\Xi}_{\mrs S}$ consisting of elements giving trivial liftings of $\ovl\rho_v$ for $v\in\Xi$. Any $\Xi$-framed lifting of $\ovl\rho$ is of the form
\begin{equation*}
((\uno+\phi\ve)\ovl\rho; \{\uno+a_v\ve\}_{v\in\Xi})
\end{equation*}
with $\phi\in \bx Z^1(\Gal_F, \ad^0(\ovl\rho))$. It is of type $\mrs S$ if and only if $\phi|_{\Gal_{F_v}}\in \bx L^1(\mrs D_v^\square)$ for every $v\in\Pla$. It gives rise to a trivial lifting of $\ovl\rho|_{\Gal_{F_v}}$ if and only if
\begin{equation*}
(\uno-a_v\ve)(\uno+\phi|_{\Gal_{F_v}}\ve)\ovl\rho(\uno+a_v\ve)=\ovl\rho_v,
\end{equation*}
or equivalently, $\phi|_{\Gal_{F_v}}=(1-\ad^0(\ovl\rho)|_{\Gal_{F_v}})a_v$. On the other hand, two tuples $(\phi; \{a_v\}_{v\in\Xi})$ and $(\phi'; \{a'_v\}_{v\in\Xi})$ are equivalent if and only if there exists $b\in \mfk g'_\kappa$ \sut
\begin{equation*}
\phi'=\phi+(\uno-\ad^0(\ovl\rho))b
\end{equation*}
and
\begin{equation*}
a_v'=a_v+b
\end{equation*}
for every $v\in\Xi$. The assertion clearly follows.

For part~(2), the proof of \cite[Corollary~2.2.13]{CHT08} goes through; cf.~\cite[Proposition~9.2]{Pat16}.

Part~(3) follows from Lemma~\ref{psooituries}.
\end{proof}

\subsection{Fontaine--Laffaille deformations}

In this subsection, we recall some results of the Fontaine--Laffaille lifting condition defined in \cite{Boo19}. Following \cites{B-K90, CHT08}, we use a covariant version of the Fontaine--Laffaille theory \cite{F-L82}. 

Assume that $\ell$ is unramified in $F$, and let $v$ be a finite place of $F$ lying above $\ell$. We denote by $\Fr\in\Gal(F_v/\bb Q_\ell)$ the absolute Frobenius element, and set $\Pla_v\defining\Hom_{\bb Z_\ell}(\mcl O_v, \ovl{\bb Q_\ell})$

\begin{assumption}\label{OApillsieiiurwvx}
The field $E$ contains the images of all embeddings of $F_v$ into $\ovl{\bb Q_\ell}$.
\end{assumption}

\begin{defi}
Suppose Assumption~\ref{OApillsieiiurwvx} holds. We denote by $\mrsMF_{\mcl O, v}$ the category of $\mcl O_v\otimes_{\bb Z_\ell}\mcl O$-modules $M$ of finite length equipped with
\begin{itemize}
\item
a decreasing filtration $\{\Fil^iM\}_{i\in\bb Z}$ by $\mcl O_v\otimes_{\bb Z}\mcl O$-submodules that are $\mcl O_v$-direct summands, satisfying $\Fil^0M=M$ and $\Fil^{\ell-1}M=0$; and
\item
a Frobenius structure, that is, a $\Fr\otimes1$-linear map $\Phi^i: \Fil^iM\to M$ for every $i\in\bb Z$, satisfying the relations $\Phi^i|_{\Fil^{i+1}M}=\ell\Phi^{i+1}$ and $\sum_{i\in\bb Z}\Phi^i\Fil^iM=M$.
\end{itemize}
\end{defi}

\begin{defi}
Suppose Assumption~\ref{OApillsieiiurwvx} holds. For each object $M$ of $\mrsMF_{\mcl O, v}$, the decomposition
\begin{equation*}
\mcl O_v\otimes_{\bb Z_\ell}\mcl O=\bplus_{\tau\in\Pla_v}\mcl O
\end{equation*}
induces canonical decompositions
\begin{equation*}
M=\bplus_{\tau\in\Pla_v}M_\tau, \quad M_\tau\defining M\otimes_{\mcl O_v\otimes_{\bb Z_\ell}\mcl O, \tau\otimes\uno}\mcl O
\end{equation*}
and
\begin{equation*}
\Fil^iM=\bplus_{\tau\in\Pla_v}\Fil^iM_\tau, \quad \Fil^iM_\tau=M_\tau\cap\Fil^iM.
\end{equation*}
Then $\Phi^i$ induces $\mcl O$-linear maps
\begin{equation*}
\Phi^i_\tau: \Fil^iM_\tau\to M_{\tau\circ\Fr^{-1}}.
\end{equation*}
for every $\tau\in\Pla_v$ and every $i\in\bb Z$. We set
\begin{equation*}
\gr^iM_\tau\defining \Fil^iM_\tau/\Fil^{i+1}M_\tau,
\end{equation*}
\begin{equation*}
\gr^\bullet M_\tau\defining \bplus_{i\in\bb Z}\gr^iM_\tau,
\end{equation*}
\begin{equation*}
\gr^\bullet M\defining\bplus_{\tau\in\Pla_v}\gr^\bullet M_\tau,
\end{equation*}
and define the set of $\tau$-Fontaine--Laffaille weights of $M$ to be
\begin{equation*}
\HT_\tau(M)\defining\{i\in\bb Z|\gr^iM_\tau\ne 0\}.
\end{equation*}
We say that $M$ has regular Fontaine--Laffaille weights if $\gr^iM_\tau$ is generated over $\mcl O$ by at most one element for every $\tau\in\Pla_v$ and every $i\in\bb Z$.
\end{defi}

\begin{exm}
Suppose Assumption~\ref{OApillsieiiurwvx} holds. Each object $\bx R$ of $\mrs C_{\mcl O}^f$ can be regarded as an object $\bx R\{0\}$ of $\mrsMF_{\mcl O, v}$ with
\begin{equation*}
\Fil^i(\bx R\{0\})=
\begin{cases}
\mcl O_v\otimes_{\bb Z_\ell}\bx R & i\le 0\\
0 &i>0
\end{cases}.
\end{equation*}
\end{exm}

We recall the notion of Tate twists on objects in $\mrsMF_{\mcl O, v}$.

\begin{defi}
Suppose Assumption~\ref{OApillsieiiurwvx} holds. For each object $M$ of $\mrsMF_{\mcl O, v}$ and $s\in\bb Z$, we define $M\{s\}$ to be the objects of $\mrsMF_{\mcl O, v}$ with $\Fil^iM\{s\}=\Fil^{i+s}M$ and $\Phi^i_{M\{s\}}=\Phi^{i+s}_M$. 
\end{defi}

For each integer $b$ satisfying $0\le b\le \ell-2$, we let $\mrsMF^{[0, b]}_{\mcl O, v}$ be the full subcategory of $\mrsMF_{\mcl O, v}$ consisting of objects $M$ satisfying $\Fil^{b+1}M=0$. 
Moreover, for each object $\bx R$ of $\mrs C_{\mcl O}^f$, we let $\mrsMF^{[0, b]}_{\mcl O, v}|_{\bx R}$ be the full subcategory of $\mrsMF^{[0, b]}_{\mcl O, v}$ consisting of objects $M$ \sut $M$ is of finite free over $\mcl O_v\otimes_{\bb Z_\ell}\bx R$ and that $\Fil^iM$ is a direct summand of $M$ as a $\mcl O_v\otimes_{\bb Z_\ell}\bx R$-submodule for every $i\in\bb Z$. 

\begin{prop}\label{ososoiinfiieifmis}
Suppose \textup{Assumption~\ref{OApillsieiiurwvx}} holds. Let $\mcl O[\Gal_{F_v}]^\fl$ be the category of $\mcl O$-modules of finite length equipped with a continuous action of $\Gal_{F_v}$. There exists an exact fully faithful covariant $\mcl O$-linear functor
\begin{equation*}
\mbf T^\cris_v: \mrsMF_{\mcl O, v}\to \mcl O[\Gal_{F_v}]^\fl
\end{equation*}
satisfying that
\begin{enumerate}
\item
The essential image of $\mbf T^\cris_v$ is closed under taking subobjects and quotient objects.
\item
For each object $M$ of $\mrsMF_{\mcl O, v}$ and each $s\in\bb Z$, $\mbf T^\cris_v(M\{s\})$ is isomorphic to $\mbf T^\cris_v(M)(-s)$.
\item
For each object $M$ of $\mrsMF_{\mcl O, v}$ of finite length (as $\mcl O$-module), we have
\begin{equation*}
\length_{\mcl O}M=[F_v: \bb Q_\ell]\cdot\length_{\mcl O}\mbf T^\cris_v(M).
\end{equation*}
\item
For each object $\bx R$ of $\mrs C_{\mcl O}^f$ and each object $M$ of $\mrsMF_{\mcl O, v}|_{\bx R}$, $\mbf T^\cris_v(M)$ is finite free over $\bx R$.
\end{enumerate}
\end{prop}
\begin{proof}
Part~(1) follows from \cite[Fact~4.10]{Boo19}. Part~(2) follows from the definition; see \cite[Definition~4.9]{Boo19}. Part~(3) follows from \cite[Lemma~4.14]{Boo19}. Part~(4) follows from \cite[Lemma~4.15]{Boo19}. 
\end{proof}

\begin{defi}\label{otetotiiemifes}
Let $a\le b$ be integers \sut $b-a\le \ell-2$. For an object $\bx R$ of $\mrs C_{\mcl O}$ and a continuous representation $\rho: \Gal_{F_v}\to G(\bx R)$, we say that $\rho$ is \emph{crystalline with regular Fontaine--Laffaille weights in $[a, b]$} if there exists a finite unramified extension $E_\dagger$ of $E$ that splits the extension $F_v/\bb Q_\ell$, with ring of integers $\mcl O_\dagger$, \sut for every quotient $\bx R'$ of $\bx R$ in $\mrs C^f_{\mcl O_\dagger}$, $((\Std\circ\rho)(a)\otimes_{\mcl O}\mcl O_\dagger)\otimes_{\bx R}\bx R'$ is isomorphic to $\mbf T^\cris_v(M')$ for some object $M'$ of $\mrsMF_{\mcl O_\dagger, v}^{[0, b-a]}|_{\bx R'}$ with regular Fontaine--Laffaille weights.
\end{defi}

\begin{defi}\label{ooidnuefiems}
Suppose $(\Gamma, \ovl\rho, \chi)$ is a triple as in Setup~\ref{osisienieeimifes} in which $\Gamma=\Gal_{F_v}$ and $\chi=\ve_\ell^{2-N}$, \sut $\Std\circ\ovl\rho$ is crystalline with regular Fontaine--Laffaille weights in $[a, b]$ for some integers $a\le b$ satisfying $b-a\le\frac{\ell-2}{2}$. We define $\mrs D^{\FL, \square}$ to be the local lifting condition (\wrt the pair $(\ovl\rho, \chi)$) that classifies regular Fontaine--Laffaille crystallines liftings of $\ovl\rho$ (see Definition~\ref{otetotiiemifes}). It is a local lifting condition by \cite[Corollary~5.6]{Boo19}.
\end{defi}

\begin{prop}\label{olaososiefefjeimss}
Let $(\Gamma, \ovl\rho, \chi)$ be a triple as in \textup{Definition~\ref{ooidnuefiems}}. If $B_\kappa$ is a Borel subgroup of $G_\kappa$, then the local lifting condition $\mrs D^{\FL, \square}$ is formally smooth over $\Spf\mcl O$ of pure relative dimension
\begin{equation*}
\dim G'_\kappa+[F_v: \bb Q_\ell]\cdot(\dim G_\kappa-\dim B_\kappa).
\end{equation*}
In particular, it follows from \textup{Equation~\ref{psisieubifes}} that
\begin{equation*}
\dim_\kappa\bx L(\mrs D^\FL)=\dim_\kappa\bx H^0(F_v, \ad^0(\ovl\rho))+[F_v: \bb Q_\ell](\dim G_\kappa-\dim B_\kappa).
\end{equation*}
\end{prop}
\begin{proof}
By descent for formal smoothness (see~\cite[06CM]{Sta}), we may assume that \textup{Assumption~\ref{OApillsieiiurwvx}} holds. Then the assertion follows from \cite[Theorem~5.2]{Boo19}.\end{proof}

\subsection{Tame representations}

In this subsection, we recall some results on minimally ramified deformations of tamely ramified self-dual local Galois representations.

We fix a positive integer $q$ that is coprime to $\ell$, and note that all results in this subsection adapt verbatim to the case where $G$ is replaced by the smooth group scheme $\GO_{2m+1}$ defined by the standard symmetric pairing given by the matrix $J_{2m+1}$, cf.~\cite[\S5.2]{Boo19a}.

\begin{defi}\label{eleifnfmeimiswsv}
The \emph{$q$-tame group}, denoted by $\bx T_q$, is defined to be the semidirect product topological group $t^{\bb Z_\ell}\rtimes\phi_q^{\hat{\bb Z}}$, satisfying $\phi_qt\phi_q^{-1}=t^q$.
\end{defi}

We recall the following result on the universal deformations of the tame group $\bx T_q$ from \cite[Proposition~3.16]{LTXZZa}.

\begin{prop}\label{osoosiiieubfes}
Suppose $(\Gamma, \ovl\rho, \chi)$ is a triple as in \textup{Setup~\ref{osisienieeimifes}} in which $\Gamma=\bx T_q$. Then the universal lifting ring $\msf R^\square_{\ovl\rho}$ (see \textup{Definition~\ref{liftings-anadi-defintiieos}}) is a local complete intersection, flat and of pure relative dimension $\dim G'_\kappa$ over $\mcl O$.
\end{prop}

For the rest of this subsection, we assume $\ell>2m$. Denote by $\mcl N_{2m}$ (resp. $\mcl U_{2m}$) the closed subscheme of $\Mat_{2m}$ (resp. $\GL_{2m}$) defined by the equation $X^{2m}=0$ (resp. $(A-1)^{2m}=0$). For every object $\bx R$ of $\mrs C_{\mcl O}$, there is a truncated exponential map
\begin{equation*}
\exp: \mcl N_{2m}(\bx R)\to \mcl U_{2m}(\bx R): X\mapsto 1+X+\ldots+\frac{X^{2m-1}}{(2m-1)!},
\end{equation*}
which is a bijection. Its inverse is given by the truncated logarithmic map
\begin{equation*}
\log: \mcl U_{2m}(\bx R)\to \mcl N_{2m}(\bx R): A\mapsto -\sum_{i=1}^{2m-1}\frac{(1-A)^i}{i}.
\end{equation*}
Furthermore, for any $C\in \GL_{2m}(\bx R)$ and any positive integer $k$, we have
\begin{equation*}
\exp(CAC^{-1})=C\exp(A)C^{-1}, \quad \log(CAC^{-1})=C\log(A)C^{-1}
\end{equation*}
and
\begin{equation*}
\exp(kA)=\exp(A)^k, \quad \log(B^k)=k\log(B);
\end{equation*}
cf. \cite[Lemma~2.4]{Tay08}.

For the rest of this subsection, we suppose $(\Gamma, \ovl\rho, \chi)$ is a triple as in Setup~\ref{osisienieeimifes} in which $\Gamma=\bx T_q$. Then $\ovl A\defining \ovl\rho(t)$ is unipotent, as it has $\ell$-power torsion in the image of $\ovl\rho$ by the continuity of $\ovl\rho$. 
Set $\ovl X\defining \log A\in \mcl N_{2m}(\kappa)$. 

\begin{defi}
Let $\sigma\vdash 2m$ be a partition $2m=m_1+\ldots+m_r$ of $2m$. It is called \tbf{admissible} if those $m_i$ satisfying $m_i\equiv N\modu2$ appears for an even number of times. Let $\mfk P_{2m}$ denote the set of admissible partitions of $2m$.
\end{defi}

By the classification of nilpotent orbits of $G$, for each $L\in\{\ovl\kappa, \ovl{\bb Q_\ell}\}$, there exists canonical surjective maps $\pi: \mcl N_{2m}(L)\to \mfk P_{2m}$ \sut the fibers of $\pi$ are exactly the orbits in $\mcl N_{2m}(L)$ under the conjugate action of $G(L)$; see \cite[Theorem~1.6]{Jan04}. The corresponding orbit is the intersection of $\mfk g_L\subset \mfk{gl}_{2m, L}$ with the $\GL_{2m, L}$-orbit corresponding to that partition of $2m$.

We introduce the following convenient assumption.

\begin{assumption}\label{uansoisoifeifmes}
$\kappa$ contains a subfield $\kappa^\flat$ of degree two \sut $\Std\circ\ovl\rho$ is defined over $\kappa^\flat$.
\end{assumption}

By \cite[\S3.1]{Boo19a}, when Assumption~\ref{uansoisoifeifmes} holds, we obtain an element $X_0\in \mcl N_{2m}(\mcl O)$ \sut
\begin{itemize}
\item
$\pi(X_0)=\pi(\ovl X)$, where we regard $X_0$ as an element in $\mcl N_{2m}(\ovl{\bb Q_\ell})$;
\item
the centralizer $Z_G(X_0)\subset G$ is smooth over $\mcl O$; and
\item
there exists an element $\Phi\in G(\mcl O)$ \sut $\Ad_G(\Phi)X_0=qX_0$.
\end{itemize}
Following \cite[\S3.2]{Boo19a}, we obtain a functor
\begin{equation*}
\Nil_{\ovl X}: \mrs C_{\mcl O}\to \Set: \bx R\mapsto \{X\in\mfk g_{\bx R}: \Ad_G(g)(X_0)=X\text{ for some }g\in G^\wedge(\mcl O)\},
\end{equation*}
which is representable. It follows from \cite[Corollary~5.2]{Boo19a} that we obtain an exponential map
\begin{equation*}
\exp: \Nil_{\ovl X}\to G
\end{equation*}
\sut for any object $\bx R$ of $\mrs C_{\mcl O}$ and $g\in G^\wedge(\bx R), X\in \Nil_{\ovl X}(\bx R)$ and any positive integer $k$, we have
\begin{equation*}
\exp(kX)=\exp(X)^k, \quad \Ad(g)\exp(X)=\exp(\Ad_G(g)X).
\end{equation*}

\begin{defi}\label{psoosiiieufunnuenes}
Suppose Assumption~\ref{uansoisoifeifmes} holds. We say that a lifting $\rho$ of $\ovl\rho$ to an object $\bx R$ of $\mrs C_{\mcl O}$ is \emph{minimally ramified} if there exists an element $X\in \Nil_{\ovl X}(\bx R)$ \sut $\rho(t)=\exp(X)$.

Let $\mrs D^{\bmin, \square}$ be the local lifting condition of $\ovl\rho$ \wrt $\chi$ (see Definition~\ref{oeieriwowpq}) that classifies minimally ramified liftings of $\ovl\rho$.
\end{defi}

We recall the following result of Booher~\cite{Boo19a} concerning minimally ramified deformation problems.

\begin{prop}\label{ossoooieifmeies}
Suppose that $\ell>2m$ and \textup{Assumption~\ref{uansoisoifeifmes}} holds. Then the local lifting condition $\mrs D^{\bmin, \square}$ is formally smooth over $\Spf\mcl O$ of pure relative dimension $\dim G'_\kappa$. In particular, it follows from \textup{Equation~\eqref{psisieubifes}} that
\begin{equation*}
\dim_\kappa\bx L(\mrs D^\bmin)=\dim_\kappa\bx H^0(\bx T_q, \ad^0(\ovl\rho)).
\end{equation*}
\end{prop}
\begin{proof}
This follows from the proofs of \cite[Proposition~5.6, Corollary~5.8]{Boo19a}.
\end{proof}

We recall the following lemma on the decomposition of liftings of representations of tame groups that will be useful later.

\begin{lm}\label{poeoiemifemifws}
Let $(\ovl r, \ovl{\bx M})$ be an unramified representation of the tame group $\bx T_q$ over $\kappa$. Suppose that $\ovl{\bx M}$ admits a decomposition
\begin{equation}\label{osoiseitieres}
\ovl{\bx M}=\ovl{\bx M}_1\oplus\ldots\oplus \ovl{\bx M}_s
\end{equation}
stable under the action of $\ovl r(\phi_q)$ \sut the characteristic polynomials of $\ovl r(\phi_q)$ on $\ovl{\bx M}_i$ are pairwise coprime for $1\le i\le s$. Let $(r, \bx M)$ be a lifting of $(\ovl r, \ovl{\bx M})$ to an object $\bx R$ of $\mrs C_{\mcl O}$.
\begin{enumerate}
\item
There is a unique decomposition
\begin{equation*}
\bx M=\bx M_1\oplus\ldots\oplus \bx M_s
\end{equation*}
of free $\bx R$-modules lifting \eqref{osoiseitieres}, \sut each $\bx M_i$ is stable under the action of $\ovl r(\phi_q)$.
\item
Suppose $q$ is not an eigenvalue for the canonical action of $\phi_q$ on $\Hom_\kappa(\ovl{\bx M}_i, \ovl{\bx M}_j)$ for all $i\ne j$, then the decomposition in (1) is stable under the whole group $\bx T_q$.
\end{enumerate}
\end{lm}
\begin{proof}
For (1): Let $f(T)$ be the characteristic polynomial of $\rho(\phi_q)$ on $M$. Its reduction decomposes as $\ovl f(T)=\prod_{i=1}^rg_i(T)$, where $g_i(T)$ is the characteristic polynomial of $\ovl\rho(\phi_q)$ on $\bx M_i$ for every $1\le i\le r$. By hypothesis, these characteristic polynomials $g_i(T)$ are pairwise coprime. So it follows from Hensel's lemma that there exists a decomposition $f(T)=\prod_{i=1}^r\tld g_i(T)$ in which $\tld g_i(T)\in \bx R[T]$ are polynomials lifting $g_i(T)$ for every $1\le i\le r$. For each $1\le i\le r$, we define $\tld Q_i=\prod_{i\ne j}\tld P_i$, and choose $A_i, B_i\in \bx R[T]$ with $A_i\tld P_i+B_i\tld Q_i=1$. Then $e_i\defining B(\rho(\phi_q))\tld Q_i(\rho(\phi_q))$ are projectors on $\bx M$ satisfying $\sum_ie_i=1$. Then we finish by taking $\bx M_i\defining e_i\bx M$.

Part (2) is proved in \cite[Lemma~3.21]{LTXZZa}.
\end{proof}

\subsection{Minimally ramified representations}\label{sosofieifiehfemis}

In this subsection, we recall the notion of minimally ramified deformations.

Let $v$ be a finite place of $F$ that is not in $\Pla(\ell)$. Let $P_v$ be the maximal subgroup of $I_{F_v}$ of pro-order coprime to $\ell$, and set $\bx T_v\defining \Gal_{F_v}/P_v$, which is isomorphic to the $\norml{v}$-tame group $\bx T_{\norml{v}}$ (see Definition~\ref{eleifnfmeimiswsv}). Suppose $(\Gamma, \ovl\rho, \chi)$ is a triple as in \textup{Setup~\ref{osisienieeimifes}} in which $\Gamma=\Gal_{F_v}$.

For any irreducible representation $\tau$ of $P_v$ with coefficients in $\kappa$, we define
\begin{equation*}
\Gamma_\tau\defining\{\sigma\in\Gal_{F_v}|\tau^\sigma\cong\tau\},
\end{equation*}
which is a subgroup of $\Gal_{F_v}$ containing $P_v$. Let $\bx T_\tau$ be the image of $\Gamma_\tau$ in $\bx T_v$.

\begin{lm}\label{oeiipwinvuiqpqqa}
Suppose $\tau$ is an irreducible representation of $P_v$ with coefficients in $\kappa$.
\begin{enumerate}
\item
The dimension of $\tau$ is coprime to $\ell$; and $\tau$ has a unique deformation to a representation $\tld\tau$ of $P_v$ over $\mcl O$;
\item
The deformation $\tld\tau$ in (1) admits a unique extension to $\Gamma_\tau\cap I_{F_v}$ over $\mcl O$ with tame determinant (that is, the determinant maps the tame generator to a root of unity with finite order coprime to $\ell$);
\item
There exists an extension of the deformation $\tld\tau$ in (2) to a representation of $\Gamma_\tau$ over $\mcl O$.
\end{enumerate}
\end{lm}
\begin{proof}
This is \cite[Lemma~2.4.11]{CHT08}.
\end{proof}

For each irreducible representation $\tau$ of $P_v$ with coefficients in $\kappa$ and each lifting $\rho$ of $\ovl\rho$ to an object $\bx R$ of $\mrs C_{\mcl O}$, we set
\begin{equation*}
\bx W_\tau(\rho)\defining \Hom_{\bx R[P_v]}(\tld\tau\otimes_{\mcl O}\bx R, \Std\circ\rho),
\end{equation*}
where $\tld\tau$ is a representation of $\Gamma_\tau$ with coefficients in $\mcl O$ from Lemma~\ref{oeiipwinvuiqpqqa}(3) (later we will fix such choices in Proposition~\ref{lsimeinfeniemiws}). Note that $\bx W_\tau(\rho)$ is a finite free $\bx R$-module equipped with the induced continuous action by $\bx T_\tau$, and $\tld\tau\otimes_{\mcl O}\bx W_\tau(\rho)$ is equipped with a natural action by $\Gamma_\tau$. 

\begin{assumption}\label{osiiusuenfueiis}
$\kappa$ contains a subfield $\kappa^\flat$ of degree $2$ \sut every irreducible summand of $(\Std\circ\ovl\rho)|_{P_v}\otimes_\kappa\ovl\kappa$ is defined over $\kappa^\flat$.
\end{assumption}

\begin{defi}\label{lsiiieinnifeis}
Assume Assumption~\ref{osiiusuenfueiis} holds. We denote by $\mfk T$ the set of isomorphism classes of irreducible representations $\tau$ of $\bx P_v$ \sut $W_\tau(\ovl\rho)\ne 0$. Then $\Gal_{F_v}$ acts on $\mfk T$ by conjugation, whose orbits we denote by $\mfk T/\Gal_{F_v}$. For each $\tau\in\mfk T$, we denote by $[\tau]$ its orbit in $\mfk T/\Gal_{F_v}$.

We denote the underlying space of $\Std\circ\ovl\rho$ by $V$, and define $V^*\defining V^\vee\otimes\chi$. Then $\rho$ defines an isomorphism of $\kappa[\Gal_{F_v}]$-modules $\psi: V\cong V^*$. By comparing $P_v$-isotypic components, we get a natural isomorphism of of $\kappa[\Gal_{F_v}]$-modules
\begin{equation*}
\psi_\tau: \bx W_\tau(\ovl\rho)^*\cong W_{\tau^\star}(\ovl\rho)
\end{equation*}
for some $\tau^\star\in\mfk T$.

Then we construct a $\Gal_{F_v}$-stable partition $\mfk T=\mfk T_1\coprod \mfk T_2\coprod \mfk T_3$ as follows.
\begin{enumerate}
\item
We define $\mfk T_1$ to be the subset of $\mfk T$ consisting of those $\tau$ \sut $[\tau]\ne [\tau^\star]$. We also fix a subset $\mfk T_1^\heart\subset \mfk T_1$ \sut $\{\tau, \tau^\star|\tau\in\mfk T_1^\heart\}$ is set of representatives for the $\Gal_{F_v}$-action on $\mfk T_1$.
\item
We define $\mfk T_2$ to be the subset of $\mfk T$ consisting of those $\tau$ \sut $\tau\cong \tau^\star$ as $P_v$-representations. We also fix a subset $\mfk T_2^\heart\subset \mfk T_2$ of representatives for the $\Gal_{F_v}$-action on $\mfk T_2$.
\item
We define $\mfk T_3$ to be the subset of $\mfk T$ consisting of those $\tau$ \sut $[\tau]=[\tau^\star]$ but $\tau\ncong \tau^\star$ as $P_v$-representations. We also fix a subset $\mfk T_3^\heart\subset \mfk T_3$ of representatives for the $\Gal_{F_v}$-action on $\mfk T_3$.
\end{enumerate}
\end{defi}

\begin{prop}\label{lsimeinfeniemiws}
Assume \textup{Assumption~\ref{osiiusuenfueiis}} holds.
\begin{enumerate}
\item
For each $\tau\in \mfk T_1^\heart$, we fix a representation $\tld\tau$ of $\Gamma_\tau$ with coefficients in $\mcl O$ from \textup{Lemma~\ref{oeiipwinvuiqpqqa}(3)}. We set
\begin{equation*}
\ovl G_\tau\defining \udl\Aut(\bx W_\tau(\ovl\rho)).
\end{equation*}
\item
For each $\tau\in \mfk T_2^\heart$, any isomorphism $\iota_\tau: \tau\xr\sim \tau^\star$ induces a sign-symmetric pairing (for some fixed sign $\eps_\tau$) on $\tau$, cf. \cite[Lemma~6.5]{Boo19a}. Moreover, there is a sign-symmetric perfect pairing $\bra{-, -}_{\bx W(\ovl\rho)}$ on $\bx W(\ovl\rho)$, and we set
\begin{equation*}
\ovl G_\tau\defining\udl{\bx{GAut}}(\bx W_\tau(\ovl\rho), \bra{-, -}_{\bx W_\tau(\ovl\rho)}),
\end{equation*}
which is a split general symplectic group or split general orthogonal group over $\kappa$. It follows from \cite[Lemma~6.10, Lemma~6.11]{Boo19a} that we can fix a lifting $\tld\tau$ of $\tau$ as in \textup{Lemma~\ref{oeiipwinvuiqpqqa}(3)} with an isomorphism $\tld\tau\xr\sim\tld\tau^*r$ lifting that on $\iota_\tau$.
\item
For each $\tau\in \mfk T_3^\heart$, we define
\begin{equation*}
\Gamma_{\tau\oplus\tau^\star}\defining\{g\in \Gal_{F_v}: (\tau\oplus\tau^\star)^g\cong \tau\oplus\tau^\star\},
\end{equation*}
which contains $\Gamma_\tau$ as a subgroup of index 2, as conjugation either preserves $\tau$ and $\tau^*$ or swaps them. Any isomorphism $\iota_\tau: (\tau\oplus\tau^\star)\xr\sim(\tau\oplus\tau^\star)^*$ of $\kappa[\Gamma_{\tau\oplus\tau^\star}]$-modules induces a sign-symmetric pairing on $\bx W_\tau(\ovl\rho)$, and we set
\begin{equation*}
\ovl G_\tau\defining\udl{\bx{GAut}}(\bx W_\tau(\ovl\rho), \bra{-, -}_{\bx W_\tau(\ovl\rho)}),
\end{equation*}
which is a split general symplectic group or split general orthogonal group over $\kappa$. Then we obtain a perfect pairing on $\tau\oplus \tau^\star$, cf.~\cite[p.~30]{Boo19a}. It follows from \cite[Lemma~6.12]{Boo19a} that we can fix a lifting $\tld{\tau\oplus\tau^\star}$ of $\tau\oplus\tau^\star$ with an isomorphism $\tld{\tau\oplus\tau^\star}\xr\sim \paren{\tld{\tau\oplus\tau^\star}}^*$ lifting that on $\iota_\tau$.
\end{enumerate}

In each case, we can lift $\ovl G_\tau$ to a split reductive group $G_\tau$ over $\mcl O$ by lifting the linear algebra data.
\end{prop}
\begin{proof}
These follow from \cite[\S6.3]{Boo19a}.
\end{proof}

\begin{prop}\label{oiueerelwos}
Assume \textup{Assumption~\ref{osiiusuenfueiis}} holds. Suppose $\bx R$ is an object of $\mrs C_{\mcl O}$ and $\rho$ is a lifting of $\ovl\rho$ to $\bx R$.
\begin{enumerate}
\item
Recall that we have defined
\begin{equation*}
\bx W_\tau(\rho)\defining \Hom_{\bx R[P_v]}(\tld\tau\otimes_{\mcl O}\bx R, \Std\circ\rho)
\end{equation*}
for each $\tau\in \mfk T_1^\heart$, which is a finite free $\bx R$-module of some rank $m_\tau\ge 1$ equipped with the induced continuous action by $\bx T_\tau$.
\item
Recall that we have defined
\begin{equation*}
\bx W_\tau(\rho)\defining \Hom_{\bx R[P_v]}(\tld\tau\otimes_{\mcl O}\bx R, \Std\circ\rho)
\end{equation*}
for each $\tau\in \mfk T_2^\heart$, which is a finite free $\bx R$-module of some rank $m_\tau\ge 1$ equipped with the induced continuous action by $\bx T_\tau$. 
\item
We define
\begin{equation*}
\bx W_{\tau\oplus\tau^\star}(\rho)\defining \Hom_{\bx R[P_v]}(\tld{\tau\oplus\tau^\star}\otimes_{\mcl O}\bx R, \Std\circ\rho),
\end{equation*}
which is a finite free $\bx R$-module of some rank $m_\tau\ge 1$ equipped with the induced continuous action by $\bx T_{\tau\oplus\tau^\star}$. 
\end{enumerate}
Then we have a decomposition of $\bx R[\Gal_{F_v}]$-modules
\begin{align*}
\Std\circ\rho&=\bplus_{\tau\in\mfk T_1^\heart}\paren{\Ind_{\Gamma_\tau}^{\Gal_{F_v}}(\tld\tau\otimes_{\mcl O}\bx W_\tau(\rho))\oplus \Ind_{\Gamma_{\tau^\star}}^{\Gal_{F_v}}(\tld\tau^*\otimes_{\mcl O}\bx W_{\tau^\star}(\rho))}\\
&\oplus\bplus_{\tau\in\mfk T_2^\heart}\Ind_{\Gamma_\tau}^{\Gal_{F_v}}(\tld\tau\otimes_{\mcl O}\bx W_\tau(\rho))\oplus\bplus_{\tau\in\mfk T_3^\heart}\Ind_{\Gamma_{\tau\oplus\tau^*}}^{\Gal_{F_v}}(\tld{\tau\oplus\tau^*}\otimes_{\mcl O}\bx W_{\tau\oplus\tau^*}(\rho)).
\end{align*}
Moreover, for $\tau\in\mfk T_2^\heart$ (resp. $\mfk T_3^\heart$), we have natural sign-symmetric pairing on $\bx W_\tau(\rho)$ (resp. $\bx W_{\tau\oplus\tau^\star}(\rho)$).
\end{prop}
\begin{proof}
These follow from \cite[\S6.3]{Boo19a}.
\end{proof}

\begin{prop}\label{lsoeieiinfituores}
Assume \textup{Assumption~\ref{osiiusuenfueiis}} holds. We also keep the choices of Definition~\ref{lsiiieinnifeis} and the isomorphisms $\iota_\tau: \tau\xr\sim\tau^\star$ as in \textup{Proposition~\ref{lsimeinfeniemiws}(2)}. Then for every object $\bx R$ of $\mrs C_{\mcl O}$, the assignment
\begin{equation*}
\rho\mapsto \paren{(\bx W_\tau(\rho))_{\tau\in\mfk T_1^\heart}, (\bx W_\tau(\rho))_{\tau\in\mfk T_2^\heart}, (\bx W_{\tau\oplus\tau^\star}(\rho))_{\tau\in\mfk T_3^\heart}}
\end{equation*}
induces a bijection between deformations of $\ovl\rho$ to $\bx R$ (\wrt $\chi$) and equivalence classes of tuples $\paren{(\bx W_\tau)_{\tau\in\mfk T_1^\heart}, (\bx W_\tau)_{\tau\in\mfk T_2^\heart}, (\bx W_{\tau\oplus\tau^\star})_{\tau\in\mfk T_3^\heart}}$ where
\begin{enumerate}
\item
for $\tau\in \mfk T_1^\heart$, $\bx W_\tau: \bx T_\tau\to \GL_{m_\tau}(\bx R)$ is a continuous homomorphism that reduces to $\bx W_\tau(\ovl\rho)$;
\item
for $\tau\in \mfk T_2^\heart$, $\bx W_\tau: \bx T_\tau\to G_\tau(\bx R)$ is a continuous homomorphism that reduces to $\bx W_\tau(\ovl\rho)$ with trivial similitude character; and
\item
for $\tau\in\mfk T_3^\heart$, $\bx W_{\tau\oplus\tau^\star}: \bx T_{\tau\oplus\tau^\star}\to G_\tau(\bx R)$ is a continuous homomorphism that reduces to $\bx W_{\tau\oplus\tau^\star}(\ovl\rho)$ with trivial similitude character.
\end{enumerate}
Here, two tuples
\begin{equation*}
\paren{(\bx W_\tau)_{\tau\in\mfk T_1^\heart}, (\bx W_\tau)_{\tau\in\mfk T_2^\heart}, (\bx W_{\tau\oplus\tau^\star})_{\tau\in\mfk T_3^\heart}}, \paren{(\bx W'_\tau)_{\tau\in\mfk T_1^\heart}, (\bx W'_\tau)_{\tau\in\mfk T_2^\heart}, (\bx W'_{\tau\oplus\tau^\star})_{\tau\in\mfk T_3^\heart}}
\end{equation*}
are said to be equivalent if 
\begin{enumerate}
\item
for $\tau\in \mfk T_1^\heart$, $\bx W_\tau$ and $\bx W'_\tau$ are conjugate by an element of $\GL_{m_\tau}^\wedge(\bx R)$;
\item
for $\tau\in \mfk T_2^\heart$, $\bx W_\tau$ and $\bx W'_\tau$ are conjugate by an element of $(G_\tau')^\wedge(\bx R)$, where $G_\tau'$ is the derived subgroup of $G_\tau$; and
\item
for $\tau\in \mfk T_3^\heart$, $\bx W_{\tau\oplus\tau^\star}$ and $\bx W'_{\tau\oplus\tau^\star}$ are conjugate by an element of $(G_\tau')^\wedge(\bx R)$, where $G_\tau'$ is the derived subgroup of $G_\tau$.\end{enumerate}
\end{prop}
\begin{proof}
The proof of \cite[Proposition~6.15]{Boo19a} goes through, cf.~\cite[Proposition~3.28]{LTXZZa}.
\end{proof}

We record the following corollary that will be useful later.

\begin{cor}\label{losoieiefniewswivmie}
Suppose $\ell\nmid\prod_{i=1}^{2m}(\norml{v}^i-1)$. Let $\iota_\ell: \bb C\xr\sim\ovl{\bb Q_\ell}$ be an isomorphism, and $\bx R$ an object of $\mrs C_{\mcl O}$ contained in $\ovl{\bb Q_\ell}$. Let $\rho_1, \rho_2$ be two liftings of $\rho$ to $\bx R$, and let $\Pi_1, \Pi_2$ be two irreducible admissible representations of $\GL_{2m}(F_v)$ \sut $\WD(\Std\circ\rho_i\otimes_{\bx R}\ovl{\bb Q_\ell})^{F\dash\sems}=\rec_{2m}(\iota_\ell\Pi_i)$ for every $i\in\{1, 2\}$. Then $\Pi_1$ and $\Pi_2$ are in the same Bernstein component.
\end{cor}
\begin{proof}
Using Proposition~\ref{oiueerelwos}, the proof of \cite[Corollary~3.31]{LTXZZa} goes through. 
\end{proof}

We can now drop the Assumption~\ref{osiiusuenfueiis} and define minimally ramified liftings of $\ovl\rho$ for general coefficient field $E$.

\begin{defi}\label{minaimlsoieieres}
Suppose $\ell>2m$. Choose a finite unramified extension $E_\dagger$ of $E$ contained in $\ovl{\bb Q_\ell}$ contained in $\ovl{\bb Q_\ell}$ that satisfies Assumption~\ref{uansoisoifeifmes} with ring of integers $\mcl O_\dagger$ and residue field $\kappa_\dagger$. We also keep the choices of Definition~\ref{lsiiieinnifeis} and the isomorphisms $\iota_\tau: \tau\xr\sim\tau^\star$ as in Proposition~\ref{lsimeinfeniemiws}(2). Let $\rho$ be a lifting of $\ovl\rho$ to an object $\bx R$ of $\mrs C_{\mcl O}$, 
then we obtain the decomposition of $\bx R[\Gal_F]$-modules
\begin{align*}
\Std\circ\rho&=\bplus_{\tau\in\mfk T_1^\heart}\paren{\Ind_{\Gamma_\tau}^{\Gal_{F_v}}(\tld\tau\otimes_{\mcl O}\bx W_\tau(\rho))\oplus \Ind_{\Gamma_{\tau^\star}}^{\Gal_{F_v}}(\tld\tau^*\otimes_{\mcl O}\bx W_{\tau^\star}(\rho))}\\
&\oplus\bplus_{\tau\in\mfk T_2^\heart}\Ind_{\Gamma_\tau}^{\Gal_{F_v}}(\tld\tau\otimes_{\mcl O}\bx W_\tau(\rho))\oplus\bplus_{\tau\in\mfk T_3^\heart}\Ind_{\Gamma_{\tau\oplus\tau^*}}^{\Gal_{F_v}}(\tld{\tau\oplus\tau^*}\otimes_{\mcl O}\bx W_{\tau\oplus\tau^*}(\rho)).
\end{align*}
We say that $\rho$ is minimally ramified with similitude $\chi$ if the following conditions hold:
\begin{enumerate}
\item
For each $\tau\in\mfk T_1^\heart$, $\bx W_\tau(\rho)$ is a minimally ramified extension of $\bx W_\tau(\ovl\rho)$ in the sense of \cite[Definition~2.4.14]{CHT08}.
\item
For each $\tau\in\mfk T_2^\heart$ (resp. $\tau\in \mfk T_3^\heart$), $\bx W_\tau$ (resp. $\bx W_{\tau\oplus\tau^*}$) is minimally ramified in the sense of Definition~\ref{psoosiiieufunnuenes} as a representation of $\bx T_\tau$ (resp. $\bx T_{\tau\oplus\tau^*}$) valued in the group $G_\tau$ with trivial similitude character.
\end{enumerate}
\end{defi}
\begin{rem}
It is easy to check that Definition~\ref{minaimlsoieieres} doesn't depend on the choices of $E_\dagger$ and the choices made in Definition~\ref{lsiiieinnifeis} and Proposition~\ref{lsimeinfeniemiws}.
\end{rem}

\begin{defi}\label{sspeoifiekfensinws}
When $\ell> 2m$, let $\mrs D^{\bmin, \square}$ denote the local lifting condition of $\ovl\rho$ (\wrt $\chi$) that classifies minimally ramified liftings of $\rho$ (see Definition~\ref{minaimlsoieieres}).
\end{defi}

\begin{prop}\label{psisiidimfimies}
The following statements hold:
\begin{enumerate}
\item
The universal lifting ring $\msf R^\square_{\ovl\rho}\in\mrs C_{\mcl O}$ is a reduced local complete intersection, flat and of pure relative dimension $\dim G'_\kappa$ over $\mcl O$.
\item
Every irreducible component of $\Spf\msf R^\square_{\ovl\rho}$ is a local lifting condition (see \textup{Definition~\ref{liftings-anadi-defintiieos}}).
\item
When $\ell>2m$, the local lifting condition $\mrs D^{\bmin, \square}$ is an irreducible component of $\msf R^\square_{\ovl\rho}$ that is formally smooth over $\Spf\mcl O$ of pure relative dimension $\dim G'_\kappa$. In particular, it follows from \textup{Equation~\eqref{psisieubifes}} that
\begin{equation*}
\dim_\kappa\bx L(\mrs D^\bmin)=\dim_\kappa\bx H^0(\Gal_{F_v}, \ad^0(\ovl\rho)).
\end{equation*}
\end{enumerate}
\end{prop}
\begin{proof}
For this proposition, we may assume that $E$ satisfies Assumption~\ref{osiiusuenfueiis}.

For (1), we define for each $\tau\in \mfk T_1^\heart\coprod \mfk T_2^\heart\coprod\mfk T_3^\heart$ the universal lifting ring $\msf R_\tau^\square$ representing the lifting condition of $\bx W_\tau(\ovl\rho)$ or $\bx W_{\tau\oplus \tau^\star}(\ovl\rho)$ valued in $G_\tau$ (with trivial similitude character when $\tau\in \mfk T_2^\heart\coprod\mfk T_3^\heart$), see Definition~\ref{liftings-anadi-defintiieos}. It follows from Proposition~\ref{lsoeieiinfituores} that $\msf R_{\ovl\rho}^\square$ is a power series ring over 
\begin{equation*}
\hat\btimes_{\tau\in\mfk T_1^\heart\coprod\mfk T_2\coprod\mfk T_3^\heart}\msf R_\tau^\square.
\end{equation*}
For each $\tau\in \mfk T_1^\heart\coprod\mfk T_2\coprod\mfk T_3^\heart$, we know $\msf R_\tau^\square$ is a local complete intersection, flat and of pure relative dimension $\dim G'_\kappa$ over $\Spf \mcl O$ by \cite[Theorem~2.5]{Sho18} for $\tau\in \mfk T_1^\heart$, and by Proposition~\ref{osoosiiieubfes} for $\tau\in\mfk T_2^\heart\coprod\mfk T_3^\heart$. On the other hand, $\msf R_{\ovl\rho}^\square[\frac{1}{\ell}]$ is a local complete intersection, generically regular and of pure dimension $\dim G'_\kappa$ by \cite[Theorem~3.3.2]{B-G19} or \cite[Theorem~1]{B-P19}. Thus $\msf R^\square_{\ovl\rho}$ is also a local complete intersection, flat and equidimensional of dimension $\dim G'_\kappa$ over $\mcl O$. Since $\msf R^\square_{\ovl\rho}$ is generically reduced and Cohen--Macaulay, it is reduced. (1) is proved.

For (2): take an irreducible component $\mrs D^\square$ of $\Spf \msf R^\square_{\ovl\rho}$, then the conjugation action induces a homomorphism $f: (G')^\wedge\times_{\Spf\mcl O}\mrs D^\square\to \Spf \msf R^\square_{\ovl\rho}$, whose image contains $\mrs D^\square$. Since the source of $f$ is irreducible, and $\Spf \msf R^\square_{\ovl\rho}$ is reduced, the image of $f$ has to be $\mrs D^\square$. In other words, $\mrs D^\square$ is a local lifting condition.

(3): Since $\mrs D^{\bmin, \square}$ is Zariski closed in $\Spf \msf R_{\ovl\rho}^\square$, it suffices to show that $\mrs D^{\bmin, \square}$ is formally smooth over $\Spf \mcl O$ of pure relative dimension $\dim G'_\kappa$. This claim follows from \cite[Corollary~6.16]{Boo19a}.
\end{proof}

\subsection{Level-raising deformations}

In this subsection, we study level-raising deformations. Assume that $N=2m+1$ is odd and that $\ell>2m$. We work in the following setting.

\begin{setup}\enskip
\begin{itemize}
\item
Let $v$ be a finite place of $F$ \sut $\ell\nmid (\norml{v}^2-1)\norml{v}$.
\item
Let $P_v$ be the maximal subgroup of $I_{F_v}$ of pro-order coprime to $\ell$, and set $\bx T_v\defining \Gal_{F_v}/P_v$, which is isomorphic to the $\norml{v}$-tame group $\bx T_{\norml{v}}$ (see Definition~\ref{eleifnfmeimiswsv}).
\item
Let $(\Gamma, \ovl\rho, \chi)$ be a triple as in Setup~\ref{osisienieeimifes}, satisfying that $\Gamma=\Gal_{F_v}, \chi=\ve_\ell^{1-2m}$, $\ovl\rho$ is unramified, and that
\begin{equation*}
\dim_\kappa\ker\big(\Std\circ\ovl\rho_v(\phi_v)-\norml{v}^{-m}\big)^{2m}=1.
\end{equation*}
\end{itemize}
\end{setup} 

By Lemma~\ref{oeiipwinvuiqpqqa}(1), every lifting $\rho$ of $\ovl\rho$ to an object $\bx R$ of $\mrs C_{\mcl O}$ (\wrt $\chi$) factors through $\bx T_{\norml{v}}$. Moreover, it follows from Lemma~\ref{poeoiemifemifws} that for such a lifting $\rho$, if we fix an isomorphism between the underlying space of $\Std\circ\rho$ with $\bx R^{2m}$, then we obtain a canonical decomposition
\begin{equation}\label{ieiurutiorpws}
\bx R^{\oplus(2m)}=\bx M_0\oplus\bx M_1
\end{equation}
of free $\bx R$-modules stable under the action of $\rho(\phi_v)$, \sut
\begin{equation*}
\det(T-\Std\circ\rho(\phi_v)|\bx M_0\otimes_{\bx R}\kappa)=(T-\norml{v}^{1-m})(T-\norml{v}^{-m})
\end{equation*}
holds in $\kappa[T]$.

\begin{defi}\label{leveiieenrifmesw}
We define $\mrs D^{\mix, \square}$ to be the the local lifting condition of $\ovl\rho$ (see Definition~\ref{oeieriwowpq}) that classifies liftings of $\ovl\rho$ to an object $R$ of $\mrs C_{\mcl O}$ \sut in the decomposition~\eqref{ieiurutiorpws}, $\rho(I_{F_v})$ preserves $\bx M_0$ and acts trivially on $\bx M_1$. We define
\begin{enumerate}
\item
$\mrs D^{\unr, \square}$ to be the local lifting condition contained in $\mrs D^{\mix, \square}$ \sut the action of $\rho(I_{F_v})$ on $\bx M_0$ is also trivial.
\item
$\mrs D^{\ram, \square}$ to be the local lifting condition contained in $\mrs D^{\mix, \square}$ \sut
\begin{equation*}
\det(T-\Std\circ\rho(\phi_v)|\bx M_0)=(T-\norml{v}^{1-m})(T-\norml{v}^{-m})
\end{equation*}
in $\bx R[T]$.
\end{enumerate}
\end{defi}

\begin{prop}\label{ososoidvienfiemsws}
The formal scheme $\mrs D^{\mix, \square}$ is formally smooth over $\Spf\mcl O[[x_0, x_1]]/(x_0x_1)$ of pure relative dimension $2m^2+m-1$ \sut the irreducible components defined by $x_0=0$ and $x_1=0$ are $\mrs D^{\unr, \square}$ and $\mrs D^{\ram, \square}$, respectively. In particular, $\mrs D^{\ram, \square}$ is formally smooth over $\Spf\mcl O$ of pure relative dimension $2m^2+m$, and it follows from \textup{Equation~\ref{psisieubifes}} that
\begin{equation*}
\dim_\kappa\bx L(\mrs D_{\ovl\rho}^\ram)=\dim_\kappa\bx H^0(\Gal_{F_v}, \ad^0(\ovl\rho)).
\end{equation*}
\end{prop}
\begin{proof}
We modify the proof of \cite[Proposition~3.35]{LTXZZa}. Identifying the underlying space of $\Std\circ\ovl\rho$ with $\kappa^{\oplus(2m)}$, there exists a decomposition
\begin{equation*}
\kappa^{\oplus(2m)}=\ovl{\bx M}_0\oplus \ovl{\bx M}_1
\end{equation*}
\sut
\begin{equation*}
\det(T-\Std\circ\ovl\rho(\phi_v)|\ovl{\bx M}_0)=(T-\norml{v}^{1-m})(T-\norml{v}^{-m})
\end{equation*}
in $\kappa[T]$. Upon changing a basis, we may assume that $\ovl{\bx M}_0$ is spanned by the first two factors and $\ovl{\bx M}_1$ is spanned by the last $2m-2$ factors. Then we get two unramified representations $\ovl\rho_0: \bx T_{\norml{v}}\to \GSp_2(\kappa)$ and $\ovl\rho_1: \bx T_{\norml{v}}\to \GSp_{2m-2}(\kappa)$. Let $\mrs D_0^\square$ be the unrestricted lifting condition of $\ovl\rho_0$ \wrt $\chi=\ve_\ell^{1-2m}$, and let $\mrs D_1^\square$ be the local lifting condition of $\ovl\rho_1$ \wrt $\chi=\ve_\ell^{1-2m}$ classifying unramified liftings.

We say that a lifting $\rho$ of $\ovl\rho$ to an object $\bx R$ of $\mrs C_{\mcl O}$ is \emph{standard} if
\begin{equation*}
\rho(t)=\begin{bmatrix}A_0& 0\\ 0 &I_{2m-2}\end{bmatrix}, \quad \rho(\phi_v)=\begin{bmatrix}B_0& 0\\ 0 &B_1\end{bmatrix},
\end{equation*}
with $A_0, B_0\in\GL_2(\bx R)$ and $B_1\in\GSp_{2m-2}(\bx R)$. Let $\mrs D_{0, 1}^{\mix, \square}\subset \mrs D^{\mix, \square}$ be the subfunctor classifying standard liftings. Then there exists a natural isomorphism
\begin{equation*}
\mrs D_{0, 1}^{\mix, \square}\cong \mrs D_0^\square\times_{\Spf\mcl O}\mrs D_1^\square.
\end{equation*}
Moreover, there is an isomorphism
\begin{equation*}
\mrs D_{0, 1}^{\mix, \square}\times_{\Spf\mcl O}\big(\Sp_2^\wedge\times_{\Spf\mcl O}\Sp_{2m-2}^\wedge\bsh\Sp_{2m}^\wedge\big)\cong\mrs D^{\mix, \square}.
\end{equation*}
By Proposition~\ref{ossoooieifmeies} (here we use that $\ell>2m$), $\mrs D_1^\square$ is formally smooth of relative dimension $(2m-1)(m-1)$ over $\Spf\mcl O$. It follows from \cite[Proposition~5.5]{Sho16}(2) that $\mrs D_0^\square$ is isomorphic to $\Spf\mcl O[[x_0, \ldots, x_3]]/(x_0x_1)$, where the irreducible component defined by $x_0=0$ classifies unramified liftings, and the irreducible component defined by $x_1=0$ classifies liftings $\rho$ \sut $\norml{v}^{-m}$ 
is an eigenvalue of $\Std\circ\rho(\phi_v)$. The assertion follows.
\end{proof}

\section{An almost minimal R=T theorem}

In this section, we prove an almost minimal $\msf R=\bb T$ theorem for self-dual Galois representations with coefficients in a finite field satisfying a property called \emph{rigid}. 

\subsection{Rigid Galois representations}

We take a finite set $\Pla^\bmin$ of finite places of $F$ containing $\Pla^\bad$. We take an odd rational prime $\ell$ larger than $2(b_\xi-a_\xi)+2$ that is unramified in $F$, \sut $\Pla^\bmin$ contains no $\ell$-adic places. We then work in the setting of Setup~\ref{seutpeoeifmeis}. 

We fix an isomorphism $\iota_\ell: \bb C\xr\sim\ovl{\bb Q_\ell}$ and assume that the complex algebraic representation of $\Res_{F/\bb Q}\GL_{2m}$ with highest weight $\xi$ can be defined over $\iota_\ell^{-1}E$.

We take a finite set $\Pla^\lr$ of finite places of $F$ \sut 
\begin{itemize}
\item
$\Pla^\lr=\vn$ when $N$ is even;
\item
$\Pla^\lr$ is disjoint from $\Pla^\bmin\cup\Pla(\ell)$; and
\item
$\norml{v}^2-1$ is divisible by $\ell$ for every place $v\in \Pla^\lr$ 
\end{itemize}

We fix an element $\xi=(\xi_\tau)_\tau\in (\bb Z_\le^{2m})^{\Pla^\infty}$ satisfying $\xi_{\tau, i}=-\xi_{\tau, 2m+1-i}$ for every $\tau\in\Pla^\infty$ and $1\le i\le 2m$. 

We consider a triple $(\Gamma, \ovl\rho, \chi)$ from Setup~\ref{osisienieeimifes} in which $\Gamma=\Gal_F$ and $\chi=\ve_\ell^{2-N}$. 

\begin{defi}\label{isiinefies}
We say that $\ovl\rho$ is \emph{rigid} for $(\Pla^\bmin, \Pla^\lr)$ 
 if the following are satisfied:
\begin{enumerate}
\item
For $v\in\Pla^\bmin$, every lifting of $\ovl\rho_v$ is minimally ramified (see Definition~\ref{minaimlsoieieres}).
\item
For $v\in\Pla^\lr$, $\dim_\kappa\ker\big(\ovl\rho_v(\phi_v)-\norml{v}^{-m}\big)^{2m}=1$.
\item
For $v\in\Pla(\ell)$, $\ovl\rho_v$ is Fontaine--Laffaille crystalline with regular weights in $[a_\xi, b_\xi]$ (see Definition~\ref{otetotiiemifes}), \sut $\ell>2(b_\xi-a_\xi)+2$.
\item
For every finite place $v$ of $F$ not in $\Pla^\bmin\cup\Pla^\lr\cup\Pla(\ell)$, the homomorphism $\ovl\rho_v$ is unramified.
\end{enumerate}
\end{defi}

Suppose now that $\ovl\rho$ is rigid for $(\Pla^\bmin, \Pla^\lr)$. Consider a global deformation problem (see Definition~\ref{defomrioainoiereid})
\begin{equation*}
\mrs S\defining\paren{\ovl\rho, \ve_\ell^{2-N}, \Pla^\bmin\cup\Pla^\lr\cup\Pla(\ell), \{\mrs D_v^\square\}_{v\in\Pla^\bmin\cup\Pla^\lr, \Pla(\ell)}}
\end{equation*}
where
\begin{itemize}
\item
for $v\in\Pla^\bmin$, $\mrs D_v^\square$ is the unrestricted local lifting condition of $\ovl\rho_v$;
\item
for $v\in\Pla^\lr$, $\mrs D_v^\square$ is the local lifting condition $\mrs D^\ram_v$ of $\ovl\rho_v$ from Definition~\ref{leveiieenrifmesw}; and
\item
for $v\in\Pla(\ell)$, $\mrs D_v^\square$ is the local lifting condition $\mrs D^{\FL, \square}$ of $\ovl\rho_v$ from Definition~\ref{ooidnuefiems}.
\end{itemize}
Then we have the global deformation ring $\msf R_{\mrs S}$ from Proposition~\ref{peoieutnmeumes}.

\subsection{Adequate subgroups}

In this subsection, we recall the notion of $G^\circ$-adequate subgroups following \cite{Whi23a}. We consider a triple $(\Gamma, \ovl\rho, \chi)$ from Setup~\ref{osisienieeimifes} in which $\Gamma=\Gal_F$ and $\chi=\ve_\ell^{2-N}$. 

\begin{defi}\label{otieinnvnuducms}
We say that a subgroup $H\le G^\circ(\kappa)$ is $G^\circ$-adequate if
\begin{enumerate}
\item
the cohomology groups
\begin{equation*}
\bx H^0(H, (\mfk g_\kappa')^\vee), \quad\bx H^1(H, \kappa), \quad\bx H^1(H, (\mfk g_\kappa')^\vee)
\end{equation*}
vanish; and
\item
for every irreducible $\kappa[H]$-submodule $W$ of $(\mfk g'_\kappa)^\vee$, there exists an element $h\in H$ \sut $w(z)\ne 0$ for some $w\in W$ and some $z\in \Lie Z(Z_{G^\circ}(h))\cap\mfk g'_\kappa$.
\end{enumerate}
\end{defi}
Here $Z_{G^\circ}(h)$ denotes the centralizer subgroup of $h$ in $G^\circ$.

\begin{prop}\label{soosieiueuenfue}
If $\ell>2m+4$ and $|\ell-4m|\ne 1$, and $\Gamma'$ is a subgroup of $\Gamma$ \sut $\ovl\rho(\Gamma')\subset G^\circ(\kappa)$ and $\ovl\rho|_{\Gamma'}$ is absolutely irreducible, then $\ovl\rho(\Gamma')$ is a $G^\circ$-adequate subgroup of $G^\circ(\kappa)$.
\end{prop}
\begin{proof}
This follows from \cite[Theorem~2.3.4]{Whi23a}.
\end{proof}

\subsection{Finiteness of deformation rings} 


In this subsection, we prove a finiteness theorem for the deformation ring of a self-dual Galois representation that will be used in~\S\ref{osisieiniiueurmeds}. 
Let $\Pi$ be a $\mfk d$-REASDC representation of $\GL_{2m}(\Ade_F)$ (see Definition~\ref{psoisieneifemiws}). Let $\Pla^\bmin$ be a finite set of finite places of $F$ containing $\Pla^\bad$ \sut $\Pi_v$ is unramified for every finite place $v$ of $F$ not in $\Pla^\bmin$. Let $E\subset\bb C$ be a strong coefficient field of $\Pi$ (see Definition~\ref{strongienfeihenss}). Let $\lbd$ be a finite place of $E$ whose underlying rational prime $\ell$ satisfies
\begin{equation*}
\Pla(\ell)\cap\Pla^\bmin=\vn, \quad \ell>2(2m+1), \quad \ell>2(b_\xi-a_\xi)+2,
\end{equation*}
and $\ell>\norml{v}^{2m}$ for every place $v$ in $\Pla^{\bx{min}}$. We set $\Pla\defining\Pla^\bmin\cup\Pla(\ell)$.

Then we adopt the notation of Setup~\ref{seutpeoeifmeis} with $E=E_\lbd$ plus the notation of Setup~\ref{osisienieeimifes} with $(\Gamma, \ovl\rho, \chi)=(\Gal_F, \ovl\rho_{\Pi, \lbd}, \ve_\ell^{1-2m})$, and fix an isomorphism $\iota_\ell: \bb C\xr\sim\ovl{\bb Q_\ell}$. Let
\begin{equation*}
\mrs S\defining\paren{\ovl\rho_{\Pi, \lbd}, \ve_\ell^{2-N}, \Pla^0, \{\mrs D_v^\square\}_{v\in\Pla}}
\end{equation*}
be a global deformation problem (see Definition~\ref{defomrioainoiereid}), where for each $v\in\Pla^\bmin$, $\mrs D_v^\square$ is an irreducible component (which is a local deformation problem by Proposition~\ref{psisiidimfimies}(2)); and for each $v\in\Pla(\ell)$, $\mrs D_v^\square$ is the local deformation $\mrs D^{\FL, \square}$ of $\ovl\rho_{\Pi, \lbd}$ from Definition~\ref{ooidnuefiems}.

\begin{defi}
Let $\mrs G_{2m}$ be the group scheme over $\bb Z$ which is the semidirect product of the group $\GL_{2m}\times\GL_1$ by the group $\{1, \jmath\}$, which acts on $\GL_{2m}\times\GL_1$ by
\begin{equation*}
\jmath(g, \mu)\jmath^{-1}=(\mu g^{-\top}, \mu).
\end{equation*}
There is a homomorphism $\nu_{\mrs G}: \mrs G_{2m}\to \GL_1$ sending $(g, \mu)$ to $\mu$ and $\jmath$ to $-1$.
\end{defi}

Let $\tld F$ be a totally imaginary quadratic extension of $F$ that is linearly disjoint from $(\ovl F)^{\ker\ovl\rho}(\mu_\ell)$. We fix a real place of $F$ that induces a complex multiplication $\cc\in \Gal_F$. We also choose an algebraic character
\begin{equation*}
\psi: \Gal(\ovl F/\tld F)\to \ovl{\bb Q_\ell}^\times
\end{equation*}
that is crystalline at places of $F$ above $\ell$, satisfying $\psi\psi^\cc=\ve_\ell$ (resp. $\psi$ is trivial) when $N$ is even (resp. $N$ is odd). This is possible by \cite[Lemma~4.1.5]{CHT08}. For our purpose of proving the finiteness of the deformation ring $\msf R_{\mrs S}$, we can enlarge $L$ so that $\psi$ is valued in $\mcl O_L^\times$.

Recall that, for any object $\bx R$ of $\mrs C$ and any continuous homomorphism $\rho': \Gal_F\to G(\bx R)$ with $\nu\circ\rho'=\ve_\ell^{2-N}$, there is a continuous homomorphism $r': \Gal_F\to \mrs G_{2m}(\bx R)$, determined by
\begin{equation*}
r'(g)=
\begin{cases}
(\psi^{-1}(g)\rho'(g), \ve_\ell^{1-2m}(g)) &g\in \Gal_{\tld F}\\
\paren{\psi^{-1}(gc_v)\rho'(g)(J_{2m}')^{-1}, -\ve_\ell^{1-2m}(g)}\jmath & g\notin \Gal_{\tld F}
\end{cases}
,
\end{equation*}
cf.~\cite[Lemma~2.1.2]{CHT08}. Moreover,
\begin{equation*}
\nu_{\mrs G}\circ r'=\ve_\ell^{1-2m}.
\end{equation*}
This construction is compatible with equivalences in the sense that if $B\in (G')^\wedge(\bx R)$ and $\rho'$ is replaced by $\rho'_B\defining B\rho' B^{-1}$, then $r'$ is replaced by $r'_B\defining (a\Std(B), 1)r'(a\Std(B), 1)^{-1}$, where $a\in \bx R$ satisfies $a^2=\nu(B)^{-1}$.

\begin{thm}\label{osiseimfiehgiemifes}
Suppose that $\Std\circ\ovl\rho_{\Pi, \lbd}|_{\Gal_{F(\mu_\ell)}}$ is absolutely irreducible. Then $\msf R_{\mrs S}$ is a finite $\mcl O$-module.
\end{thm}
\begin{proof}
Fix a lift $\rho_{\mrs S}$ of $\ovl\rho$ to $\msf R_{\mrs S}$ representing the universal deformation. From the construction above, we obtain a homomorphism
\begin{equation*}
r_{\mrs S}: \Gal_F\to \mrs G_{2m}(\msf R_{\mrs S}).
\end{equation*}
We may consider the corresponding residual representation
\begin{equation*}
\ovl r_{\mrs S}: \Gal_F\to \mrs G_{2m}(\kappa)
\end{equation*}
and for each $v\in \Pla^\bmin$ a local deformation problem $\tld{\mrs D}_v^\square$ of $\ovl r_{\mrs S}$ corresponding to an irreducible component of the unrestricted deformation problem (cf.~\cite[Proposition~3.33]{LTXZZa}) \sut $r_{\mrs S}|_{\Gal_{F_v}}$ is an object of $\tld{\mrs D}_v^\square(\msf R_{\mrs S})$. We consider (in the notation of \cite[Definition~3.6]{LTXZZa}) the deformation problem
\begin{equation*}
\tld{\mrs S}\defining (\ovl r_{\mrs S}, \ve_\ell^{1-2m}, \Pla, \{\tld{\mrs D}_v^\square\}_{v\in\Pla}),
\end{equation*}
$\tld{\mrs D}_v^\square$ is the Fontaine--Laffaille deformation problem as defined in~\cite[Definition~3.12]{LTXZZa} for each $v\in\Pla(\ell)$. Fix a universal lift
\begin{equation*}
r_{\tld{\mrs S}}: \Gal_F\to \mrs G_{2m}(\msf R_{\tld{\mrs S}}).
\end{equation*}
Then there exists a homomorphism $\theta: \msf R_{\tld{\mrs S}}\to \msf R_{\mrs S}$ \sut $\theta\circ r_{\tld{\mrs S}}$ is equivalent to $r_{\mrs S}$.

We claim that $\theta$ makes $\msf R_{\mrs S}$ a finite module over $\msf R_{\tld{\mrs S}}$. To show this, we denote $\msf R_{\tld{\mrs S}, \mrs S}\defining \msf R_{\mrs S}/\theta(\mfk m_{\msf R_{\tld{\mrs S}}})$, and let $\rho_{\tld F, F}$ denote the $G(\msf R_{\tld{\mrs S}, \mrs S})$-valued representation obtained from $\theta_{\mrs S}$, and let
\begin{equation*}
r_{\tld{\mrs S}, \mrs S}: \Gal_F\to \mrs G_{2m}(\msf R_{\tld{\mrs S}, \mrs S})
\end{equation*}
denote the corresponding homomorphism. 
Then $r_{\tld{\mrs S}, \mrs S}$ is equivalent to $\ovl r_{\mrs S}$. In particular, $r_{\tld{\mrs S}, \mrs S}$ has finite image. It follows from Lemma~\ref{osoosieiwumfiev} that $\msf R_{\mrs S}$ is generated by the traces and similitudes of images of $\Std\circ\rho_{\mrs S}$.
Consider any prime ideal $\mfk p$ of $\msf R_{\tld{\mrs S}, \mrs S}$. Because the image of $r_{\tld{\mrs S}, \mrs S}$ is finite, we see that the images of these traces in $\msf R_{\tld{\mrs S}, \mrs S}/\mfk p$ are sums of roots of unity of bounded degree (up to a scalar in $\mcl O^\times$). In particular, $\msf R_{\tld{\mrs S}, \mrs S}/\mfk p$ is finite. Thus $\msf R_{\tld{\mrs S}, \mrs S}$ is Artinian. Because it has finite residue field, it is itself finite. The claim then follows from Nakayama's lemma.

Note that $\Std\circ\ovl\rho_{\Pi, \lbd}|_{\Gal(\ovl F/F(\mu_\ell))}$ is absolutely irreducible by our hypothesis on $\tld F$; which implies $\Std\circ\ovl\rho_{\Pi, \lbd}(\Gal(\ovl F/\tld F(\mu_\ell)))$ is adequate in the sense of \cite[Definition~2.3]{Tho12} by \cite[Theorem~A.9]{Tho12}. By the same proof of~\cite[Theorem~5.1.4]{Gee11} (strengthened using Thorne's adequacy condition \cite{Tho12}), we know that $\msf R_{\tld{\mrs S}}$ is a finite $\mcl O$-module.

The theorem is proved.
\end{proof}

\subsection{An almost minimal R=T theorem}

In this subsection, we work in the following setting:

\begin{setup}\enskip
\begin{itemize}
\item
Let $\mbf V$ be a quadratic space over $F$ of rank $N$ and discriminant $\mfk d$, \sut $\mbf V_v$ is not split for $v\in\Pla^\lr$. Let $(p_\tau, q_\tau)_{\tau\in\Pla^\infty}$ be the signature of $\mbf V$. Suppose that $q_\tau\in\{0, 2\}$ for each $\tau\in\Pla^\infty$, and set $d(\mbf V)\defining \big(\sum_{\tau\in\Pla^\infty}p_\tau q_\tau\big)/2$.
\item
Let $\Lbd$ be a self-dual $\prod_{v\in\Pla^\infty\cup\Pla^\bmin}\mcl O_v$-lattice in $\mbf V\otimes_F\Ade_F^{\Pla^\infty\cup\Pla^\bmin\cup\Pla^\lr}$ and a neat compact open subgroup of the form
\begin{equation*}
\mdc K=\prod_{v\in\Pla^\bmin\cup\Pla^\lr}\mdc K_v\times\prod_{v\notin\Pla^\infty\cup\Pla^\bmin\cup\Pla^\lr}\SO(\Lbd)(\mcl O_v)
\end{equation*}
in which $\mdc K_v$ is special maximal for $v\in\Pla^\lr$. 
\item
Let $\Pla$ be a finite set of finite places of $F$ containing $\Pla^\bmin\cup\Pla^\lr\cup\Pla(\ell)$.
\end{itemize}
\end{setup}

We obtain a system of (complex) Shimura varieties $\{\bSh(\mbf V, \mdc K')\}_{\mdc K'}$ attached to $\Res_{F/\bb Q}\SO(\mbf V)$ indexed by compact open subgroups $\mdc K'$ of $\mdc K$, which are quasi-projective smooth complex schemes of dimension $d(\mbf V)$. The element $\xi$ gives rise to a continuous homomorphism
\begin{equation*}
\prod_{v\in\Pla(\ell)}\SO(\Lbd)(\mcl O_v)\to \GL(L_\xi),
\end{equation*}
where $L_\xi$ is a finite free $\mcl O$-module, hence induces an $\mcl O$-linear \eTale local system $\mrs L_\xi$ on $\bSh(\mbf V, \mdc K')$ for every compact open subgroups $\mdc K'$ of $\mdc K$, compatible under pullback maps; see~\cite[1.5.8]{KSZ21}.

Recall the abstract orthogonal Hecke algebra $\bb T^{m, \mfk d,\Pla}$ defined in Definition~\ref{aosisieeinifmews}. Let $\phi: \bb T^{m, \mfk d, \Pla}\to \kappa$ be the homomorphism \sut for every finite place $v$ of $F$ not in $\Pla$, we have $\phi|_{\bb T^{\mfk m, \mfk d}_v}=\phi_{\bm\alpha_v}$ (see Definition~\ref{ieifiemifesqpo}) where $\bm\alpha_v=(\alpha_{v, 1}^{\pm1}, \ldots, \alpha_{v, m}^{\pm1})$ is the orthogonal abstract Satake parameter at $v$ (see Definition~\ref{slabossleiisTakeows}) satisfying that
\begin{equation*}
\Brace{\alpha_{v, 1}^{\pm1}\iota_\ell\big(\sqrt{\norml{v}}\big)^{2-N}, \ldots, \alpha_{v, m}^{\pm1}\iota_\ell\big(\sqrt{\norml{v}}\big)^{2-N}}
\end{equation*}
is the multiset of generalized eigenvalues of $\ovl\rho_v(\phi_v)$ in $\ovl{\bb F_\ell}$. We denote by $\mfk m$ the kernel of $\phi$.

\begin{thm}\label{peeoieowivmincs}
Under the above setup, we assume
\begin{enumerate}
\item[(D0)]
(already assumed) $\ell> 2(b_\xi-a_\xi)+2$, and $\ell$ is unramified in $F$;
\item[(D1)]
$\ell>2(2m+1)$; 
\item[(D2)]
$\Std\circ\ovl\rho|_{\Gal\paren{\ovl F/F(\sqrt{\mfk d})(\mu_\ell)}}$ is absolutely irreducible;
\item[(D3)]
$\ovl\rho$ is rigid for $(\Pla^\bmin, \Pla^\lr)$;
\item[(D4)]
for every finite set $\Pla'$ of finite places of $F$ containing $\Pla$, and every compact open group $\mdc K'$ of $\mdc K$ of the form $\mdc K'=\mdc K'_{\Pla'}\times \mdc K^{\Pla'}$, we have
\begin{equation*}
\etH^d(\bSh(\mbf V, \mdc K'), \mrs L_\xi\otimes_{\mcl O}\kappa)_{\bb T^{m, \mfk d, \Pla'}\cap\mfk m}=0
\end{equation*}
unless $d=d(\mbf V)$.
\end{enumerate}
Let $\bb T$ be the image of $\bb T^{m, \mfk d, \Pla}$ in $\End_{\mcl O}\big(\etH^{d(\mbf V)}(\bSh(\mbf V, \mdc K), \mrs L_\xi)\big)$. If $\bb T_{\mfk m}\ne 0$, then
\begin{enumerate}
\item
There is a canonical isomorphism $\msf R_{\mrs S}\cong \bb T_{\mfk m}$ of local complete intersection rings over $\mcl O$.
\item
The $\bb T_{\mfk m}$-module $\etH^{d(\mbf V)}(\bSh(\mbf V, \mdc K), \mrs L_\xi)_{\mfk m}$ is finite and free.
\end{enumerate}
\end{thm}
\begin{rem}
The Condition~(D4) is assumed because we need to use torsion-freeness of the cohomology with coefficients in the local system $\mrs L_\xi$. When $\mrs L_{\xi, \lbd}$ is a constant local system, this condition can be dropped by result of the author \cite{Pen25a} when $\bSh(\mbf V, \mdc K)$ is proper, and by Yang--Zhu \cite{Y-Z25} in general.\end{rem}

The rest of this section is devoted to the proof of the theorem. We modify the proof of \cite[Theorem~3.38]{LTXZZa}. By definition, we have an identification
\begin{equation*}
\mrs H_{\mdc K_0(\vn), \mfk m}\defining\Hom_{\mcl O}\paren{\etH^{d(\mbf V)}\paren{\bSh(\mbf V, \mdc K_0(\vn)), \mrs L_\xi}_{\mfk m}, \mcl O},
\end{equation*}
under which $\bb T_{\mfk m}$ is identified with the image of $\bb T^{m, \mfk d, \Pla}$ in $\End(\mrs H_{\mdc K_0(\vn), \mfk m})$ since $\etH^{d(\mbf V)}\paren{\bSh(\mbf V, \mdc K, \mrs L_\xi)_{\mfk m}}$ is $\mcl O$-torsion free.

We first construct a canonical homomorphism $\msf R_{\mrs S}\to \bb T_{\mfk m}$. It follows from semisimplicity of Hecke operators that $\bb T_{\mfk m}[\frac{1}{\ell}]$ is a reduced ring that is finite over $E$. As $\mrs H_{\mdc K_0(\vn), \mfk m}$ is a finite free $\mcl O$-module, $\bb T_{\mfk m}$ is also finite free over $\mcl O$. Via $\iota_\ell^{-1}$, every maximal ideal $x$ of $\bb T_{\mfk m}[\frac{1}{\ell}]$ gives rise to an $\mfk d$-REASDC representation $\Pi_x$ of $\GL_{2m}(\Ade_F)$ (see Definition~\ref{psoisieneifemiws}), satisfying that
\begin{enumerate}[(a)]
\item
the associated Galois representation $\rho_{\Pi_x, \iota_\ell}$ from Proposition~\ref{associateGlaosireifnies} is residually isomorphic to $\ovl\rho\otimes_\kappa\ovl{\bb F_\ell}$;
\item
$\Pi_x$ is the functorial lifting (see Definition~\ref{funcroliaureofnils}) of a cuspidal automorphic representation $\pi$ of $\SO(\mbf V)(\Ade_F)$ satisfying that $(\pi^\infty)^{\mdc K}$ appears nontrivially in $\etH^{d(\mbf V)}\paren{\bSh(\mbf V, \mdc K), \mrs L_\xi\otimes_{\mcl O, \iota_\ell^{-1}}\bb C}$. In particular, the Archimedean weights of $\Pi$ equals $\xi$ (see Definition~\ref{psoisieneifemiws}). Note that $\Pi_x$ is cuspidal because it is associated to an irreducible Galois representation.
\end{enumerate}
Denote by $\bb T_x$ the localization of $\bb T_{\mfk m}[\frac{1}{\ell}]$ at $x$. By an argument using Baire category theorem as in \cite[Lemma~2.1.5]{CHT08}, we may assume that the Galois representation $\rho_{\Pi_x, \iota_\ell}$ is actually valued in $G(\mcl O_x)$, where $\mcl O_x$ is the ring of integers of some finite extension of $\bb T_{\mfk m}[\frac{1}{\ell}]/x$. We may also assume that the reduction of $\rho_x$ modulo the maximal ideal of $\mcl O_x$ equals $\ovl\rho$. (Not just conjugate to $\ovl\rho$). Let $\bx A$ denote the subring of $\kappa\oplus\bplus_{x\in\Spec\bb T_{\mfk m}[\frac{1}{\ell}]}\mcl O_x$ consisting of elements $(a_{\mfk m}, a_x)$ \sut the reduction of $a_x$ modulo the maximal ideal of $\mcl O_x$ is $a_{\mfk m}$ for all $x\in\Spec\bb T_{\mfk m}[\frac{1}{\ell}]$. Then we have a continuous homomorphism
\begin{equation*}
\ovl\rho\oplus\bplus_{x\in\Spec\bb T_{\mfk m}[\frac{1}{\ell}]}\rho_x: \Gal_F\to G(\bx A).
\end{equation*}
It follows from Lemma~\ref{osoosieiwumfiev} that $\ovl\rho\oplus\bplus_{x\in\Spec\bb T_{\mfk m}[\frac{1}{\ell}]}\rho_x$ is conjugate to a continuous homomorphism
\begin{equation*}
\rho_{\mfk m}: \Gal_F\to G(\bb T_{\mfk m})
\end{equation*}
that lifts $\ovl\rho$.

We claim that $\rho_{\mfk m}$ satisfies the global deformation problem $\mrs S$. By property (b) and Arthur's multiplicity formula, $\Pi_{x, v}$ is unramified for every finite place $v$ of $F$ not in $\Pla^\bmin\cup\Pla^\lr$. So it follows from property (b) and Proposition~\ref{associateGlaosireifnies}(2) that $\rho_{\mfk m, v}$ is unramified for every $v$ not in $\Pla$. It follows from property (b), assumption (D0) and Proposition~\ref{associateGlaosireifnies}(3) that $\rho_{\mfk m, v}$ belongs to $\mrs D_v^{\FL, \square}$ for every $v\in\Pla(\ell)$. It follows from property (a)--(b), Proposition~\ref{associateGlaosireifnies}(2), Corollary~\ref{pssienifienfiwwss} and Arthur's multiplicity formula that $\rho_{\mfk m, v}$ belongs to $\mrs D_v^{\ram, \square}$ for every $v\in\Pla^\lr$. 

Therefore, by the universal property of $\msf R_{\mrs S}$, we obtain a canonical homomorphism
\begin{equation*}
\vp: \msf R_{\mrs S}\to \bb T_{\mfk m}
\end{equation*}
of rings over $\mcl O$. 

We now prove that $\vp$ is surjective. For this, we may assume that $E$ contains $\bb Q_{\ell^2}$ and every algebraic representation of $G$ is defined over $E$, so $\norml{v}$ is a square in $E^\times$ for every finite place of $F$ not in $\Pla$. Fix a lift $\rho_{\mrs S}$ of $\ovl\rho$ to $\msf R_{\mrs S}$. For each algebraic $\mcl O$-representation $(r_V, V)$ of $G$ and each finite place $v$ of $F$ not in $\Pla$, define
\begin{equation*}
t_{v, V}^{\msf R}\defining \tr(r_V(\rho_{\mrs S}(\phi_v)))\in \msf R_{\mrs S}
\end{equation*}
and
\begin{equation*}
t_{v, V}^{\bb T}\defining \tr(r_V(\rho_{\mfk m}(\phi_v)))\in \bb T_{\mfk m}.
\end{equation*}
By construction of $\phi$, we have
\begin{equation*}
\vp(t_{v, V}^{\msf R})=t_{v, V}^{\bb T}.
\end{equation*}
Thus it follows from the integral Satake isomorphism (cf.~\cite{Zhu24}) that $\bb T_{\mfk m}$ is generated by the elements $t_{v, V}^{\bb T}$. In particular, $\vp$ is surjective.

In the next two subsections, we will use the Taylor--Wiles patching argument to prove that $\vp$ is injective.

\subsection{Local computations}

In this subsection, we study local Langlands correspondence for principal series representations of split $p$-adic odd special orthogonal groups. We work in the following setting


\begin{setup}\label{osoitienifemifs}\enskip
\begin{itemize}
\item
Let $v$ be a finite place of $F$.
\item
Let $\SO_N$ be the split orthogonal group of rank $m$ over $\mcl O_v$.
\item
Let $T^{\SO}$ be a split maximal torus of $\SO_N$. Let $B^{\SO}$ be a Borel subgroup containing $T^{\SO}$, with unipotent radical $U_0^{\SO}$. Let $W$ be the Weyl group of $\SO_N$.
\item
Let $\ell$ be an odd rational prime not dividing the order of $W$, and suppose that $\norml{v}\equiv 1\modu\ell$.
\item
Let $T_\kappa$ be a split maximal torus contained in a Borel subgroup $B_\kappa$ of $G^\circ_\kappa$. Suppose the induced pinned root datum is dual to that of $\SO_N$, and let $\iota$ denote the induced isomorphisms $X_*(T^{\SO})\xr\sim X^*(T_\kappa)$ and $X^*(T^{\SO})\xr\sim X_*(T_\kappa)$, cf.~\cite[\S2.2.3]{Whi23a}.
\item
Let $\ovl g$ be an element of $T_\kappa(\kappa)$. Then we can define the following objects:
\begin{itemize}
\item
Let $M_{\ovl g}\defining Z_{G^\circ_\kappa}(\ovl g)$ be the centralizer of $\ovl g$ in $G^\circ_\kappa$.
\item
Let $S_\kappa$ be the identity component of the center of $M_{\ovl g}$, which is a subtorus of $T_\kappa$.
\item
Let $L_\kappa\defining Z_{G^\circ_\kappa}(S_\kappa)$ be the centralizer of $S_\kappa$ in $G^\circ_\kappa$, which is a Levi subgroup containing $M_{\ovl g}$.
\item
Let $P_\kappa$ denote the standard parabolic subgroup of $G^\circ_\kappa$ admitting $L_\kappa$ as a Levi factor.
\item
Let $P^{\SO}$ be a standard parabolic subgroup of $\SO_N\otimes\kappa_v$, with unipotent radical $U^{\SO}$ and a standard Levi subgroup $L^{\SO}$, such that $P^{\SO}$ (resp. $L^{\SO}$) corresponds to $P_\kappa\cap G'_\kappa$ (resp. $L_\kappa\cap G'_\kappa$) via $\iota$, cf.~\cite[\S2.2.3]{Whi23a}. Let $W_{L^{\SO}}$ be the Weyl group of $L^{\SO}$.
\item
let $\mdc K_0$ be the parahoric subgroup of $\SO_N(\mcl O_v)$ consisting of elements whose reduction is in $P^{\SO}(\kappa_v)$.
\item
Let $\Delta$ be the maximal quotient of $(L^{\SO}/[L^{\SO}, L^{\SO}])(\kappa_v)$ of $\ell$-power order.
\item
Let $\mdc K_1$ be the kernel of the composition map
\begin{equation*}
\mdc K_0\to P^{\SO}(\kappa)\to L^{\SO}(\kappa)\to \Delta,
\end{equation*}
which is a normal subgroup of $\mdc K_0$.
\item
Let $\ovl\chi: X_*(T^{\SO})\to \kappa^\times$ be the character corresponding to $\ovl g$ via $\iota$.
\item
Let $\mfk n_1$ (resp. $\mfk n_0$) denote the maximal ideal of $\mcl O[T^{\SO}(F_v)/(T^{\SO}(\mcl O_v)\cap\mdc K_1)]^{W_{L^{\SO}}}$ (resp. $\mcl O[T^{\SO}(F_v)/T^{\SO}(\mcl O_v)]^{W_{L^{\SO}}}$) corresponding to $\ovl\chi$.
\end{itemize}
\end{itemize}
\end{setup}


Suppose $(U_0^{\SO}, \psi)$ is a Whittaker datum, that is, $\psi$ is a nondegenerate character of $U_0^{\SO}$. Let $\Irr(\SO_N, T^{\SO})$ denote the set of irreducible smooth representations of $\SO_N(F_v)$ in the principal series. Let $\Phi_e(\SO_N, T^{\SO})$ denote the set of enhanced $L$-parameters for $\SO_N$ associated with principal series, that is, enhanced $L$-parameters whose cuspidal support is an enhanced $L$-parameter for $T^{\SO}$; cf.~\cite[p.271]{Sol25}. In~\cite{Sol25}, Solleveld uses the Whittaker datum $(U_0^{\SO}, \psi)$ to construct a canonical (bijective) local Langlands correspondence 
\begin{equation*}
\Irr(\SO_N, T^{\SO})\leftrightarrow \Phi_e(\SO_N, T^{\SO}): \quad \pi\mapsto (\phi_\pi, \rho_\pi), \quad \pi(\phi, \rho)\mapsfrom(\phi, \rho).
\end{equation*}
well-defined up to $G'$-conjugacy.

\begin{prop}\label{osseeifmeisw}
The local Langlands correspondence defined by Solleveld in \cite{Sol25} is compatible with that of Arthur for principal series representations of $\SO_N(F_v)$.
\end{prop}
\begin{proof}
We use \cite[Lemma~7.13]{Sol25}. In particular, it is known that the local Langlands correspondence 
\begin{equation*}
\pi\mapsto \phi_\pi
\end{equation*}
of Solleveld is uniquely determined by the following properties:
\begin{enumerate}
\item
This bijection is equivariant for the canonical actions of $\bx H^1(W_{F_v}, Z(G'))$.
\item
This bijection is compatible with the cuspidal support maps on both sides in the sense of \cite[Lemma~7.6]{Sol25}.
\item
Under this bijection, $\pi(\phi, \rho)$ is tempered if and only if $\phi$ is bounded.
\item
The bijection is compatible with the Langlands classification and (for tempered representations) with parabolic induction, in the sense of \cite[Lemma~7.11]{Sol25}.
\end{enumerate}
Thus it suffices to show that the local Langlands correspondence defined by Arthur in \cite{Art13} (and refined in \cite{Pen25a} when $N$ is even) satisfies these properties.

Property (1) holds because Arthur's packets are characterized by endoscopic character identities, and the action of $\bx H^1(W_{F_v}, Z(G'))$ preserves these identities.

Property (2--4) follow from \cite[Theorem~7.1.1]{Pen25a} for $N$ even and \cite[Theorem~3.2]{A-G17} for $N$ odd.
\end{proof}

\begin{prop}\label{sosieifeiifmeisw}
Let $\pi$ be a smooth irreducible representation of $\SO_N(F_v)$. Suppose that $(\pi^{\mfk p_1})_{\mfk n_1}\ne 0$. Let $(r, N)$ be the Weil--Deligne representation corresponding to $\pi$ under the local Langlands correspondence of Solleveld \cite{Sol25}. Then $N=0$.
\end{prop}
\begin{proof}
The argument is analogous to \cite[Theorem~2.4.26]{Whi23a}, where a similar result is established with the local Langlands correspondence of \cite{Sol25} replaced by that of \cite{ABPSa}.

To prove this claim, recall that the local Langlands correspondence in \cite{Sol25} is obtained on each principal series Bernstein component $\mfk s=[T^{\SO}, \chi_0]_{\SO_N}$ via a series of canonical bijections
\begin{equation*}
\Irr(\SO_N, T^{\SO})^{\mfk s}\leftrightarrow \Irr(\End_{\SO_N}(\Pi_{\mfk s})^\op)\leftrightarrow\Irr(\mcl H(\mfk s)^\op)\leftrightarrow \Irr(\mcl H(\mfk s^\vee, \norml{v}^{1/2}))\leftrightarrow \Phi_e(\SO_N)^{\mfk s^\vee};
\end{equation*}
cf.~\cite[\S7]{Sol25}. Here
\begin{itemize}
\item
$\mfk s^\vee$ is the dual Bernstein component in $\Phi_e(T^{\SO})$ associated with $\mfk s$ via the local Langlands correspondence for tori; cf.~\cite[\S5]{Sol25};
\item
$\Pi_{\mfk s}$ is a progenerator of the Bernstein block $\Irr(\SO_N, T^{\SO})^{\mfk s}$ in the principal series, given by
\begin{equation*}
\Pi_{\mfk s}\defining \bx I_{B^{\SO}}^{\SO_N}(\ind_{T^{\SO}_\cpct}^{T^{\SO}(F_v)}(\chi_c)),
\end{equation*}
where $T^{\SO}_\cpct$ is the unique maximal compact subgroup of $T^{\SO}(F_v)$, and $\chi_c\defining \chi_0|_{T^{\SO}_\cpct}$;
\item
$\mcl H(\mfk s)=\mcl H(\mfk s)^\circ\rtimes\Gamma_{\mfk s}$ is certain (extended) affine Hecke algebra that is (canonically when given the Whittaker datum) isomorphic to $\End_{\SO_N}(\Pi_{\mfk s})$, cf.~\cite[Theorem~2.7]{Sol25}; and
\item
$(\mcl H(\mfk s^\vee, \norml{v}^{1/2})$ is certain (extended) affine Hecke algebra from \cite{AMS24} attached to the Bernstein component $\Phi_e(\SO_N)^{\mfk s^\vee}$.
\end{itemize}
The affine Hecke algebra $\mcl H(\mfk s)$ contains a finite-dimensional Iwahori--Hecke subalgebra
\begin{equation*}
\mcl H(W(R_{\mfk s}^\vee), \norml{v}^\lbd),
\end{equation*}
where $R_{\mfk s}^\vee$ is the root datum attached to the Bernstein component $\mfk s$ and $\lbd: R_{\mfk s}^\vee/W_{\mfk s}\to \bb N$ is the label function; cf.\cite[p.272]{Sol25}.

By \cite[Proposition~2.4.22]{Whi23a}, $\pi$ is contained in a principal series Bernstein component $\mfk s$. By \cite[Lemma~2.4.25]{Whi23a}, the $\mcl H(W(R_{\mfk s}^\vee), \norml{v}^\lbd)$-module $\mcl M_\pi$ corresponding to $\pi$ contains the one-dimensional character on which each simple generator $T_{s_\alpha}$ acts by $\norml{v}$ (i.e., the $\norml{v}$-trivial character).

Thus it follows from \cite[Proposition~10.1]{ABPSa} that the Kazhdan--Lusztig triple $(\sigma, x, \rho^\circ)$ attached to $\mcl M_\pi$ satisfies $x=1$.

Fix an element $u$ in the maximal compact real subtorus of $T_{\mfk s^\vee}$. Using graded Hecke algebras, Solleveld attached to $\mcl M_\pi$ a triple $(\sigma, y, \rho^\circ)$ in which
\begin{enumerate}
\item
$\sigma\in\Lie(G'_u)$ is semisimple, 
\item
$y\in \Lie(G'_u)$ is unipotent,
\item
$[\sigma, y]=-\log(\norml{v})y$, and
\item
$\rho^\circ$ is an irreducible representation of $\pi_0(Z_{G_u'}/Z(G_u'))$ appearing in certain Springer-type homology;
\end{enumerate}
cf.\cite[p.310]{Sol25}. It follows from the proof of \cite[Proposition~6.5]{Sol25} that $(\exp(\sigma), \exp(y), \rho^\circ)$ is the Kazhdan--Lusztig triple associated to $\mcl M_\pi$. In particular, $y=0$.

On the other hand, it follows from the proof of \cite[Lemma~6.10]{Sol25} on page~316 that
\begin{equation*}
y=\log\phi_\pi(1, \begin{bmatrix}1 & 1\\ & 1\end{bmatrix}).
\end{equation*}
Thus $\phi_\pi$ is trivial on the $\SL_2$-factor, and the assertion follows.
\end{proof}

Combining Proposition~\ref{sosieifeiifmeisw} and Proposition~\ref{osseeifmeisw}, we get the following corollary.

\begin{cor}\label{psosieiifmfieifs}
Let $\pi$ be a smooth irreducible representation of $\SO_N(F_v)$. Suppose that $(\pi^{\mfk p_1})_{\mfk n_1}\ne 0$. Let $(r, N)$ be the Weil--Deligne representation corresponding to $\pi$ under the local Langlands correspondence of Solleveld \cite{Sol25}. Then $N=0$.
\end{cor}

\subsection{Taylor--Wiles patching argument}

We will use the Taylor--Wiles patching argument following \cite{CHT08}, \cite{Tho12}, and \cite{Whi23a}. Set $\Pla^0\defining\Pla^\bmin\cup\Pla^\lr\cup\Pla(\ell)$. 
To continue the proof of Theorem~\ref{peeoieowivmincs}, we may replace $E$ by a finite unramified extension in $\ovl{\bb Q_\ell}$. Thus we may assume that $\kappa$ contains all eigenvalues of matrices in $\ovl\rho(\Gal_F)$. For each compact open subset $\mdc K'$ of $\mdc K$, we set
\begin{equation*}
\mrs H_{\mdc K'}\defining \Hom_{\mcl O}\paren{\etH^{d(\mbf V)}\paren{\bSh(\mbf V, \mdc K'), \mrs L_\xi}, \mcl O}.
\end{equation*}

We recall the definition of Taylor--Wiles places and Taylor--Wiles data, cf.~\cite[Definitions~2.2.8--2.2.9]{Whi23a}.

\begin{defi}
A \emph{Taylor--Wiles place} for the global deformation problem $\mrs S$ is a finite place $v$ of $F$ not in $\Pla$ that is split in $F(\sqrt{\mfk d})$, \sut $\norml{v}\equiv 1\modu \ell$ and $\ovl g_v\defining \ovl\rho(\phi_v)\in G^\circ(\kappa)$ is semisimple. Such a place is said to be of level $n\ge 1$ if $\norml{v}\equiv 1\modu{\ell^n}$. 

If $v$ is a Taylor--Wiles place for $\mrs S$, we let $M_{\ovl g_v}\defining Z_{G^\circ_\kappa}(\ovl g_v)$ denote the centralizer of $\ovl g_v$ in $G^\circ_\kappa$. 

A \emph{Taylor--Wiles datum} is a pair
\begin{equation*}
\big(\Pla^\TW, \{(T_v, B_v)\}_{v\in\Pla^\TW}\big)
\end{equation*}
in which
\begin{enumerate}
\item
$\Pla^\TW$ is a finite set of Taylor--Wiles places of $F$; and
\item
for each $v\in \Pla^\TW$, $T_v\le G^\circ_\kappa$ is a split maximal torus containing $Z(M_{\ovl g_v})^\circ$, and $B_v$ is a Borel subgroup of $G^\circ_\kappa$ containing $T_v$.
\end{enumerate}
\end{defi}

\begin{prop}\label{lOIEYTleosoiwleiifns}
For each Taylor--Wiles place $v\in \Pla^\TW$, there is a local deformation condition $\mrs D_v^\square$ of $\ovl\rho_v$ (\wrt $\chi=\ve_\ell^{2-N})$ representing liftings $\rho_v$ of $\ovl\rho_v$ to objects $\bx A$ of $\mrs C_{\mcl O}$ satisfying $\rho_v(I_v)\subset Z(M_{g_v})^\circ(\bx A)$. Here $g_v=\rho_v(\phi_v)$ and $M_{g_v}$ is the subgroup constructed in \cite[Theorem~2.1.14]{Whi23a}, which is a lift of $M_{\ovl g_v}$.

Moreover, we have
\begin{equation*}
\dim_\kappa\bx L(\mrs D_v^\square)=\dim_\kappa \bx H^0(F_v, \ad^0(\ovl\rho))+\dim_\kappa Z(M_{\ovl\rho(\phi_v)})-1.
\end{equation*}
\end{prop}
\begin{proof}
This follows from \cite[Definition~2.2.10 and Corollary~2.2.13]{Whi23a} applied to $G^\circ$. The last equation follows from the computation following \cite[Corollary~2.2.13]{Whi23a}.
\end{proof}

\begin{defi}\label{ososieueunufuess}
For each Taylor--Wiles datum $\big(\Pla^\TW, \{(T_v, B_v)\}_{v\in\Pla^\TW}\big)$, we obtain another global deformation problem
\begin{equation*}
\mrs S(\Pla^\TW)\defining\paren{\ovl\rho, \ve_\ell^{2-N}, \Pla^0\cup\Pla^\TW, \{\mrs D_v^\square\}_{v\in\Pla\cup\Pla^\TW}}
\end{equation*}
where $\mrs D_v^\square$ is the same as in $\mrs S$ for $v\in\Pla^0$, and for $v\in\Pla^\TW$, $\mrs D_v^\square$ is the local lifting condition of $\ovl\rho_v$ defined in Proposition~\ref{lOIEYTleosoiwleiifns}. Moreover, for each place $v\in \Pla^\TW$, we adopt the notation from Setup~\ref{osoitienifemifs} with subscript ``$v$'' added to indicate the place $v$, and with $(T_\kappa, B_\kappa)=(T_v, B_v)$ and $\ovl g=\ovl g_v$. In particular, we have compact subgroups
\begin{equation*}
\mdc K_{v, 1}\le \mdc K_{v, 0}\le \mdc K_v=\SO_N(\mcl O_v),
\end{equation*}
\sut $\mdc K_{v, 1}$ is a normal subgroup of $\mdc K_{v, 0}$ with quotient $\Delta_v$.

We then set
\begin{itemize}
\item
$\Delta_{\Pla^\TW}\defining\prod_{v\in\Pla^\TW}\Delta_v$;
\item
$\mfk a_{\Pla^\TW}$ to be the augmentation ideal of $\mcl O[\Delta_{\Pla^\TW}]$ over $\mcl O$ generated by $\delta-1$ for all $\delta\in \Delta_{\Pla^\TW}$;
\item
$\mfk m_{\Pla^\TW}\defining\mfk m\cap\bb T^{m, \mfk d, \Pla\cup\Pla^\TW}$;
\item
$\mfk n_i=\prod_{v\in\Pla^\TW_n}\mfk n_{v, i}$ to be the maximal ideal of $\prod_{v\in\Pla^\TW_n}\mcl O[T^{\SO}(F_v)/(T^{\SO}(\mcl O_v)\cap\mdc K_1)]^{W_{L^{\SO}}}$ (resp. $\prod_{v\in\Pla^\TW_n}\mcl O[T^{\SO}(F_v)/T^{\SO}(\mcl O_v)]^{W_{L^{\SO}}}$) corresponding to $(\ovl g_v)_{v \in \Pla^\TW_n}$, for each $i\in\{0, 1\}$; and
\item
\begin{equation*}
\mdc K_i(\Pla^\TW)\defining \prod_{v\not\in\Pla^\TW}\mdc K_v\times\prod_{v\in\Pla^\TW}\mdc K_{v, i}
\end{equation*}
for each $i\in\{0, 1\}$, which are compact open subgroups of $\mdc K$.
\end{itemize}
In particular, $\mdc K_1(\Pla^\TW)$ is a normal subgroup of $\mdc K_0(\Pla^\TW)$, and we have a canonical isomorphism
\begin{equation}\label{osopsieueuriewaq}
\mdc K_0(\Pla^\TW)/\mdc K_1(\Pla^\TW)\cong \Delta_{\Pla^\TW}.
\end{equation}
By \eqref{osopsieueuriewaq}, $\mrs H_{\mdc K_1(\Pla^\TW)}$ is canonically a module over $\mcl O[\Delta_{\Pla^\TW}]$.
\end{defi}

\begin{lm}\label{oooootitiiieeos}
Let the situation be as in \textup{Theorem~\ref{peeoieowivmincs}}. If $\big(\Pla^\TW, \{(T_v, B_v)\}_{v\in\Pla^\TW}\big)$ is a Taylor--Wiles datum, then the $\mcl O[\Delta_{\Pla^\TW}]$-module $\mrs H_{\mdc K_1(\Pla^\TW), \mfk m_{\Pla^\TW}}$ is finite and free. Moreover, the canonical map
\begin{equation*}
\mrs H_{\mdc K_1(\Pla^\TW), \mfk m_{\Pla^\TW}}/\mfk a_{\Pla^\TW}\to \mrs H_{\mdc K_0(\Pla^\TW), \mfk m_{\Pla^\TW}}
\end{equation*}
is an isomorphism.
\end{lm}
\begin{proof}
We modify the proof of \cite[Lemma~3.6.5]{LTXZZa}. For each smooth schemes $X$ over $\bb C$ and each Abelian group $A$, we denote by $X^\an$ the analytification of $X$, and by $\bx C_\bullet(X^\an, A)$ the complex of singular chains of $X^\an$ with coefficients in $A$.

For each $i\in\{0, 1\}$ and each positive integer $m$, let $\mdc K_i^m(\Pla^\TW)$ be the kernel of the composite map
\begin{equation*}
\mdc K_i(\Pla^\TW)\to \mdc K\to \SO(\Lbd)(\mcl O_F/\ell^m).
\end{equation*}
Then $\mdc K_i^m(\Pla^\TW)$ acts trivially on $L_\xi/\ell^m$. Hence $\mrs L_\xi\otimes_{\mcl O}\mcl O/\ell^m$ is constant on $\bSh(\mbf V, \mdc K_i^m(\Pla^\TW))$. It follows from (D4) that there is a canonical isomorphism
\begin{equation*}\label{iepmvimvneif}
\etH^{d(\mbf V)}\paren{\bSh(\mbf V, \mdc K_i(\Pla^\TW)), \mrs L_\xi}_{\mfk m_{\Pla^\TW}}\cong \etH^{d(\mbf V)}\paren{\bSh(\mbf V, \mdc K_i^m(\Pla^\TW)), L_\xi\otimes_{\mcl O}\mcl O/\ell^m}_{\mfk m_{\Pla^\TW}}^{\SO(\Lbd)(\mcl O_F/\ell^m)}
\end{equation*}
of $\mcl O$-torsion free modules for every $i\in\{0, 1\}$. By Artin's comparison theorem between singular cohomologies and \eTale cohomologies, the dual complex
\begin{equation*}
\Hom_{\mcl O}\paren{\bx C_\bullet(\bSh(\mbf V, \mdc K_i^m(\Pla^\TW))^\an, L_\xi\otimes_{\mcl O}\mcl O/\ell^m), \mcl O/\ell^m}
\end{equation*}
calculates $\etH^\bullet\paren{\bSh(\mbf V, \mdc K_i^m(\Pla^\TW)), L_\xi\otimes_{\mcl O}\mcl O/\ell^m}$. It follows from (D4) again that
\begin{equation*}
\bx H_d\paren{\bSh(\mbf V, \mdc K_i^m(\Pla^\TW))^\an, L_\xi\otimes_{\mcl O}\mcl O/\ell^m}=0
\end{equation*}
unless $d=d(\mbf V)$.

On the other hand, $\mdc K_i^m(\Pla^\TW)$ is neat for every $i\in\{0, 1\}$, which implies that $t^{-1}\SO(\mbf V)(F)t\cap\mdc K_i^m(\Pla^\TW)$ has no torsion elements for every $t\in \SO(\mbf V)(\Ade_F^\infty)$. Using triangularizations of Shimura varieties as in the proof of~\cite[Lemma~6.9]{K-T17}, we see that $\bx C_\bullet(\bSh(\mbf V, \mdc K_1^m(\Pla^\TW))^\an, L_\xi\otimes_{\mcl O}\mcl O/\ell^m)$ is quasiisomorphic to a perfect complex of free $(\mcl O/\ell^m)[\Delta_{\Pla^\TW}]$-modules; and there is a canonical isomorphism
\begin{equation*}
\bx C_\bullet(\bSh(\mbf V, \mdc K_1^m(\Pla^\TW))^\an, L_\xi\otimes_{\mcl O}\mcl O/\ell^m)/\mfk a_{\Pla^\TW}\cong \bx C_\bullet(\bSh(\mbf V, \mdc K_0^m(\Pla^\TW))^\an, L_\xi\otimes_{\mcl O}\mcl O/\ell^m)
\end{equation*}
of complexes of $(\mcl O/\ell^m)[\bb T^{m, \mfk d,\Pla\cup\Pla^\TW}]$-modules.

After localizing at $\mfk m_{\Pla^\TW}$ and taking $\SO(\Lbd)(\mcl O_F/\ell^m)$-invariants, it follows that the canonical map
\begin{equation*}
\etH^{d(\mbf V)}\paren{\bSh(\mbf V, \mdc K_0^m(\Pla^\TW)), L_\xi\otimes_{\mcl O}\mcl O/\ell^m}/\mfk a_{\Pla^\TW}\to \etH^{d(\mbf V)}\paren{\bSh(\mbf V, \mdc K_1^m(\Pla^\TW)), L_\xi\otimes_{\mcl O}\mcl O/\ell^m}
\end{equation*}
is an isomorphism for every positive integer $m$. Then the lemma follows by passing to the limit for $m$.
\end{proof}

We now discuss the existence of Taylor--Wiles data, using the notion of $G^\circ$-adequate subgroups (see Definition~\ref{otieinnvnuducms}).

\begin{lm}\label{osissssneieiroes}
Let the situation be as in \textup{Theorem~\ref{peeoieowivmincs}}. Let $\Xi$ be a subset of $\Pla^0$.
For every integer $b\ge \dim_\kappa\bx H^1_{\mrs L^\perp, \Xi}(\Gal_{F, \Pla}, \ad^0(\ovl\rho)(1))$ (see \textup{Equation~\eqref{psosiieuueufemes}}) and every positive integer $n$, there is a Taylor--Wiles datum $\big(\Pla^\TW_n, \{(T_v, B_v)\}_{v\in\Pla^\TW_n}\big)$ satisfying
\begin{enumerate}
\item
$\#\Pla^\TW_n=b$;
\item
each $v\in\Pla^\TW_n$ is a Taylor--Wiles place of level $n$; and
\item
$\msf R^{\square_\Xi}_{\mrs S(\Pla^\TW_n)}$ can be topologically generated over $\msf R^\square_{\mrs S, \Xi}$ by
\begin{equation*}
\sum_{v\in \Pla_n^{\TW}}(\dim_\kappa Z(M_{\ovl\rho(\phi_v)})-1)-\sum_{v\in\Xi\cap\Pla(\ell)}[F_v: \bb Q_\ell](\dim G_\kappa-\dim B_\kappa)
\end{equation*}
elements.
\end{enumerate}
\end{lm}
\begin{proof}
We modify the proof of \cite[Theorem~4.4]{Tho12}. 
By Proposition~\ref{olaososiefefjeimss}, Proposition~\ref{psisiidimfimies}, and Proposition~\ref{ososoidvienfiemsws}, for each place $v\in \Pla$,
\begin{equation*}
\dim_\kappa\bx L(\mrs D_v)-\dim_\kappa\bx H^0(F_v, \ad^0(\ovl\rho)=
\begin{cases}
[F_v: \bb Q_\ell](\dim G_\kappa-\dim B_\kappa) & \If v\in\Pla(\ell)\\
0 & \If v\notin\Pla(\ell)
\end{cases}.
\end{equation*}
For any Taylor--Wiles datum $\big(\Pla^\TW, \{(T_v, B_v)\}_{v\in\Pla^\TW}\big)$, it follows from the Wiles' formula (see Lemma~\ref{sliimiifhiemiws}(5)), \ref{lOIEYTleosoiwleiifns}, and the proof of~\cite[Corollary~2.2.12]{CHT08} that $\msf R_{\mrs S(\Pla^\TW)}^{\square_\Xi}$ can be generated over $\msf R^\square_{\mrs S, \Xi}$ by
\begin{align}\label{pllpueunvnueimsiw}
\dim_\kappa\bx H^1_{\mrs S(\Pla^\TW), \Xi}(\Gal_{F, \Pla}, \ad^0(\ovl\rho))&=\dim_\kappa\bx H^1_{\mrs L(\Pla^\TW)^\perp, \Xi}(\Gal_{F, \Pla}, \ad^0(\ovl\rho)(1))+\sum_{v\in \Pla^{\TW}}(\dim_\kappa Z(M_{\ovl\rho(\phi_v)})-1)\\
&-\sum_{v\in\Xi\cap\Pla(\ell)}[F_v:\bb Q_\ell](\dim G_\kappa- \dim B_\kappa)-\dim_\kappa\bx H^0(\Gal_{F, \Pla}, \ad^0(\ovl\rho)(1))
\end{align}
elements.
It follows from \cite[Proposition~2.2.18 and Lemma~2.2.20]{Whi23a} that there exists a set $\Pla^\TW_n$ of Taylor--Wiles places satisfying Conditions (1)--(2), such that
\begin{equation*}
\dim_\kappa\bx H^1_{\mrs L(\Pla^\TW)^\perp, \Xi}(\Gal_{F, \Pla}, \ad^0(\ovl\rho)(1))=0.
\end{equation*}
Thus the lemma follows.
\end{proof}
\begin{cor}\label{osisneieiroes}
Let the situation be as in \textup{Theorem~\ref{peeoieowivmincs}}. Let $\Xi$ be a subset of $\Pla^0$. Then there exists a positive integers $t, b\in\bb Z_+$, \sut for every positive integer $n$, there is a Taylor--Wiles datum $\big(\Pla^\TW_n, \{(T_v, B_v)\}_{v\in\Pla^\TW_n}\big)$ satisfying
\begin{enumerate}
\item
$\#\Pla^\TW_n=b$;
\item
$\norml{v}\equiv 1\modu{\ell^n}$ for each $v\in\Pla^\TW_n$.
\item
$\msf R^{\square_\Xi}_{\mrs S(\Pla^\TW_n)}$ can be topologically generated over $\msf R^\square_{\mrs S, \Xi}$ by
\begin{equation*}
g_{t, \Xi}\defining t-\sum_{v\in\Xi\cap\Pla(\ell)}[F_v: \bb Q_\ell](\dim G_\kappa-\dim B_\kappa)
\end{equation*}
elements. 
\item
$\Delta_{\Pla^\TW_n}$ is an Abelian $\ell$-group with $t$ cyclic factors.
\end{enumerate}
\end{cor}
\begin{proof}
This follows from Lemma~\ref{osissssneieiroes} and the Pigeonhole principle.
\end{proof}

We now prove our main result.

\begin{proof}[Proof of \textup{Theorem~\ref{peeoieowivmincs}}]
We fix a universal lifting
\begin{equation*}
\rho_{\mrs S}: \Gal_F\to G(\msf R_{\mrs S}).
\end{equation*}
Let $\Xi$ be a fixed subset of $\Pla^0$. We fix positive integer $b, t\in\bb Z_+$ and a Taylor--Wiles datum $(\Pla_n^\TW, \{(T_v, B_v)\}_{v\in\Pla^\TW_n})$ of size $b$ for each positive integer $n$ from Lemma~\ref{osisneieiroes}. For each $v\in \Pla^\TW_n$, we
\begin{itemize}
\item
let $\Art_v: F_v^\times\xr\sim \Gal_{F_v}^\ab$ be the local Artin map; and
\item
let $\varpi_v\in F_v$ be a uniformizer, \sut $\Art_v(\varpi_v)$ coincides with the image of $\phi_v^{-1}$ in $\Gal_{F_v}^\ab$.
\end{itemize}

We adopt notation from Definition~\ref{ososieueunufuess}, and
\begin{itemize}
\item
set $\mrs H\defining \mrs H_{\mdc K_0(\vn), \mfk m}$;
\item
let $\mrs H_{i, n}$ be the localized module $\paren{\mrs H_{\mdc K_i(\Pla^\TW_n), \mfk m_{\Pla^\TW_n}}}_{\mfk n_i}$ for each $n\in\bb Z_+$ and $i\in\{0, 1\}$; and
\item
let $\bb T_{i, n}$ be the image of $\bb T^{m, \mfk d, \Pla^\TW_n}$ in $\End_{\mcl O}(\mrs H_{i, n})$, for each $n\in\bb Z_+$ and $i\in\{0, 1\}$. 
\end{itemize}

For each $n\in\bb Z_+$, since the canonical map $\mrs H\to \mrs H_{\mdc K, \mfk m_{\Pla^\TW_n}}$ is an isomorphism, we obtain canonical surjections
\begin{equation}\label{ososiuuuturneus}
\bb T_{1, n}\surj \bb T_{0, n}\surj\bb T_{\mfk m}
\end{equation}
of rings over $\mcl O$. Similar to the construction of $\rho_{\mfk m}$, we obtain liftings $\rho_{i, n}$ of $\ovl\rho$ to $\bb T_{i, n}$ that are compatible with the maps in \eqref{ososiuuuturneus}.

Analogous to \cite[Propositions~5.9, 5.12]{Tho12}, we claim that
\begin{enumerate}
\item
$\rho_{1, n}$ satisfies the global deformation problem $\mrs S(\Pla^\TW_n)$ (see Definition~\ref{ososieueunufuess}). Moreover, for each $v\in\Pla^\TW_n$, the induced map $\Delta_v\to \bb T_{1, n}$ as defined in \cite[pp.59--60]{Whi23a} coincides with the scalar action of $\Delta_v$ on $\mrs H_{1, n}$ arising from viewing $\Delta_v=T^{\SO}(\mcl O_v)/(T^{\SO}(\mcl O_v)\cap \mdc K_{v, 1})$ as Hecke operators.
\item
The natural map
\begin{equation*}
\mrs H_{0, n}\to \mrs H
\end{equation*}
is an isomorphism for every $n\in\bb Z_+$. In particular, the canonical homomorphism $\bb T_{0, n}\to \bb T_{\mfk m}$ is an isomorphism, and $\rho_{0, n}$ are equivalent to $\rho_{\mfk m}$ as liftings of $\ovl\rho$.
\end{enumerate}

The first claim follows from applying \cite[Proposition~2.5.3]{Whi23a} to the Galois representation associated to the maximal ideals of $\bb T_{1,n}[1/\ell]$, by the construction of $\rho_{1,n}$. 
Here we note (\ref{iseihienfies}) and Corollary~\ref{psosieiifmfieifs} together guarantee that the hypotheses of \cite[Proposition~2.5.3]{Whi23a} hold, except for the fact that the local Galois representations are only valued in $G^\circ$ rather than $G'$, which does not affect the proof.
The second claim follows from applying \cite[Corollary 2.4.14]{Whi23a} to the module $\mrs H_{\mdc K_{\mathrm{Iw}}(\Pla^\TW_n), \mfk m_{\Pla^\TW_n}}$, where $\mdc K_{\mathrm{Iw}}(\Pla^\TW_n)$ is defined such that $\mdc K_{\mathrm{Iw}}(\Pla^\TW_n)_v = \mathrm{Iw}_v$ for  $v\in \Pla^\TW_n$ and $\mdc K_{\mathrm{Iw}}(\Pla^\TW_n)_v = \mdc K_v$ for $v \not\in \Pla^\TW_n$. Here $\mathrm{Iw}_v \subset \SO(\mcl O_v)$ is the standard Iwahori subgroup corresponding to $B^{\SO}$.

It follows from (1) above that $\rho_{1, n}$ satisfies the global deformation problem $\mrs S(\Pla^\TW_n)$, which induces a canonical surjective homomorphism
\begin{equation*}
\vp_n: \msf R_{\mrs S(\Pla^\TW_n)}\to \bb T_{1, n}.
\end{equation*}
of $\mcl O[\Delta_{\Pla^\TW_n}]$-modules.

By (2) and Lemma~\ref{oooootitiiieeos}, we obtain a canonical commutative diagram
\begin{equation*}
\begin{tikzcd}[sep=large]
\msf R_{\mrs S(\Pla^\TW_n)}/\mfk a_{\Pla^\TW_n}\ar[r, "\sim"]\ar[d, "\vp_n/\mfk a_{\Pla^\TW_n}"] &\msf R_{\mrs S}\ar[d, "\vp"]\\
\bb T_{1, n}/\mfk a_{\Pla^\TW_n}\ar[r, "\sim"] & \bb T_{\mfk m}
\end{tikzcd}
\end{equation*}
of ring over $\mcl O$, where the horizontal maps are isomorphisms.

For each $n\in\bb Z_+$, we choose a universal lifting
\begin{equation*}
\rho_{\mrs S(\Pla_n^\TW)}: \Gal_F\to G(\msf R_{\mrs S(\Pla_n^\TW)})
\end{equation*}
of $\rho_{\mrs S}$ to $\msf R_{\mrs S(\Pla_n^\TW)}$. By Proposition~\ref{peoieutnmeumes}(3), we obtain isomorphisms
\begin{equation*}
\msf R_{\mrs S}\hat\otimes\mrs T_\Xi\cong \msf R^{\square_\Xi}_{\mrs S}, \quad \msf R_{\mrs S(\Pla^\TW)}\hat\otimes\mrs T_\Xi\cong \msf R^{\square_\Xi}_{\mrs S(\Pla^\TW)}
\end{equation*}
of rings over $\mcl O$. In particular, we can choose a surjective local homomorphism $\msf R_{\mrs S}^{\square_\Xi}\to \msf R_{\mrs S}$, making $\msf R_{\mrs S}$ an algebra over $\msf R^\square_{\mrs S, \Xi}$.

We set
\begin{equation*}
\msf S_\infty\defining\mrs T_\Xi[[Y_1, \ldots, Y_t]],
\end{equation*}
where $\mrs T_\Xi$ is the coordinate ring of the formal smooth group $\prod_{v\in\Xi}(G')^\wedge$ over $\mcl O$. Let $\mfk a_\infty\subset\msf S_\infty$ be the augmentation ideal over $\mcl O$. Set 
\begin{equation*}
\msf R_\infty\defining\msf R^\square_{\mrs S, \Xi}[[Z_1, \ldots, Z_{g_{t, \Xi}}]],
\end{equation*}
where $g_{t, \Xi}$ is the integer appearing in Lemma~\ref{osisneieiroes}. Applying the usual patching lemma (see the proof of \cite[Thm~3.6.1]{LGG11} or \cite[Lemma~6.10]{Tho12}), it follows that
\begin{itemize}
\item
there exists a local homomorphism $\msf S_\infty\to \msf R_\infty$ over $\mcl O$ \sut we have a surjective homomorphism $\msf R_\infty/\mfk a_\infty\msf R_\infty\to \msf R_{\mrs S}$ of rings over $\msf R^\square_{\mrs S, \Xi}$.
\item
there exists a $\msf R_\infty$-module $\mrs H_\infty$ \sut
\begin{itemize}
\item
$\mrs H_\infty$ is finite free over $\msf S_\infty$, and
\item
$\mrs H_\infty/\mfk a_\infty\mrs H_\infty$ is isomorphic to $\mrs H$ as $\msf R_\infty/\mfk a_\infty\msf R_\infty$-modules.
\end{itemize}
\end{itemize}
In particular, we have
\begin{equation*}
\depth_{\msf R_\infty}\mrs H_\infty\ge \depth(\msf S_\infty)=1+\dim G'_\kappa\cdot\#\Xi+t.
\end{equation*}
On the other hand, it follows from Proposition~\ref{olaososiefefjeimss}, Proposition~\ref{psisiidimfimies}, and Proposition~\ref{ososoidvienfiemsws} 
that $\msf R_{\mrs S, \Xi}^\square$ is a formal power series over $\mcl O$ in
\begin{equation*}
\dim G'_\kappa\cdot\#\Xi+\sum_{v\in\Xi\cap\Pla(\ell)}[F_v: \bb Q_\ell](\dim G_\kappa-\dim B_\kappa)
\end{equation*}
variables. It follows that $\msf R_\infty$ is a regular local ring of dimension
\begin{equation*}
1+\dim G'_\kappa\cdot\#\Xi+\sum_{v\in\Xi\cap\Pla(\ell)}[F_v: \bb Q_\ell](\dim G_\kappa-\dim B_\kappa)+g_{t, \Xi}=1+\dim G'_\kappa\cdot\#\Xi+t.
\end{equation*}
Thus it follows from the Auslander--Buchsbaum theorem that $\mrs H_\infty$ is a finite free $\msf R_\infty$-module.

Then the $\msf R_\infty$-module $\mrs H=\mrs H_\infty/\mfk a_\infty\mrs H_\infty$ is finite free over $\msf R_\infty/\mfk a_\infty\msf R_\infty$. But the action of $\msf R_\infty/\mfk a_\infty\msf R_\infty$ on $\mrs H$ factors through the surjective homomorphism $\msf R_\infty/\mfk a_\infty\msf R_\infty\to \msf R_{\mrs S}$, implying that this surjection must be an isomorphism. Thus $\mrs H$ is a finite free $\msf R_{\mrs S}$-module. As a result, the map $\vp: \msf R_{\mrs S}\to \bb T_{\mfk m}$ is an isomorphism.
\end{proof}

\section{Rigidity}

\subsection{Rigidity of odd symmetric powers of elliptic curves}

In this subsection, we study rigidity of reductions of odd symmetric powers of Galois representation attached to elliptic curves.

Let $A$ be an elliptic curve over $F$. For every rational prime $\ell$, we fix an isomorphism $\etH^1(A_{\ovl F}, \bb Z_\ell)\cong \bb Z_\ell^{\oplus 2}$, and let
\begin{equation*}
\rho_{A, \ell}: \Gamma_F\to \GL(\etH^1(A_{\ovl F}, \bb Z_\ell))\cong \GL_2(\bb Z_\ell)
\end{equation*}
be the continuous homomorphism attached to $A$. Then we obtain a continuous homomorphism
\begin{equation*}
\rho_{A, m, \ell}\defining \Sym^{2m-1}(\rho_{A, \ell}): \Gal_F\to \GSp_{2m}(\bb Z_\ell).
\end{equation*}
Recall that $\GSp_{2m}$ is defined by the standard alternating pairing on $\bb Z_\ell^{\oplus(2m)}$ given by the matrix $J_{2m}'$; see Setup~\ref{seutpeoeifmeis}.

When $v$ is a finite place of $F$ and $\ell$ is a rational prime, let $G$ be the smooth group scheme $\GSp_{2m}$ over $\bb Z_\ell$, and we work in Setup~\ref{osisienieeimifes} with $\Gamma=\Gal_{F_v}$ and $\chi=\ve_\ell^{1-2m}$.

\begin{prop}
Let $v$ be a finite place of $F$. For all but finitely many rational primes $\ell>2m$, every lifting of $\ovl\rho_{A, m, \ell}|_{\Gal_{F_v}}$ 
to an object $\bx R$ of $\mrs C_{\bb Z_\ell}$ (\wrt to $\chi=\ve_\ell^{1-2m}$) is minimally ramified in the sense of \textup{Definition~\ref{minaimlsoieieres}}.
\end{prop}
\begin{proof}
Choose a finite totally ramified extension $F'_v$ of $F_v$ inside $\ovl{F_v}$ \sut $A'\defining A\otimes_FF_v'$ has either good or multiplicative reduction. We adopt the notation of \S\ref{sosofieifiehfemis}, and let $\bx T'_v$ be the image of the subgroup $\Gal(\ovl{F_v}/F'_v)$ in $\bx T_v\defining \Gal_{F_v}/P_v$. We fix an isomorphism between $\bx T_v$ and the $\norml{v}$-tame group $\bx T_{\norml{v}}$. Assume that $\ell\nmid [F_v': F_v]$, so that $\bx T_v'=\bx T_v$. Choose a finite unramified extension $E$ of $\bb Q_\ell$ contained in $\ovl{\bb Q_\ell}$ with ring of integers $\mcl O$ and residue field $\kappa$ that satisfies Assumption~\ref{osiiusuenfueiis} for $\ovl\rho$. Let $\mfk T=\mfk T(\ovl\rho_{A, m, \ell, v}\otimes\kappa)$ be the set of isomorphism classes of absolutely irreducible representations of $P_v$ appearing in $\Std\circ\ovl\rho_{A, m, \ell, v}\otimes\kappa$.

We first consider the case when $A'$ has good reduction. Let $\alpha, \beta\in\ovl{\bb Q_\ell}$ be the two eigenvalues of $\rho_{A', \ell}(\phi_q)$. 
Then $\alpha$ and $\beta$ are Weil $\norml{v}^{-1/2}$-numbers in $\ovl{\bb Q}$, which depends only on $A'$, not on $\ell$. We further assume that $\ell$ satisfies $\alpha, \beta\in (\ovl{\bb Z}_{(\ell)})^\times$, and that the image of the set
\begin{equation*}
\{(\alpha/\beta)^{2m}, (\alpha/\beta)^{2m-1},\ldots, (\alpha/\beta)^{-2m}\}
\end{equation*}
in $\ovl{\bb F_\ell}^\times$ does not contain $\norml{v}$. It follows that for any $\tau\in\mfk T$, every lifting of $\bx W_\tau(\rho_{A,m,\ell,v}\otimes\kappa)$ is actually unramified by Lemma~\ref{poeoiemifemifws}(2) if we assume that $\ell\nmid\prod_{i=1}^{2m+1}(\norml{v}^i-1)$. 

We then consider the case when $A'$ has multiplicative reduction. Let $u=\ord_v(j(A))<0$ be the valuation of the $j$-invariant of $A$. Assume further that $\ell$ is coprime to $u$. Then $\rho_{A, \ell}(t)$ is conjugate to $\pm(1+N_2)$ in $\GSp_2(\bb Z_\ell)$, which implies that $\Sym^{2m-1}\rho_{A, \ell}(t)$ is conjugate to $\pm\exp(N_{2m})$ in $\GSp_{2m}(\bb Z_\ell)$. It follows that $\mfk T$ is a singleton. Let $\tau$ be the unique element of $\mfk T$, then every lifting of $\bx W_\tau(\ovl \rho_{A,m,\ell,v}\otimes\kappa)$ is minimally unramified because $N_{2m}$ is of maximal rank.
\end{proof}

Combining this proposition with Proposition~\ref{associateGlaosireifnies}, we have the following immediate corollary. 

\begin{cor}
Let $\Pla$ be a finite set of finite places of $F$ containing $\Pla^\bad$ \sut $A$ has good reduction outside $\Pla$. Then all for but finitely many rational primes $\ell$, $\ovl\rho_{A, m, \ell}$ is rigid for $(\Pla, \vn)$ in the sense of \textup{Definition~\ref{isiinefies}} (with $\mcl O=\bb Z_\ell$).
\end{cor}

\subsection{Rigidity of automorphic Galois representations}\label{osisieiniiueurmeds}

In this subsection, we study rigidity of reductions of automorphic Galois representations. We work in the following setting.

\begin{setup}\enskip
\begin{itemize}
\item
Let $\Pi$ be a $\mfk d$-REASDC representation of $\GL_{2m}(\Ade_F)$ (see Definition~\ref{psoisieneifemiws}).
\item
Le $\Pla^\Pi$ be the smallest (finite) set of finite places of $F$ containing $\Pla^\bad$ \sut $\Pi_v$ is unramified for every finite place $v$ of $F$ not in $\Pla^\Pi$.
\item
Let $E\subset\bb C$ be a strong coefficient field of $\Pi$ (see Definition~\ref{strongienfeihenss}).
\end{itemize}
\end{setup}

For each finite place $\lbd$ of $E$, we obtain a continuous homomorphism $\rho_{\Pi, \lbd}:\Gal_F\to \GL_{2m}(E_\lbd)$.

\begin{conj}\label{psospivnihfs}
Fix a finite set $\Pla$ of finite places of $F$ containing $\Pla^\Pi$. Suppose that the base change of $\Pi$ to $F(\sqrt{\mfk d})$ is also cuspidal. Then for all but finitely many primes $\lbd$ of $E$ with underlying rational prime $\ell$, 
\begin{enumerate}
\item
$\ovl\rho_{\Pi, \lbd}|_{\Gal(\ovl F/F(\mu_\ell)(\sqrt{\mfk d}))}$ is absolutely irreducible; and 
\item
$\ovl\rho_{\Pi, \lbd}$ is rigid for $(\Pla, \vn)$ in the sense of \textup{Definition~\ref{isiinefies}}.
\end{enumerate}
\end{conj}

Concerning Conjecture~\ref{psospivnihfs}, we have the following theorem.

\begin{thm}\label{ososinbuufemws}
Suppose there exists a finite place $v_0$ of $F$ split in $F(\sqrt{\mfk d})$ \sut $\Pi_{v_0}$ is supercuspidal, then \textup{Conjecture~\ref{psospivnihfs}} holds.
\end{thm}
\begin{proof}
For part (1): We denote $F(\sqrt{\mfk d})$ by $\dot F$, and choose a place $w_0$ of $\dot F$ over $v_0$. In particular, we have a natural isomorphism $F_{v_0}\cong \dot F_{w_0}$. Since $\Pi_{v_0}$ is supercuspidal, the continuous representation $\rho_{\Pi_{v_0}}: W_{F_{v_0}}\to \GL_{2m}(\bb C)$ associated with $\Pi_{v_0}$ via the local Langlands correspondence is irreducible. By Clifford's theorem, there exists an irreducible representation $\tau$ of $I_{F_{v_0}}$ (with complex coefficients) and a character $\chi$ of $I_{F_{v_0}}\rtimes\bra{\varphi^b}$, \sut $\rho_{\Pi_{v_0}}$ is isomorphic to $\Ind_{I_{F_{v_0}}\rtimes\bra{\varphi_{v_0}^b}}^{W_{F_{v_0}}}(\tau\otimes\chi)$, where $b$ is the smallest positive integer satisfying $\tau^{\phi_{v_0}^b}\cong \tau$. We may choose a finite extension $E'$ of $E$ inside $\bb C$ and a finite set $\Lbd'$ of finite places of $E'$, \sut both $\tau$ and $\chi$ are defined over $\mcl O_{E', (\Lbd')}$. In particular, we may assume that the image of $\rho_{\Pi_{v_0}}$ is contained in $\GL_{2m}(\mcl O_{E', (\Lbd')})$.

Let $\Lbd_1'$ be the smallest set of finite places of $E'$ containing $\Lbd'$ \sut every finite place $\lbd'$ of $E'$ not in $\Lbd'$ satisfies
\begin{itemize}
\item
$\ovl\tau_{\lbd'}\defining \tau\otimes_{\mcl O_{E', (\Lbd')}}\ovl{\kappa_{\lbd'}}$ is irreducible; and
\item
$b$ is the smallest positive integer $k$ that satisfies $\ovl\tau_{\lbd'}^{\phi_{w_0}^k}$ is conjugate to $\ovl\tau_{\lbd'}$.
\end{itemize}
Let $\Lbd_1$ be the finite set of finite places of $E$ underlying $\Lbd_1'$, and let $\Lbd_2$ be the union of $\Lbd_1$ and the set of finite places of $E$ whose underlying rational prime $\ell$ satisfies either $\ell\le 2m(b_\xi-a_\xi)+1$ or $\Pla(\ell)\cap\Pla^\Pi\ne \vn$. Take a finite place $\lbd$ not in $\Lbd_2$ with underlying rational prime $\ell$. Then $\ovl\rho_{\Pi, \lbd}|_{\Gal(\ovl F/\dot F)}$ is absolutely irreducible by Proposition~\ref{associateGlaosireifnies}(2). By abuse of notation, we regard $\ovl\rho_{\Pi, \lbd}|_{\Gal(\ovl F/\dot F)}$ as a representation with $\ovl{\kappa_\lbd}$-coefficients. Since $[\dot F(\mu_\ell):\dot F]$ is coprime to $\ell$, the representation $\ovl\rho_{\Pi, \lbd}|_{\Gal(\ovl F/\dot F(\mu_\ell))}$ is semisimple. We claim that $\ovl\rho_{\Pi, \lbd}|_{\Gal(\ovl F/\dot F)}$ is an induction of an irreducible representation $\rho'$ of $\Gal(\ovl F/F')$ for some field extension $\dot F\subset F'\subset\dot F(\mu_\ell)$ \sut $[F':\dot F]$ equals the number of irreducible summands of $\ovl\rho_{\Pi, \lbd}|_{\Gal(\ovl F/\dot F(\mu_\ell))}$. By \cite[Lemma~4.3]{C-G13}, it suffices to show that the irreducible summands of $\ovl\rho_{\Pi, \lbd}|_{\Gal(\ovl F/\dot F(\mu_\ell))}$ are pairwise nonisomorphic. Since $v$ is unramified in $\dot F(\mu_\ell)$, it suffices to check that the irreducible summands of $\ovl\rho_{\Pi, \lbd}|_{I_{F_v}}$ are pairwise non-isomorphic. But this follows from the definition of $\Lbd_1'$.

By our definition of $\Lbd_2$ and Proposition~\ref{associateGlaosireifnies}(3), $\ovl\rho_{\Pi, \lbd}|_{\Gal(\ovl F/\dot F)}$ is crystalline with regular Fontaine--Laffaille weights in $[a_\xi, b_\xi]$ and $\ell>2m(b_\xi-a_\xi)+1\ge (b_\xi-a_\xi)+2$. Thus we must have $F'=\dot F$ by \cite[Lemma~4.7]{LTXZZa}, and then (1) follows.

For part (2): Let $\xi$ be the Archimedean weights of $\Pi$ (see Definition~\ref{psoisieneifemiws}). Let $\Lbd_2$ be the finite set of finite places of $E$ from part~(1). We show that each of the four conditions in Definition~\ref{isiinefies} excludes only finitely many finite places $\lbd$ of $E$. Condition~(2) is empty. Condition~(3) holds if the underlying rational prime $\ell$ of $\lbd$ satisfies $\ell\ge 2(b_\xi-a_\xi)+2$ and $\Pla(\ell)\cap\Pla=\vn$, by Proposition~\ref{associateGlaosireifnies}(3). Condition~(4) is automatic by Proposition~\ref{associateGlaosireifnies}(2). It remains to consider Condition~(1).

Let $\lbd$ be a finite place of $E$ whose underlying rational prime $\ell$ satisfies
\begin{equation*}
\Pla(\ell)\cap\Pla=\vn, \quad \ell>2(2m+1), \quad \ell>2(b_\xi-a_\xi)+2,
\end{equation*}
and $\ell>\norml{v}^{2m}$ for every place $v$ in $\Pla$. Then
\begin{enumerate}[(a)]
\item
$\ell$ is unramified in $F$;
\item
$\Pi_v$ is unramified for every $v\in\Pla(\ell)$;
\item
$\Std\circ\ovl\rho_{\Pi, \lbd}|_{\Gal(\ovl F/F(\mu_\ell)(\sqrt{\mfk d}))}$ is absolutely irreducible; which implies $\ovl\rho_{\Pi, \lbd}(\Gal(\ovl F/F(\mu_\ell)))$ is $G^\circ$-adequate by Proposition~\ref{soosieiueuenfue}; 
\item
Proposition~\ref{olaososiefefjeimss} holds for the local deformation problem $\mrs D^\FL$ of $\ovl\rho_{\Pi, \lbd, v}$ (see Definition~\ref{ooidnuefiems}) for every $v\in \Pla(\ell)$; and
\item
Proposition~\ref{psisiidimfimies} holds for the local deformation problem $\mrs D^\bmin$ of $\ovl\rho_{\Pi, \lbd, v}$ (see Definition~\ref{sspeoifiekfensinws}) for every $v\in\Pla$.
\end{enumerate}

We work in the setting of Setup~\ref{seutpeoeifmeis} with $E=E_\lbd$ plus setting of Setup~\ref{osisienieeimifes} with $(\Gamma, \ovl\rho, \chi)=(\Gal_F, \ovl\rho, \ve_\ell^{2-N})$, and fix an isomorphism $\iota_\ell: \bb C\xr\sim\ovl{\bb Q_\ell}$. For any collection $\mrs D_\Pla^\square\defining \{\mrs D_v^\square|v\in\Pla\}$ in which $\mrs D_v^\square$ is an irreducible component of $\Spf \msf R_{\ovl\rho_{\Pi, \lbd, v}}^\square$ for every $v\in \Pla$, we define a global deformation problem (see Definition~\ref{defomrioainoiereid})
\begin{equation*}
\mrs S(\mrs D_\Pla^\square)\defining\paren{\ovl\rho_{\Pi, \lbd}, \ve_\ell^{2-N}, \Pla\cup\Pla(\ell), \{\mrs D_v^\square\}_{v\in\Pla\cup\Pla(\ell)}},
\end{equation*}
where for each $v\in\Pla$, $\mrs D_v^\square$ is the prescribed irreducible component (which is a local deformation problem by Proposition~\ref{psisiidimfimies}(2)); and for each $v\in\Pla(\ell)$, $\mrs D_v^\square$ is the local deformation $\mrs D^{\FL, \square}$ of $\ovl\rho_{\Pi, \lbd}$ from Definition~\ref{ooidnuefiems}.

It follows from Theorem~\ref{osiseimfiehgiemifes} that $\msf R_{\mrs S(\mrs D_\Pla^\square)}$ is a finite $\mcl O$-module. On the other hand, it follows from Proposition~\ref{peoieutnmeumes} and the proof of \cite[Corollary~2.3.5]{CHT08} that \begin{align*}
\dim \msf R_{\mrs S(\mrs D_\Pla^\square)}\ge& 1+\sum_{v\in\Pla\cup\Pla(\ell)}(\dim_\kappa\bx L(\mrs D_v)-\dim_\kappa\bx H^0(\Gal_{F_v}, \ad^0(\ovl\rho)))-\dim_\kappa\bx H^0(G_{F, \Pla\cup\Pla(\ell)}, \ad^0(\ovl\rho)(1))\\
&-[F: \bb Q](\dim G_\kappa-\dim B_\kappa).
\end{align*} 
It follows from (d)--(e) that for each place $v\in \Pla$,
\begin{equation*}
\dim_\kappa\bx L(\mrs D_v)-\dim_\kappa\bx H^0(\Gal_{F_v}, \ad^0(\ovl\rho)(1))=
\begin{cases}
[F_v: \bb Q_\ell](\dim G_\kappa-\dim B_\kappa) & \If v|\ell\\
0 & \If v\nmid \ell
\end{cases}.
\end{equation*}
The assumption that $\ovl\rho_{\Pi, \lbd}(\Gal(\ovl F/F(\mu_\ell)))$ is adequate implies that $\bx H^0(G_{F, \Pla\cup\Pla(\ell)}, \ad^0(\ovl\rho)(1))=0$ by Definition~\ref{otieinnvnuducms}. Thus we have $\dim \msf R_{\mrs S(\mrs D_\Pla^\square)}\ge 1$. In particular, $\Spec \msf R_{\mrs S(\mrs D^\square_\Pla)}[\frac{1}{\ell}]$ is nonzero.

Choose an arbitrary $\ovl{\bb Q_\ell}$-point of $\Spec \msf R_{\mrs S(\mrs D^\square_\Pla)}[\frac{1}{\ell}]$. Via $\iota_\ell^{-1}$ and classical modularity theorems (see, for example, \cite[Theorem~4.2]{Gue11} strengthened using Thorne's adequacy condition \cite{Tho12}), we obtain a $\mfk d$-REASDC representation 
$\Pi(\mrs D_\Pla^\square)$ satisfying 
\begin{itemize}
\item
$\Pi(\mrs D_\Pla^\square)$ is unramified outside $\Pla\cup\Pla^\infty$;
\item
for each place $v$ of $F$ in $\Pla$, there is an open compact subgroup $U_v$ of $\GL_{2m}(F_v)$ depending only on $\Pi_v$, \sut $\Pi(\mrs D_\Pla^\square)_v$ has nontrivial vectors fixed by $U_v$; 
\item
the Archimedean weights of $\Pi(\mrs D_\Pla^\square)$ equals $\xi$;
\item
$\rho_{\Pi(\mrs D_\Pla^\square), \iota_\ell}$ and $\rho_{\Pi, \lbd}\otimes_{E_\lbd}\ovl{\bb Q_\ell}$ are residually isomorphic.
\end{itemize}
Note that the second property is a consequence of Proposition~\ref{associateGlaosireifnies}(2), Corollary~\ref{losoieiefniewswivmie} (which is applicable since $\ell>\norml{v}^{2m}$), and the fact that irreducible admissible representations in a common Bernstein component have a common level. Note that there are only finitely many isomorphism classes of $\mfk d$-REASDC representation of $\GL_{2m}(\Ade_F)$ satisfying the first three properties. By the strong multiplicity one property of $\GL_{2m}$ \cite{Sha79}, when $\ell$ is large, $\Pi$ is the only $\mfk d$-REASDC representation of $\GL_{2m}(\Ade_F)$ up to isomorphism satisfying all the four properties. 

By Proposition~\ref{psisiidimfimies}, the theorem would follow if we show that for any two different collections $\mrs D_\Pla^\square$ and $\mrs D_\Pla^{\square, \prime}$, the $\mfk d$-REASDC representations $\Pi(\mrs D_\Pla^\square)$ and $\Pi(\mrs D_\Pla^{\square, \prime})$ are not isomorphic. To show this, let $v$ be a place in $\Pla$, and let $x$ be the $\ovl{\bb Q_\ell}$-point of $\Spec \msf R^\square_{\ovl\rho_{\Pi, \lbd, v}}[\frac{1}{\ell}]$ corresponding to $\rho_{\Pi(\mrs D_\Pla^\square), \iota_\ell}|_{F_v}$. The dimension of the tangent space of $\Spec\msf R_{\ovl\rho_{\Pi, \lbd, v}}^\square[\frac{1}{\ell}]$ at $x$ is equal to
\begin{align*}
\bx Z^1(F_v, \ad^0(\rho)_{\Pi(\mrs D_\Pla^\square), \iota_\ell}|_{F_v})=&\dim G'_\kappa+\dim_{\ovl{\bb Q_\ell}}\bx H^1(F_v, \ad^0(\rho)_{\Pi(\mrs D_\Pla^\square), \iota_\ell}|_{F_v})-\bx H^0(F_v, \ad^0(\rho)_{\Pi(\mrs D_\Pla^\square), \iota_\ell}|_{F_v})\\ 
=&\dim G'_\kappa+\dim_{\ovl{\bb Q_\ell}}\bx H^2(F_v, \ad^0(\rho)_{\Pi(\mrs D_\Pla^\square), \iota_\ell}|_{F_v})\\
=&\dim G'_\kappa+\dim_{\ovl{\bb Q_\ell}}\bx H^0(F_v, \ad^0(\rho)_{\Pi(\mrs D_\Pla^\square), \iota_\ell}|_{F_v}(1)).
\end{align*}
Since $\Pi(\mrs D_\Pla^\square)_v$ is tempered by Proposition~\ref{associateGlaosireifnies}(1), we have $\dim_{\ovl{\bb Q_\ell}}\bx H^0(F_v, \ad^0(\rho)_{\Pi(\mrs D_\Pla^\square), \iota_\ell}|_{F_v}(1))=0$ by the proof of \cite[Lemma~1.3.2(1)]{LGGT}. It follows from Proposition~\ref{psisiidimfimies}(1) that $\Spec\msf R_{\ovl\rho_{\Pi, \lbd, v}}^\square[\frac{1}{\ell}]$ is smooth at $x$. Thus $x$ cannot lie on two irreducible components. Considering each place $v$ in $\Pla$, we conclude that $\Pi(\mrs D_\Pla^\square)$ and $\Pi(\mrs D_\Pla^{\square, \prime})$ are not isomorphic.

The theorem is proved.
\end{proof}

\bibliography{bibliography}

\begin{bibdiv}
\begin{biblist}

\bib{ABPSa}{article}{
      author={Aubert, Anne-Marie},
      author={Baum, Paul},
      author={Plymen, Roger},
      author={Solleveld, Maarten},
       title={The principal series of {$p$}-adic groups with disconnected
  center},
        date={2017},
        ISSN={0024-6115,1460-244X},
     journal={Proc. Lond. Math. Soc. (3)},
      volume={114},
      number={5},
       pages={798\ndash 854},
         url={https://doi.org/10.1112/plms.12023},
      review={\MR{3653247}},
}

\bib{A-G17}{article}{
      author={Atobe, Hiraku},
      author={Gan, Wee~Teck},
       title={On the local {L}anglands correspondence and {A}rthur conjecture
  for even orthogonal groups},
        date={2017},
        ISSN={1088-4165},
     journal={Represent. Theory},
      volume={21},
       pages={354\ndash 415},
         url={https://doi.org/10.1090/ert/504},
      review={\MR{3708200}},
}

\bib{AMS24}{article}{
      author={Aubert, Anne-Marie},
      author={Moussaoui, Ahmed},
      author={Solleveld, Maarten},
       title={Affine hecke algebras for {L}anglands parameters},
        date={2024},
      eprint={1701.03593},
         url={https://arxiv.org/abs/1701.03593},
}

\bib{Art13}{book}{
      author={Arthur, James},
       title={The endoscopic classification of representations},
      series={American Mathematical Society Colloquium Publications},
   publisher={American Mathematical Society, Providence, RI},
        date={2013},
      volume={61},
        ISBN={978-0-8218-4990-3},
         url={https://doi.org/10.1090/coll/061},
        note={Orthogonal and symplectic groups},
      review={\MR{3135650}},
}

\bib{Bal12}{book}{
      author={Balaji, Sundeep},
       title={G-valued potentially semi-stable deformation rings},
   publisher={ProQuest LLC, Ann Arbor, MI},
        date={2012},
        ISBN={978-1303-00490-2},
  url={http://gateway.proquest.com/openurl?url_ver=Z39.88-2004&rft_val_fmt=info:ofi/fmt:kev:mtx:dissertation&res_dat=xri:pqm&rft_dat=xri:pqdiss:3557387},
        note={Thesis (Ph.D.)--The University of Chicago},
      review={\MR{3152673}},
}

\bib{B-G19}{article}{
      author={Bellovin, Rebecca},
      author={Gee, Toby},
       title={{$G$}-valued local deformation rings and global lifts},
        date={2019},
        ISSN={1937-0652,1944-7833},
     journal={Algebra Number Theory},
      volume={13},
      number={2},
       pages={333\ndash 378},
         url={https://doi.org/10.2140/ant.2019.13.333},
      review={\MR{3927049}},
}

\bib{B-K90}{incollection}{
      author={Bloch, Spencer},
      author={Kato, Kazuya},
       title={{$L$}-functions and {T}amagawa numbers of motives},
        date={1990},
   booktitle={The {G}rothendieck {F}estschrift, {V}ol.\ {I}},
      series={Progr. Math.},
      volume={86},
   publisher={Birkh\"auser Boston, Boston, MA},
       pages={333\ndash 400},
      review={\MR{1086888}},
}

\bib{LGG11}{article}{
      author={Barnet-Lamb, Thomas},
      author={Gee, Toby},
      author={Geraghty, David},
       title={The {S}ato-{T}ate conjecture for {H}ilbert modular forms},
        date={2011},
        ISSN={0894-0347,1088-6834},
     journal={J. Amer. Math. Soc.},
      volume={24},
      number={2},
       pages={411\ndash 469},
         url={https://doi.org/10.1090/S0894-0347-2010-00689-3},
      review={\MR{2748398}},
}

\bib{LGGT}{article}{
      author={Barnet-Lamb, Thomas},
      author={Gee, Toby},
      author={Geraghty, David},
      author={Taylor, Richard},
       title={Potential automorphy and change of weight},
        date={2014},
        ISSN={0003-486X,1939-8980},
     journal={Ann. of Math. (2)},
      volume={179},
      number={2},
       pages={501\ndash 609},
         url={https://doi.org/10.4007/annals.2014.179.2.3},
      review={\MR{3152941}},
}

\bib{Boo19a}{article}{
      author={Booher, Jeremy},
       title={Minimally ramified deformations when {$\ell \neq p$}},
        date={2019},
        ISSN={0010-437X,1570-5846},
     journal={Compos. Math.},
      volume={155},
      number={1},
       pages={1\ndash 37},
         url={https://doi.org/10.1112/S0010437X18007546},
      review={\MR{3875451}},
}

\bib{Boo19}{article}{
      author={Booher, Jeremy},
       title={Producing geometric deformations of orthogonal and symplectic
  {G}alois representations},
        date={2019},
        ISSN={0022-314X,1096-1658},
     journal={J. Number Theory},
      volume={195},
       pages={115\ndash 158},
         url={https://doi.org/10.1016/j.jnt.2018.05.022},
      review={\MR{3867437}},
}

\bib{B-P19}{article}{
      author={Booher, Jeremy},
      author={Patrikis, Stefan},
       title={{$G$}-valued {G}alois deformation rings when {$\ell\neq p$}},
        date={2019},
        ISSN={1073-2780,1945-001X},
     journal={Math. Res. Lett.},
      volume={26},
      number={4},
       pages={973\ndash 990},
         url={https://doi.org/10.4310/MRL.2019.v26.n4.a2},
      review={\MR{4028108}},
}

\bib{Car12}{article}{
      author={Caraiani, Ana},
       title={Local-global compatibility and the action of monodromy on nearby
  cycles},
        date={2012},
        ISSN={0012-7094,1547-7398},
     journal={Duke Math. J.},
      volume={161},
      number={12},
       pages={2311\ndash 2413},
         url={https://doi.org/10.1215/00127094-1723706},
      review={\MR{2972460}},
}

\bib{Car14}{article}{
      author={Caraiani, Ana},
       title={Monodromy and local-global compatibility for {$l=p$}},
        date={2014},
        ISSN={1937-0652,1944-7833},
     journal={Algebra Number Theory},
      volume={8},
      number={7},
       pages={1597\ndash 1646},
         url={https://doi.org/10.2140/ant.2014.8.1597},
      review={\MR{3272276}},
}

\bib{Car94}{incollection}{
      author={Carayol, Henri},
       title={Formes modulaires et repr\'esentations galoisiennes \`a{} valeurs
  dans un anneau local complet},
        date={1994},
   booktitle={{$p$}-adic monodromy and the {B}irch and {S}winnerton-{D}yer
  conjecture ({B}oston, {MA}, 1991)},
      series={Contemp. Math.},
      volume={165},
   publisher={Amer. Math. Soc., Providence, RI},
       pages={213\ndash 237},
         url={https://doi.org/10.1090/conm/165/01601},
      review={\MR{1279611}},
}

\bib{Cas80}{article}{
      author={Casselman, W.},
       title={The unramified principal series of {${\mathfrak p}$}-adic groups.
  {I}. {T}he spherical function},
        date={1980},
        ISSN={0010-437X,1570-5846},
     journal={Compositio Math.},
      volume={40},
      number={3},
       pages={387\ndash 406},
         url={http://www.numdam.org/item?id=CM_1980__40_3_387_0},
      review={\MR{571057}},
}

\bib{C-G13}{article}{
      author={Calegari, Frank},
      author={Gee, Toby},
       title={Irreducibility of automorphic {G}alois representations of
  {$GL(n)$}, {$n$} at most 5},
        date={2013},
        ISSN={0373-0956,1777-5310},
     journal={Ann. Inst. Fourier (Grenoble)},
      volume={63},
      number={5},
       pages={1881\ndash 1912},
         url={https://doi.org/10.5802/aif.2817},
      review={\MR{3186511}},
}

\bib{C-H13}{article}{
      author={Chenevier, Ga\"etan},
      author={Harris, Michael},
       title={Construction of automorphic {G}alois representations, {II}},
        date={2013},
        ISSN={2168-0930,2168-0949},
     journal={Camb. J. Math.},
      volume={1},
      number={1},
       pages={53\ndash 73},
         url={https://doi.org/10.4310/CJM.2013.v1.n1.a2},
      review={\MR{3272052}},
}

\bib{CHT08}{article}{
      author={Clozel, Laurent},
      author={Harris, Michael},
      author={Taylor, Richard},
       title={Automorphy for some {$l$}-adic lifts of automorphic mod {$l$}
  {G}alois representations},
        date={2008},
        ISSN={0073-8301,1618-1913},
     journal={Publ. Math. Inst. Hautes \'Etudes Sci.},
      number={108},
       pages={1\ndash 181},
         url={https://doi.org/10.1007/s10240-008-0016-1},
        note={With Appendix A, summarizing unpublished work of Russ Mann, and
  Appendix B by Marie-France Vign\'eras},
      review={\MR{2470687}},
}

\bib{C-H16}{article}{
      author={Caraiani, Ana},
      author={Le~Hung, Bao~V.},
       title={On the image of complex conjugation in certain {G}alois
  representations},
        date={2016},
        ISSN={0010-437X,1570-5846},
     journal={Compos. Math.},
      volume={152},
      number={7},
       pages={1476\ndash 1488},
         url={https://doi.org/10.1112/S0010437X16007016},
      review={\MR{3530448}},
}

\bib{C-Z21}{article}{
      author={Chen, Rui},
      author={Zou, Jialiang},
       title={Local {L}anglands correspondence for even orthogonal groups via
  theta lifts},
        date={2021},
        ISSN={1022-1824,1420-9020},
     journal={Selecta Math. (N.S.)},
      volume={27},
      number={5},
       pages={Paper No. 88, 71},
         url={https://doi.org/10.1007/s00029-021-00704-8},
      review={\MR{4308934}},
}

\bib{C-Z24}{article}{
      author={Chen, Rui},
      author={Zou, Jialiang},
       title={Arthur's multiplicity formula for even orthogonal and unitary
  groups},
        date={2024},
     journal={Preprint},
      eprint={2103.07956},
         url={https://arxiv.org/abs/2103.07956},
}

\bib{F-L82}{article}{
      author={Fontaine, Jean-Marc},
      author={Laffaille, Guy},
       title={Construction de repr\'esentations {$p$}-adiques},
        date={1982},
        ISSN={0012-9593},
     journal={Ann. Sci. \'Ecole Norm. Sup. (4)},
      volume={15},
      number={4},
       pages={547\ndash 608},
         url={http://www.numdam.org/item?id=ASENS_1982_4_15_4_547_0},
      review={\MR{707328}},
}

\bib{Gee11}{article}{
      author={Gee, Toby},
       title={Automorphic lifts of prescribed types},
        date={2011},
        ISSN={0025-5831,1432-1807},
     journal={Math. Ann.},
      volume={350},
      number={1},
       pages={107\ndash 144},
         url={https://doi.org/10.1007/s00208-010-0545-z},
      review={\MR{2785764}},
}

\bib{Gue11}{article}{
      author={Guerberoff, Lucio},
       title={Modularity lifting theorems for {G}alois representations of
  unitary type},
        date={2011},
        ISSN={0010-437X,1570-5846},
     journal={Compos. Math.},
      volume={147},
      number={4},
       pages={1022\ndash 1058},
         url={https://doi.org/10.1112/S0010437X10005154},
      review={\MR{2822860}},
}

\bib{Hen00}{article}{
      author={Henniart, Guy},
       title={Une preuve simple des conjectures de {L}anglands pour {${\rm
  GL}(n)$} sur un corps {$p$}-adique},
        date={2000},
        ISSN={0020-9910,1432-1297},
     journal={Invent. Math.},
      volume={139},
      number={2},
       pages={439\ndash 455},
         url={https://doi.org/10.1007/s002220050012},
      review={\MR{1738446}},
}

\bib{HSBT}{article}{
      author={Harris, Michael},
      author={Shepherd-Barron, Nick},
      author={Taylor, Richard},
       title={A family of {C}alabi-{Y}au varieties and potential automorphy},
        date={2010},
        ISSN={0003-486X,1939-8980},
     journal={Ann. of Math. (2)},
      volume={171},
      number={2},
       pages={779\ndash 813},
         url={https://doi.org/10.4007/annals.2010.171.779},
      review={\MR{2630056}},
}

\bib{H-T01}{book}{
      author={Harris, Michael},
      author={Taylor, Richard},
       title={The geometry and cohomology of some simple {S}himura varieties},
      series={Annals of Mathematics Studies},
   publisher={Princeton University Press, Princeton, NJ},
        date={2001},
      volume={151},
        ISBN={0-691-09090-4},
        note={With an appendix by Vladimir G. Berkovich},
      review={\MR{1876802}},
}

\bib{Ish24}{article}{
      author={Ishimoto, Hiroshi},
       title={The endoscopic classification of representations of
  non-quasi-split odd special orthogonal groups},
        date={2024},
        ISSN={1073-7928,1687-0247},
     journal={Int. Math. Res. Not. IMRN},
      volume={2024},
      number={14},
       pages={10939\ndash 11012},
         url={https://doi.org/10.1093/imrn/rnae113},
      review={\MR{4776199}},
}

\bib{Jan04}{incollection}{
      author={Jantzen, Jens~Carsten},
       title={Nilpotent orbits in representation theory},
        date={2004},
   booktitle={Lie theory},
      series={Progr. Math.},
      volume={228},
   publisher={Birkh\"auser Boston, Boston, MA},
       pages={1\ndash 211},
      review={\MR{2042689}},
}

\bib{Kon03}{article}{
      author={Konno, Takuya},
       title={A note on the {L}anglands classification and irreducibility of
  induced representations of {$p$}-adic groups},
        date={2003},
        ISSN={1340-6116,1883-2032},
     journal={Kyushu J. Math.},
      volume={57},
      number={2},
       pages={383\ndash 409},
         url={https://doi.org/10.2206/kyushujm.57.383},
      review={\MR{2050093}},
}

\bib{KSZ21}{article}{
      author={Kisin, Mark},
      author={Shin, Sug~Woo},
      author={Zhu, Yihang},
       title={The stable trace formula for {S}himura varieties of abelian
  type},
        date={2021},
     journal={Preprint},
      eprint={2110.05381},
         url={https://arxiv.org/abs/2110.05381},
}

\bib{K-T17}{article}{
      author={Khare, Chandrashekhar~B.},
      author={Thorne, Jack~A.},
       title={Potential automorphy and the {L}eopoldt conjecture},
        date={2017},
        ISSN={0002-9327,1080-6377},
     journal={Amer. J. Math.},
      volume={139},
      number={5},
       pages={1205\ndash 1273},
         url={https://doi.org/10.1353/ajm.2017.0030},
      review={\MR{3702498}},
}

\bib{LTXZZa}{article}{
      author={Liu, Yi~Feng},
      author={Tian, Yi~Chao},
      author={Xiao, Liang},
      author={Zhang, Wei},
      author={Zhu, Xin~Wen},
       title={Deformation of rigid conjugate self-dual {G}alois
  representations},
        date={2024},
        ISSN={1439-8516,1439-7617},
     journal={Acta Math. Sin. (Engl. Ser.)},
      volume={40},
      number={7},
       pages={1599\ndash 1644},
         url={https://doi.org/10.1007/s10114-024-1409-x},
      review={\MR{4777059}},
}

\bib{Pat16}{article}{
      author={Patrikis, Stefan},
       title={Deformations of {G}alois representations and exceptional
  monodromy},
        date={2016},
        ISSN={0020-9910,1432-1297},
     journal={Invent. Math.},
      volume={205},
      number={2},
       pages={269\ndash 336},
         url={https://doi.org/10.1007/s00222-015-0635-3},
      review={\MR{3529115}},
}

\bib{Pen25a}{article}{
      author={Peng, Hao},
       title={Fargues-{S}cholze parameters and torsion vanishing for special
  orthogonal and unitary groups},
        date={2025},
     journal={Preprint},
      eprint={2503.04623},
         url={https://arxiv.org/abs/2503.04623},
}

\bib{Pen25c}{article}{
      author={Peng, Hao},
       title={On the {B}eilinson-{B}loch-{K}ato conjecture for polarized
  motives},
        date={2025},
      eprint={2509.18615},
         url={https://arxiv.org/abs/2509.18615},
}

\bib{Pen26}{article}{
      author={Peng, Hao},
       title={Higher {K}olyvagin theorems in the arithmetic {G}ross--{P}rasad
  setting},
        date={2026},
     journal={Ph.D. thesis to appear},
}

\bib{Sha79}{incollection}{
      author={Piatetski-Shapiro, I.~I.},
       title={Multiplicity one theorems},
        date={1979},
   booktitle={Automorphic forms, representations and {$L$}-functions ({P}roc.
  {S}ympos. {P}ure {M}ath., {O}regon {S}tate {U}niv., {C}orvallis, {O}re.,
  1977), {P}art 1},
      series={Proc. Sympos. Pure Math.},
      volume={XXXIII},
   publisher={Amer. Math. Soc., Providence, RI},
       pages={209\ndash 212},
      review={\MR{546599}},
}

\bib{Shi24}{article}{
      author={Shin, Sug~Woo},
       title={Weak transfer from classical groups to general linear groups},
        date={2024},
        ISSN={2834-4626,2834-4634},
     journal={Essent. Number Theory},
      volume={3},
      number={1},
       pages={19\ndash 62},
         url={https://doi.org/10.2140/ent.2024.3.19},
      review={\MR{4752712}},
}

\bib{Sho16}{article}{
      author={Shotton, Jack},
       title={Local deformation rings for {${\rm GL}_2$} and a
  {B}reuil-{M}\'ezard conjecture when {$\ell\ne p$}},
        date={2016},
        ISSN={1937-0652,1944-7833},
     journal={Algebra Number Theory},
      volume={10},
      number={7},
       pages={1437\ndash 1475},
         url={https://doi.org/10.2140/ant.2016.10.1437},
      review={\MR{3554238}},
}

\bib{Sho18}{article}{
      author={Shotton, Jack},
       title={The {B}reuil-{M}\'ezard conjecture when {$l\neq p$}},
        date={2018},
        ISSN={0012-7094,1547-7398},
     journal={Duke Math. J.},
      volume={167},
      number={4},
       pages={603\ndash 678},
         url={https://doi.org/10.1215/00127094-2017-0039},
      review={\MR{3769675}},
}

\bib{Sol25}{article}{
      author={Solleveld, Maarten},
       title={On principal series representations of quasi-split reductive
  {$p$}-adic groups},
        date={2025},
        ISSN={0030-8730,1945-5844},
     journal={Pacific J. Math.},
      volume={334},
      number={2},
       pages={269\ndash 327},
         url={https://doi.org/10.2140/pjm.2025.334.269},
      review={\MR{4862261}},
}

\bib{Sta}{article}{
      author={{Stacks Project Authors}, The},
       title={\textit{Stacks Project}},
         how={\url{https://stacks.math.columbia.edu}},
        date={2018},
}

\bib{Tay08}{article}{
      author={Taylor, Richard},
       title={Automorphy for some {$l$}-adic lifts of automorphic mod {$l$}
  {G}alois representations. {II}},
        date={2008},
        ISSN={0073-8301,1618-1913},
     journal={Publ. Math. Inst. Hautes \'Etudes Sci.},
      number={108},
       pages={183\ndash 239},
         url={https://doi.org/10.1007/s10240-008-0015-2},
      review={\MR{2470688}},
}

\bib{Tho12}{article}{
      author={Thorne, Jack},
       title={On the automorphy of {$l$}-adic {G}alois representations with
  small residual image},
        date={2012},
        ISSN={1474-7480,1475-3030},
     journal={J. Inst. Math. Jussieu},
      volume={11},
      number={4},
       pages={855\ndash 920},
         url={https://doi.org/10.1017/S1474748012000023},
        note={With an appendix by Robert Guralnick, Florian Herzig, Richard
  Taylor and Thorne},
      review={\MR{2979825}},
}

\bib{T-W95}{article}{
      author={Taylor, Richard},
      author={Wiles, Andrew},
       title={Ring-theoretic properties of certain {H}ecke algebras},
        date={1995},
        ISSN={0003-486X,1939-8980},
     journal={Ann. of Math. (2)},
      volume={141},
      number={3},
       pages={553\ndash 572},
         url={https://doi.org/10.2307/2118560},
      review={\MR{1333036}},
}

\bib{T-Y07}{article}{
      author={Taylor, Richard},
      author={Yoshida, Teruyoshi},
       title={Compatibility of local and global {L}anglands correspondences},
        date={2007},
        ISSN={0894-0347,1088-6834},
     journal={J. Amer. Math. Soc.},
      volume={20},
      number={2},
       pages={467\ndash 493},
         url={https://doi.org/10.1090/S0894-0347-06-00542-X},
      review={\MR{2276777}},
}

\bib{Wal03}{article}{
      author={Waldspurger, J.-L.},
       title={La formule de {P}lancherel pour les groupes {$p$}-adiques
  (d'apr\`es {H}arish-{C}handra)},
        date={2003},
        ISSN={1474-7480,1475-3030},
     journal={J. Inst. Math. Jussieu},
      volume={2},
      number={2},
       pages={235\ndash 333},
         url={https://doi.org/10.1017/S1474748003000082},
      review={\MR{1989693}},
}

\bib{Whi23a}{thesis}{
      author={Whitmore, Dmitri},
       title={The {T}aylor--{W}iles method for reductive groups},
        type={Ph.D. Thesis},
        date={2023},
         url={https://www.repository.cam.ac.uk/handle/1810/368883},
}

\bib{Wil95}{article}{
      author={Wiles, Andrew},
       title={Modular elliptic curves and {F}ermat's last theorem},
        date={1995},
        ISSN={0003-486X,1939-8980},
     journal={Ann. of Math. (2)},
      volume={141},
      number={3},
       pages={443\ndash 551},
         url={https://doi.org/10.2307/2118559},
      review={\MR{1333035}},
}

\bib{Y-Z25}{article}{
      author={Yang, Xiangqian},
      author={Zhu, Xinwen},
       title={On the generic part of the cohomology of shimura varieties of
  abelian type},
        date={2025},
      eprint={2505.04329},
         url={https://arxiv.org/abs/2505.04329},
}

\bib{Zha25}{article}{
      author={Zhang, Xiaoyu},
       title={Automorphic side of {T}aylor-{W}iles method for orthogonal and
  symplectic groups},
        date={2025},
      eprint={2411.04897},
         url={https://arxiv.org/abs/2411.04897},
}

\bib{Zhu24}{incollection}{
      author={Zhu, Xinwen},
       title={A note on integral {S}atake isomorphisms},
        date={2024},
   booktitle={Arithmetic geometry},
      series={Tata Inst. Fundam. Res. Stud. Math.},
      volume={41},
   publisher={Tata Inst. Fund. Res., Mumbai},
       pages={469\ndash 489},
      review={\MR{4812712}},
}

\end{biblist}
\end{bibdiv}

\vspace{2em}
\noindent\textsc{Hao Peng}\\
Department of Mathematics, Massachusetts Institute of Technology, Cambridge, MA 02139, USA\\
\textit{Email}: \texttt{hao\_peng@mit.edu}


\vspace{2em}
\noindent\textsc{Dmitri Whitmore}\\
DPMMS, Wilberforce Road, Cambridge CB3 0WA, UK\\
\textit{Email}: \texttt{dw517@cam.ac.uk}

\end{document}